\pdfoutput=1
\documentclass[11pt]{amsart}
\usepackage[letterpaper,margin=1in]{geometry}

\usepackage[T1]{fontenc}
\usepackage[utf8]{inputenc}
\usepackage{microtype}

\usepackage{amsmath,amssymb,mathtools}
\usepackage{amsthm}
\usepackage{bbm}
\usepackage{bm}

\usepackage{enumitem}
\usepackage{booktabs}
\usepackage{multirow}
\usepackage{array}
\usepackage{graphicx}
\usepackage{caption}
\usepackage{subcaption}
\usepackage{algorithm}
\usepackage{algpseudocode}
\usepackage{tikz}

\usepackage[authoryear,round]{natbib}
\usepackage{xcolor}
\newif\ifmor\morfalse %
\usepackage{hyperref}
\usepackage[nameinlink,capitalize]{cleveref}
\crefname{assumption}{assumption}{assumptions}
\Crefname{assumption}{Assumption}{Assumptions}
\colorlet{red}{black}

\hypersetup{
    colorlinks=true,
    linkcolor=blue!50!black,
    citecolor=blue!50!black,
    urlcolor=blue!50!black
}

\numberwithin{equation}{section}

\theoremstyle{plain}
\newtheorem{theorem}{Theorem}[section]
\newtheorem{lemma}{Lemma}[section]
\newtheorem{proposition}{Proposition}[section]

\theoremstyle{definition}
\newtheorem{definition}{Definition}[section]
\newtheorem{assumption}{Assumption}[section]
\newtheorem{example}{Example}[section]

\theoremstyle{remark}
\newtheorem{remark}[theorem]{Remark}

\makeatletter
\renewcommand{\paragraph}{%
  \@startsection{paragraph}{4}{\z@}%
  {1.25ex \@plus 0.2ex \@minus 0.2ex}%
  {-0.75em}%
  {\normalfont\normalsize\bfseries}%
}
\makeatother

\title[The BAR-SOT Method]
{The BAR-SOT Method: Long-term Average Cost Control as Stochastic Optimal Self-Transport}

\author[Srinivasan, Turkay, and Honnappa]{
Sharan Srinivasan$^{1}$, Berke M. Turkay$^{2}$, and Harsha Honnappa$^{3}$\\
\NoCaseChange{{\normalfont\small
$^{1}$Elmore School of Electrical and Computer Engineering,
Purdue University, West Lafayette, Indiana 47907\\
$^{2}$Daniels School of Business,
Purdue University, West Lafayette, Indiana 47907\\
$^{3}$Edwardson School of Industrial Engineering,
Purdue University, West Lafayette, Indiana 47907}}
}


\begin{document}

\begin{abstract}
    We present a novel reformulation of long-term average cost control of a Markov jump process as a finite-horizon stochastic optimal transport (SOT) problem that jointly optimizes over the controlled evolution and the marginal law from which it starts and returns to (i.e., a self-transport). The constraint relating the marginal flow to the controlled generator is the basic adjoint relationship (BAR), so we call the resulting problem BAR-SOT. For any time horizon $T>0$, the optimal value of this self-transport problem is a constant scaling of the long-run average-cost rate. We give three equivalent formulations (through controlled processes, a Fokker--Planck constraint, and relaxed marginal measures) and show that the optimal dual is a stationary potential plus a term linear in time, with slope equal to the rate. We show that the optimal dual must be solved for by solving a Hamilton-Jacobi-Bellman (HJB) equation with a periodic boundary condition. In general, a direct solve is hard to achieve, and therefore we propose to relax the problem. Specifically, a relative-entropy penalty on the control yields a cost-tilted Schr\"odinger bridge problem, computable by a Sinkhorn-type iteration, whose value converges to the unregularized optimum as the penalty vanishes. We illustrate the BAR-SOT method on a multiclass stochastic network scheduling problem. A neural parametrization of the dual, trained as a physics-informed neural network (PINN), that encodes Harrison's equivalent-workload formulation (see \citet{harrisonvanmieghem1997}) matches the strong reinforcement-learning baseline of \citet{daigluzman2022}; numerically conditioning the fit of a sub-dominant transverse correction then improves on both that baseline and the best priority heuristic. On the input-queued switch, a graph-attention parametrization of the dual improves on the strongest matching heuristic we are aware of, with a parameter count that does not grow with the switch size.
\end{abstract}

\maketitle

\section{Introduction}

\subsection{Background and Motivation}

We present a novel characterization of long-term average cost (or `average cost') control of Markov jump processes as a {\it stochastic optimal self-transport} problem. Average-cost problems are typically harder to solve than discounted-cost problems. Unlike the latter, the Bellman operator is not a contraction. The average-cost optimality equation (ACOE) ties the optimal cost rate, the relative value function and the policy together, with each only defined through the others. The existence of a solution strongly depends on the ergodicity of the controlled chain. This is further compounded by the fact that the long-run cost of a policy is implicit, as it is the running cost averaged against the stationary law induced by \textcolor{red}{that} policy. Of course, this law can only be derived from a balance equation once the policy is fixed. Consequently, a decision-maker must settle on a policy and an invariant distribution it produces at once. 

Classical methods, such as relative value iteration (RVI) and policy iteration (PI), attempt to solve this through a fixed point \citep{howard1960dynamic,puterman1994markov,arapostathis1993discrete}. For example, RVI re-centers the average cost around a reference state's value. Convergence in this case typically depends on the mixing rate of the controlled Markov chain. The alternative PI method requires solving the Poisson equation, along with a greedy improvement step in the iteration. The convergence rate is again conditioned by the mixing rate of the controlled Markov chain. Alternatively, vanishing discount methods, where the discount {factor} approaches 1, can only work under strong geometric ergodicity or Foster--Lyapunov conditions. 

An alternate perspective is provided by the convex-analytic/linear programming (LP) reformulations; see \citep{manne1960linear,borkar1988convex,hernandez1996discrete} (in discrete time), \citep{bhatt1996occupation,kurtz1998existence,kurtz2001stationary} (in continuous time). In these settings, the average-cost problem is demonstrated to be equivalent to a minimization over occupation measures on the state--action space, subject to stationarity or balance constraints. In general, the optimum is characterized cleanly only under similar ergodicity conditions. The convex-analytic perspective recasts ergodic control as a single linear program, and requires minimization of the average cost over a single stationary occupation measure on the state--action space, with the latter subject to a stationary basic adjoint relationship (BAR) constraint that the measure must satisfy. In continuous time, \citet{kurtz2001stationary} show that every feasible measure is the occupation measure of an admissible process. Here, we use this equivalence but do not solve the stationary linear program. Instead, we ask for the cheapest controlled Markov evolution that starts at a given marginal law over the state space and returns to it over a fixed, finite horizon of length $T > 0$. This is a {stochastic optimal `self-transport'} (SOT) problem, a variant of Mikami's stochastic optimal transport~\citep{mikami2021stochastic} in which the two endpoint marginals coincide. As it turns out, with a further optimization over this marginal law, {\it the problem has value exactly $T$ times the ergodic rate}. We call this the {\it meta-SOT} problem. For a fixed endpoint law the problem is the self-transport problem, and the meta-SOT problem is the further minimization over the endpoint law. Because the constraint tying the flow to the generator is the basic adjoint relationship, we refer to the formulation, at either level, as BAR-SOT. Thus, what was the ACOE together with its stationarity condition is reformulated as a stochastic optimal transport problem over controlled distribution flows. 

Equivalently, one may view the meta-SOT problem as an optimization over a flow of probability distributions that begins and ends at a common law, that is, a loop, and over the law itself, subject to the BAR condition. The SOT problem is handled through its dual, which is a problem in a single value function. The optimal dual function is time-periodic up to a shift that does not depend on the state, and this shift, which we call the endpoint gap, is precisely the meta-SOT value. Therefore, the meta-SOT problem is solved through a periodic boundary-value problem (BVP). Note that lifting the problem to a finite, periodic horizon, however, does not remove the self-consistency that makes average-cost problems difficult. 

The periodic BVP is the central object of the paper, in the sense that a solution $f$ solves the ergodic problem. The cost rate is read off as the endpoint gap, $g^*=(f_T(x)-f_0(x))/T$ for every state $x$, an optimal control is any maximizer of the Hamiltonian evaluated at $f$, and the optimal endpoint laws are the invariant laws of such maximizing controls. Under conditions stated in \Cref{sec:periodic-bvp}, we prove in \Cref{thm:periodic-wellposed} that the BVP is well posed, in the sense that a solution exists and every solution in a given class of functions coincides with it up to an additive constant. For time-homogeneous models the BVP is therefore equivalent to the ACOE, albeit differently posed.  The BVP is an equation, whereas the dual of the occupation-measure linear program is an inequality that is optimized. Well-posedness does not, by itself, provide a way to compute the solution, and a direct solve of the BVP is difficult (\Cref{sec:optimality-system}). In scheduling problems, for instance, the maximizer of the Hamiltonian selects a single action and switches between actions across a surface in the state space whose location is itself unknown. The Hamiltonian is therefore not differentiable in the dual function on that surface, and a gradient method that drives the residual of the BVP to zero receives a discontinuous search direction there.

Gradient-based computation of $f$ is possible, however, through entropic regularization, which adds a relative-entropy penalty on the control (measured against a fixed reference kernel) to the SOT objective. On the dual side, it replaces the hard maximization over controls in the Hamiltonian by a log-sum-exp, and the optimal control by a smooth Gibbs kernel that randomizes over the admissible actions rather than selecting one.  At a fixed endpoint law the regularized problem is an entropic optimal transport, or Schr\"odinger bridge, and when every transition rate is controlled it is computable by a Sinkhorn-type iteration; the periodicity condition then fixes the endpoint gap, i.e.\ the meta-SOT value. The Lagrange multiplier on the relative-entropy regularization is a `temperature' parameter and, as it is annealed to zero, the soft Hamiltonian sharpens back to the `hard' maximization. 

To the best of our knowledge, this reformulation of the average-cost control problem as a stochastic optimal self-transport problem is novel. We present the setting and our main results in Section~\ref{sec:barsot}. \ifmor\else{} Therefore, in the next section we will start with a simple discrete-time finite-state, finite-action average-cost Markov decision process (MDP) problem and show how it can be lifted to the BAR-SOT problem. In this case, the regularized problem is really a Schr\"odinger bridge problem.\fi{} In \Cref{sec:bar-sot}, we establish our main results. We demonstrate the equivalence of three seemingly different relaxations of the BAR-SOT problem in \Cref{thm:equivalence} and establish the equivalence between the ergodic cost rate and the meta-SOT value in \Cref{thm:ergodic-value}. In \Cref{sec:periodic-bvp}, we exploit the structure of the dual problem to recast the meta-SOT problem as a periodic BVP for the dual alone, with no regularization. As noted above, solving this BVP directly is difficult, and we explain why in \Cref{sec:optimality-system}.
This leads to the entropic regularization of \Cref{sec:entropy-reg}, which smooths the Hamiltonian so that the periodic BVP can be solved by gradient methods; there we also establish annealed consistency and a rate of convergence for the entropy-regularized problem in \Cref{thm:cont-action-consistency} and \Cref{thm:finite-action-consistency}. These main theoretical results are illustrated in \Cref{sec:numerical}{} on two scheduling problems. On a standard six-class reentrant queueing benchmark (\Cref{sec:heavy-traffic}), the learned dual matches a strong reinforcement-learning baseline and, after a refinement of the fit, improves on both that baseline and the best priority heuristic in heavy traffic. On the input-queued switch (\Cref{sec:switch}), a graph parametrization of the dual improves on the strongest matching heuristic of which we are aware~\citep{lu2025weighteddelay}, with a parameter count that does not grow with the switch size. We end with a discussion and future work in \Cref{sec:conclusion}.

\subsection{Literature Survey}

\paragraph{Convex-analytic and occupation-measure formulations.}
The fact that an average-cost control problem can be recast as a linear program over
occupation measures goes back to \citet{manne1960linear}. The state--action
occupation measure records the long-run fraction of time the controlled process
spends in each state--action pair. The cost is \emph{linear} in this measure, which must be
invariant under the controlled generator,
and the dynamics enter as a single linear constraint. \citet{borkar1988convex} developed
this convex-analytic viewpoint for controlled Markov processes, and it was 
 extended to a broad class of diffusions in \citet{bhatt1996occupation},  which also
established the attainment of optimal stationary controls.
\citet{hernandez1996discrete} gave the systematic LP treatment for
discrete-time chains. The invariance constraint in these formulations is
exactly the basic adjoint relationship we impose. {For deterministic control, \citet{gaitsgory2006linear, finlay2008duality} develop this linear program and its duality, with applications to periodic optimization. Their dual is a stationary Hamilton--Jacobi inequality in the state alone, whereas ours is a time-periodic boundary-value problem (\Cref{sec:periodic-bvp}) and the invariant law is left free.} This work builds on this prior art by adding
the transport structure that emerges only after lifting to a finite, periodic
horizon, which the steady-state LP does not automatically expose. 

\paragraph{Realization of measures as controlled processes.}
In continuous time, the link between the measure LP and genuine controlled
processes is due to \citet{kurtz1998existence,kurtz2001stationary}. They show
that the controlled martingale problem is equivalent to a linear program over
measures annihilating the generator, and that
\emph{every} measure satisfying this constraint is realized as the occupation
measure of an admissible stationary Markov control; their later work extends
this to forward equations and to singular controls. We use this realization in
both directions, passing freely between the process formulation of the SOT
problem and its marginal-measure formulation.

\paragraph{Solvers for the average-cost equation.}
Algorithmically, the average-cost optimality equation is classically treated by
policy iteration and relative value iteration \citep{howard1960dynamic,
puterman1994markov}, surveyed for continuous state spaces by
\citet{arapostathis1993discrete}, with average-reward reinforcement learning as
the data-driven counterpart \citep{mahadevan1996average, abounadi2001learning}.
These methods are constrained by the fact that the relative value function is fixed/estimated only up to
an additive constant, just as we see in
\Cref{sec:optimality-system}, and the gain is fixed by a solvability condition. Our
route to computing the ergodic rate differs in that we do not iterate on the steady-state equation itself but rather 
solve the lifted periodic problem. The periodic BVP makes a smoothed Hamiltonian
and a learned dual representation applicable.

\paragraph{Stochastic optimal transport and Schr\"odinger bridges.}
The transport reformulation rests on the theory of stochastic optimal transport, developed by
\citet{mikami2021stochastic}, whose framework gives the equivalence of the
process, Fokker--Planck, and marginal-measure formulations. Mikami's SOT generalizes problems including the Schr\"odinger bridge~\citep{leonard2014} and Nelson's stochastic mechanics~\citep{nelson1966derivation,nelson2012review} to a general class of `marginal' constraint problems. We exploit precisely this characterization in our reformulation of the average-cost problem, albeit in the context of controlled Markov jump processes which Mikami did not consider. The
entropic regularization of \Cref{sec:entropy-reg} is the control analog of
entropic optimal transport and the Schr\"odinger bridge problem
\citep{leonard2014, chen2021stochastic}, whose solution is computed by
Sinkhorn / iterative proportional fitting \citep{cuturi2013sinkhorn,
peyre2019computational}. Our equal endpoint condition
places the problem squarely in the self-transport (bridge) family, but with a
genuine control cost in place of a fixed reference dynamics, and with the
endpoint law itself an unknown that is optimized. Recent work extends the Schr\"odinger bridge to mean-field (McKean--Vlasov)
dynamics, steering an interacting population between two prescribed
marginals~\citep{backhoff2020mean, rapakoulias2025steering}. The entropic
problem is then solved by a Sinkhorn-type
recursion~\citep{eldesoukey2026generalized} of the same form as the endpoint scaling of \ifmor\else\Cref{sec:entropy-reg}\fi. Our problem
differs in \textcolor{red}{three} ways. It is a self-transport, and the
common law is itself optimized. The dynamics in our problem do not have any mean-field interactions.  

{\paragraph{Entropy-regularized control and Kullback--Leibler (KL) control.}
Several objects in \ifmor\else\Cref{sec:finite-mdp,sec:entropy-reg}\fi{} are shared with the literature on linearly solvable and entropy-regularized control. Tilting a reference kernel by the Gibbs factor of the cost, the Gibbs form of the optimal randomized control, and the log-sum-exp (soft) Hamiltonian appear in the linearly solvable Markov decision problems of \citet{todorov2009efficient} and the path-integral control of \citet{kappen2005path}, and entropy-regularized Markov decision processes are treated systematically by \citet{neu2017unified}, who derive the regularized Bellman optimality equations for the average-reward criterion on finite state and action spaces, and by \citet{geist2019theory} in the discounted setting. For an entropy regularizer with temperature $\varepsilon$, the bound of \citet{geist2019theory} on the gap between the regularized and unregularized optimal values specializes to $\varepsilon\log|\mathcal U|$ divided by one minus the discount factor, whereas the gap in our finite-action consistency theorem (\Cref{thm:finite-action-consistency}) is at most $\varepsilon\log|\mathcal U|$ per unit time. KL-weighted optimal control relative to a reference measure is characterized through information projection by \citet{selk2021information}, in the state-independent setting on Banach spaces. The present use differs in terms of what is constrained and what is optimized. The KL penalty acts on the control relative to a reference, subject to the self-transport endpoint constraint on the marginal flow, and the endpoint law is itself optimized. This meta-problem, its identification with the long-run average cost (\Cref{thm:ergodic-value}), and the periodic structure of the optimal dual have no counterpart in that literature.}

\paragraph{Risk-sensitive control.}
The same regularization connects our dual to risk-sensitive control. By the
Gibbs variational principle, a relative-entropy penalty on the control is dual
to an exponential-of-cost criterion, so the soft Hamiltonian is the log-sum-exp
of risk-sensitive ergodic control, with the temperature playing the role of the
risk-sensitivity parameter \citep{fleming1995risk, borkarMeyn2002risk,
arapostathisBorkarGhosh2012}. There, the additive eigenvalue problem of the
average-cost case becomes a multiplicative, principal-eigenvalue problem; the
same passage underlies the smoothing we use.

\paragraph{Density control, drift control, and scheduling.}
Finally, the paper sits alongside work that controls the law of a diffusion
through its Fokker--Planck equation \citep{annunziato2018fokker,
daudin2023optimal} and on singular and drift control of queues and their
diffusion limits \citep{ata2005drift}.
{The Fokker--Planck methods fix the initial law and optimize a terminal or running cost over a finite horizon, whereas here the flow is constrained to return to its initial law, and the law is itself a decision variable.}

\section{The BAR-SOT Method}~\label{sec:barsot}
We briefly introduce the methodology developed in this paper, which we term {\it BAR-SOT}. The idea is that we are given a controlled stochastic dynamics (discrete time in Section \ref{sec:finite-mdp}, and continuous time in Section \ref{sec:bar-sot}) and we want to transform the long term average cost, which is averaged over the dynamics, to an infimum over probability measures. This change not only turns the problem into a deterministic linear program, but also gives us a dual formulation. In this section, we introduce Stochastic Optimal Transport, and the BAR identity with its connections to our problem. 
\subsection{Optimal Transport}
Optimal Transport studies the optimal way of transporting mass from an initial distribution to a final distribution.    

Let $(\mathbf{X},\nu_0)$ and $(\mathbf{Y},\nu_1)$ be two measure spaces. Monge's problem asks for a map $T:\mathbf{X}\to \mathbf{Y}$, that $T$ pushes forward $\nu_0$ to $\nu_1$; $T_\#\nu_0 = \nu_1$, such that the cost
\begin{equation}\label{eq: Monge problem}
    \inf_{T}\int_\mathbf{X} c(x,T(x))\nu_0(dx),
\end{equation}
is minimized. Where $c:\mathbf{X}\times\mathbf{Y}\to\mathbb{R}_+$ is the cost function.
The push-forward means for any measurable set $B\subseteq\mathbf{Y}$, $\nu_1(B) = \nu_0(T^{-1}(B))$. Notice that the transport map $T$ is deterministic.

The problem with Monge's problem is that the set of transport maps $T:\mathbf{X}\to\mathbf{Y}$ such that $T_\#\nu_0 = \nu_1$ might be empty. Brenier showed that when $\mathbf{X } = \mathbf{Y} = \mathbb{R}^d$, $c(x,y) = |x-y|^2$ and $\nu_0$ is absolutely continuous with respect to the Lebesgue measure, there exists a unique transport map $T$ and it is the gradient of a convex function.

In order to address existence issues, Kantorovich relaxed Monge's problem by allowing probabilistic transport maps. Instead of a deterministic map $T$, we consider a joint probability measure $\pi \in \mathcal{P}(\mathbf{X\times Y})$ with marginals $\nu_0$ and $\nu_1$. Let $\Pi(\nu_0,\nu_1) := \{\pi\in \mathcal{P}(\mathbf{X\times Y})| (\pi)_{\#1} = \nu_0, (\pi)_{\#2} = \nu_1\}$. The Kantorovich problem is:
\begin{equation}\label{eq:Kantorovich problem}
    \inf_{\pi\in\Pi(\nu_0,\nu_1)}\int_{\mathbf{X\times Y}}c(x,y)\pi(dx,dy).
\end{equation}
The existence of a solution is guaranteed under milder conditions (e.g., when $\mathbf{X}$ and $\mathbf{Y}$ are Polish spaces and $c$ is lower semi-continuous). A key takeaway of Kantorovich's reformulation is that the problem becomes convex allowing for duality theory.

\subsection{Stochastic Optimal Transport}
The goal of Stochastic Optimal Transport (SOT) is to control a given stochastic dynamics which starts at a certain distribution to reach a target distribution. Let $I$ be an ordered compact index set (we will use $0$ as the minimum element and $1$ as the maximum element), $\mathcal{U}$ be a set of admissible controls, and $\mathcal{S}$ be a state space endowed with a topology. Let, $\{X_t^u\}_{t\in I}, \mu: I\times \mathcal{S}\to \mathcal{U}$ be a controlled stochastic process. Let $P_0 = \mathcal{L}(X_0^u)$, given a probability measure $P_1 \in\mathcal{P}(\mathcal{S})$, the SOT problem is to find an optimal control map $u$ such that the cost:
\begin{equation}\label{eq:General SOT}
    C(P_0,P_1) : =\inf_{\mu,\\\mathcal{L}(X_1^\mu)=P_1}\mathbb{E}\left[\int_I L(t,X_t^\mu;\mu(t,X))dt\right],
\end{equation}
is minimized, where $L:I\times \mathcal{S}\times\mathcal{U}\to\mathbb{R}_+$ is the running cost. 

When $X$ is a Markov process, it has a $\mu$-dependent generator $\mathcal{G}^\mu: \mathcal{D}\to\mathcal{D}$, where $\mathcal{D}\subseteq C_b(\mathcal{S})$. The generator induces an adjoint $(\mathcal{G}^\mu)^*$ acting on the space probability measures $\mathcal{P}(\mathcal{S})$. The generator induces the Basic Adjoint Relationship (BAR),
\begin{equation}
    \partial_t\rho_t = (\mathcal{G}^\mu)^*\rho_t,
\end{equation}
where $\rho_t = \mathcal{L}(X_t^\mu)$. This motivates us to reformulate the SOT problem in terms of flows of probability measures. The constraint in this formulation, as opposed to having a dynamics equation for the process $X_t^\mu$, is the weak BAR identity (a slight generalization of the BAR). The problem becomes minimizing the cost:
\begin{equation}
    \tilde{C}(P_0,P_1) := \inf_{(\{\rho_t\}_{t\in I}, \mu)\in \mathcal{A},\\\rho_0 = P_0,\rho_1=P_1}\int_I\int_\mathcal{S}L(t,x;\mu(t,x))\rho_t(dx)dt,
\end{equation}
where $\mathcal{A}$ is the set of admissible flows of measures and controls satisfying the weak BAR identity for all $t\in I$:
\[
\int_\mathcal{S}f(t,x)\rho_t(dx) - \int_\mathcal{S}f(0,x)\rho_0(dx) = \int_0^t\int_\mathcal{S}(\partial_sf(s,x) + \mathcal{G}^{u(s,x)}f(s,x))\rho_s(dx)ds.
\]

In both of the above formulations, the control used is deterministic. We now introduce a formulation in which the control is also chosen according to a probability law, and that law is optimized as well. Let $\pi\in\mathcal{P}(I\times\mathcal{S\times U})$ be a measure which can be disintegrated as 
\[
\pi(dt,dx,du) = dt\rho_t(dx)q_t(du|x),
\]
such that $\rho_t$ satisfies the weak BAR,
\[
\int_\mathcal{S}f(1,x)\rho_1(dx) - \int_\mathcal{S}f(0,x)\rho_0(dx) = \int_I\int_\mathcal{S}(\partial_sf(s,x) + \mathcal{G}^{\mu}f(s,x))\pi(dt,dx,du),
\]
and $\rho_0 = P_0$, $\rho_1 = P_1$. Then the SOT problem can be restated as,
\begin{equation}
    c(P_0,P_1) : =\inf_{\pi}\int_{I\times\mathcal{S\times U}}L(t,x;u)\pi(dt,dx,du).
\end{equation}

Mikami (in his book \cite{mikami2021stochastic}) considers controlled diffusions, where the control is the drift function. He states and proves (under certain assumptions) that the three formulations of SOT are equal, i.e., for all $P_0,P_1\in \mathcal{P}(\mathcal{S})$, we have $C(P_0,P_1) = \tilde{C}(P_0,P_1) = c(P_0,P_1)$. Taking inspiration from Mikami's formulation, in Section \ref{sec:bar-sot}, we introduce SOT for jump processes. 

\subsection{Average Cost Control as Self-Transport}
We consider the average cost control problem
\begin{equation}\label{eq:ergodic-cost}
    g^{\mathrm{erg}} := \inf_{\mu}\ \limsup_{T\to\infty}\frac{1}{T}\,\mathbb{E}\!\left[\int_0^T c\big(s,X^\mu(s),\mu(s,X^\mu(s))\big)\,ds\right].
\end{equation}
Under the optimal control $\mu^*$, we expect the controlled process $X_t^{\mu^*}$ to reach a stationary distribution, which we call the optimal distribution $\pi^*$. Assuming fast mixing time, the distribution $\rho_t$ of the controlled process will reach $\pi^*$ at a finite time, and remain there for all times after that. We, thus, relax the problem by looking at a finite horizon problem over $I$, and instead of fixing the distribution $\rho_t$ to stay at a fixed distribution $\pi$ (which we optimize over as well) at all times, only impose that it starts and ends at $\pi$. Thus, we impose the self-transport condition $\rho_0=\rho_1=\pi$ along with the weak BAR identity, and define the finite horizon problem (where we have dropped the running cost's $t$ dependence), 
\[
\tilde{C}(\pi) := \inf_{(\rho,u),\\\rho_0=\rho_1=\pi} \int_I\int_\mathcal{S}L(x;u(t,x))\rho_t(dx)dt.
\]
We expect a relationship between the ergodic problem and the self-transport problem. In the cases we consider, this relationship is exact (Proposition \ref{prop:discrete-value} and Theorem \ref{thm:ergodic-value}),
\[
\inf_{\pi\in\mathcal{P}(\mathcal{S})}\tilde C(\pi) = g^\mathrm{erg}.
\]

In Section~\ref{sec:finite-mdp}, we explore the case where the control set is discrete and the process $X_t^\mu$ is an MDP. Those computations are meant to illustrate the BAR-SOT method in a discrete setting. In the next section, though, we move to develop the BAR-SOT method in the more general continuous-time controlled Markov jump processes.

\section{Finite Markov Decision Processes as a Self-Transport Problem}\label{sec:finite-mdp}

Consider a discrete-time Markov decision process with finite state space \(\mathcal{S}\) and finite action sets \(\mathcal U(x)\). At each state \(x\in\mathcal{S}\), action \(u \in \mathcal U(x)\) incurs a one step cost \(c(x,u)\) and determines the distribution of the next state \(P(\cdot\mid x,u)\). For a finite set $A$, let $\mathcal P(A)$ denote the set of probability distributions on $A$.

A deterministic stationary policy is a map \(\mu\) satisfying $\mu(x)\in\mathcal U(x)$. The optimal long-run average cost is
\begin{equation}\label{eq:mdp-ergodic}
    g^{\mathrm{erg}} := \inf_{\mu}\ \limsup_{N\to\infty}\ \frac{1}{N}\sum_{n=0}^{N-1}\mathbb{E}\big[c(X_n,\mu(X_n))\big].
\end{equation}
 Here, the infimum is over deterministic stationary policies, \(X_0\sim\pi_0\), and $X_{n+1}$ has the conditional distribution \(P(\cdot\mid X_n,\mu(X_n))\). 

\subsection{The Ergodic Problem as a Linear Program}
At step $n$, a randomized decision rule
\[
    q_n(\cdot\mid x)\in\mathcal P(\mathcal U(x))
\]
assigns probability $q_n(\cdot\mid x) \in \mathcal P(\mathcal U(x))$ to action $u$ at state $x$, and induces the transition
matrix
\[
    P^{q_n}(x,y)
    =
    \sum_{u\in\mathcal U(x)}
    q_n(u\mid x)P(y\mid x,u).
\]
Let \(\pi_n\in\mathcal P(\mathcal{S})\) be the distribution of $X_n$. The joint probability of occupying state $x$ and choosing action $u$ at step $n$ is
\[
    m_n(x,u)
    :=
    \pi_n(x)q_n(u\mid x),
    \qquad x\in\mathcal{S},\ u\in\mathcal U(x).
\]
The state distribution evolves according to
\[
    \pi_{n+1}(y)
    =
    \sum_{x\in\mathcal{S}}\pi_n(x)P^{q_n}(x,y),
    \qquad y\in\mathcal{S}.
\]

{\begin{assumption}[Unichain]\label{assump:unichain}
Under every deterministic stationary policy, the controlled chain has a single recurrent class.
\end{assumption}
Under \Cref{assump:unichain},} an optimal deterministic stationary policy exists, and \(g^{\mathrm{erg}}\) is independent of the initial law \(\pi_0\) \citep[Chapter~8]{puterman1994markov}. The same value is obtained from the stationary occupation-measure linear program \citep{manne1960linear,hordijkkallenberg1979,borkar1988convex}. Here, \(\bar m(x,u)\) represents the long-run fraction of time spent in state \(x\) using action \(u\):
\begin{equation}\label{eq:occupation-LP}
  \begin{aligned}
    &\min_{\bar m\ge 0}\ \sum_{x\in\mathcal{S}}\sum_{u\in\mathcal U(x)} c(x,u)\,\bar m(x,u)\\
    &\text{subject to}\quad\\
    &\qquad\qquad\sum_{x,u}\bar m(x,u)=1,\quad
    \sum_{u\in\mathcal U(y)}\bar m(y,u)=\sum_{x,u}\bar m(x,u)\,P(y\mid x,u)\ \ \forall y\in\mathcal{S}.
      \end{aligned}
\end{equation}
The optimum of the linear program occurs at a vertex of the occupation polytope, namely a deterministic stationary policy. {Admitting randomized rules \(q_n(\cdot\mid x)\), equivalently optimizing \(\bar m\) directly, leaves the value~\eqref{eq:mdp-ergodic} unchanged, because the optimum sits at such a vertex.}

\subsection{\texorpdfstring{The $N$-Step Self-Transport Problem}{The N-Step Self-Transport Problem}}
Solving \eqref{eq:mdp-ergodic} through a self-transport reformulation leads to an \(N\)-step problem. Instead of optimizing the stationary occupation measure directly, we seek the cheapest controlled flow that, over \(N\) steps, leaves a marginal \(\tilde\pi\) and returns to it, and then optimize over \(\tilde\pi\). A controlled flow is a sequence $\{m_n\}_{n=0}^{N-1}$ of the state--action marginals of the preceding subsection, $m_n(x,u)=\pi_n(x)\,q_n(u\mid x)$, subject to the constraints displayed below. Randomization allows the flow to match the prescribed endpoint marginal, playing a different role here than in the ergodic problem. The controlled flow forms a {loop} in the space of marginal laws (see \Cref{fig:self-transport-loop}). Optimizing over both the endpoint law $\tilde\pi$ and the flow therefore recovers the ergodic rate~\eqref{eq:mdp-ergodic}, for every horizon $N\ge1$.
The \(N\)-step self-transport problem minimizes
\[
    \frac{1}{N}
    \sum_{n=0}^{N-1}
    \sum_{x\in\mathcal{S}}
    \sum_{u\in\mathcal U(x)}
    c(x,u)m_n(x,u)
\]
over state--action flows \(\{m_n\}_{n=0}^{N-1}\). The self-transport objective above is subject, for each \(n=0,\ldots,N-1\), to
\[
    m_n(x,u)\geq 0,
    \qquad x\in\mathcal{S},\ u\in\mathcal U(x),
\]
\[
    \sum_{u\in\mathcal U(x)}m_n(x,u)=\pi_n(x),
    \qquad x\in\mathcal{S},
\]
\[
    \pi_{n+1}(y)
    =
    \sum_{x\in\mathcal{S}}
    \sum_{u\in\mathcal U(x)}
    m_n(x,u)P(y\mid x,u),
    \qquad y\in\mathcal{S},
\]
and the self-transport endpoint condition
\[
    \pi_0=\pi_N=\tilde\pi.
\]

The $N$-step self-transport problem is the finite-state, discrete-time instance of the BAR-SOT formulation of the introduction, and the subsequent minimization over $\tilde\pi$ is the meta-SOT problem.

\begin{figure}[t]
\centering
\includegraphics[width=\linewidth]{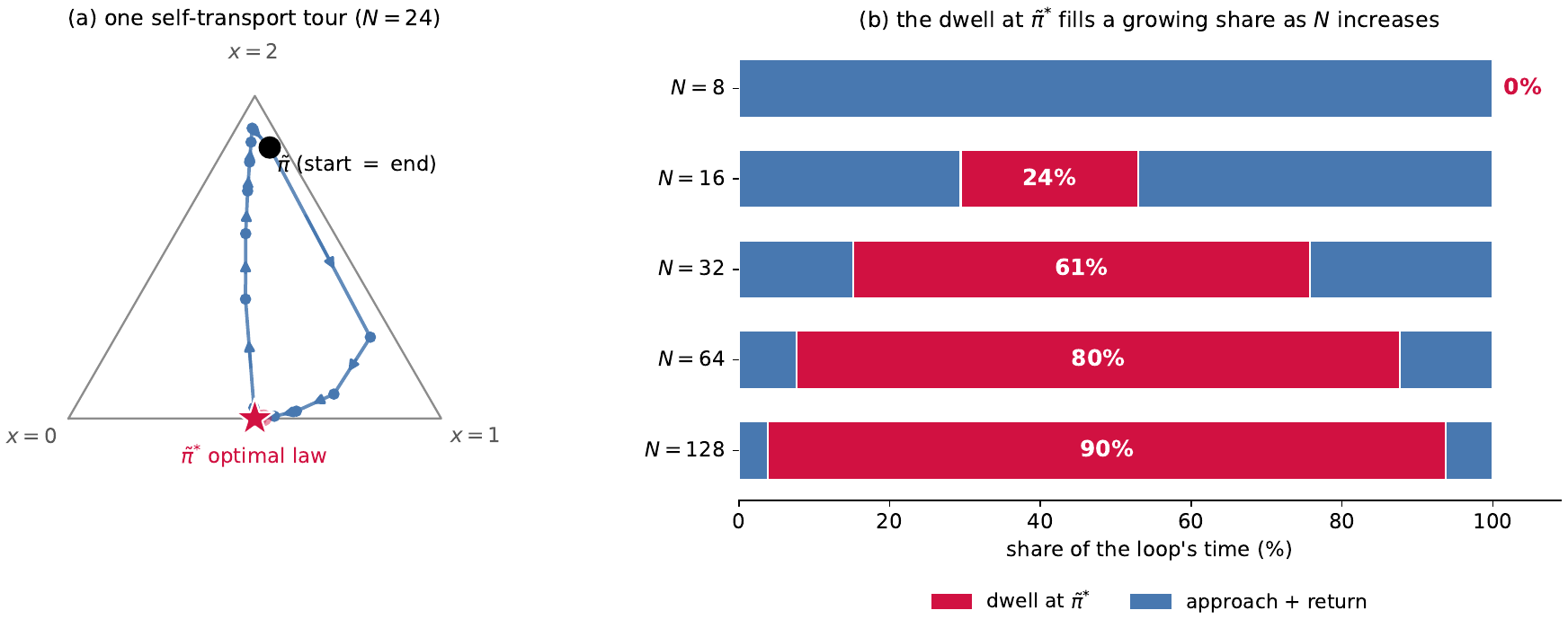}
\caption{{The self-transport {loop}} (running example: a three-state controlled queue; at each state the action is either \emph{serve} or \emph{idle}, with one-step cost $c(x,a)=x+\tfrac32\,\mathbf 1\{a=\mathrm{serve}\}$). The optimal stationary policy is bang-bang with optimal stationary law $\tilde\pi^{*}=(\tfrac12,\tfrac12,0)$. \textbf{(a)}~One optimal $N$-step flow $\{\pi_n\}_{n=0}^{N}$ of the self-transport problem. From the anchor $\tilde\pi$ (black), the marginal flows out to $\tilde\pi^{*}$ (red star), dwells there, and is steered back to close the loop $\pi_N=\pi_0=\tilde\pi$. \textbf{(b)}~The same tour's timeline split into the excursion (approach $+$ return, blue) and the dwell at $\tilde\pi^{*}$ (red), normalized to the horizon, for increasing $N$. The excursion has a fixed length, so the dwell share rises toward $100\%$ and the per-period average cost converges to the ergodic rate $g^{\mathrm{erg}}$ of~\eqref{eq:mdp-ergodic}.\label{fig:self-transport-loop}}

\end{figure}

\subsection{\texorpdfstring{The Discrete Basic Adjoint Relationship and the Ergodic Value}{The Discrete Basic Adjoint Relationship and the Ergodic Value}}\label{sec:dbar}

{We now identify the value of the $N$-step self-transport problem. We exploit the discrete form of the basic adjoint relationship (BAR), which is a family of linear identities obtained by testing the flow constraints against functions on the state space. }

Specifically, for every function \(f:\mathcal{S}\to\mathbb R\), the transition kernel $P^u$, under action
\(u\in\mathcal U(x)\) acts on \(f\) by
\[
    P^u f(x)
    :=
    \sum_{y\in\mathcal{S}}P(y\mid x,u)f(y).
\]
The one-step increment operator is
\[
    (P^u-I)f(x)
    =
    \sum_{y\in\mathcal{S}}P(y\mid x,u)f(y)-f(x)
    .
\]
The one-step controlled flow identity follows from the flow constraint and the
state--action marginal constraint:
\[
\begin{aligned}
    \langle f,\pi_{n+1}-\pi_n\rangle
    &=
    \sum_{x\in\mathcal{S}}
    \sum_{u\in\mathcal U(x)}
    m_n(x,u)(P^u-I)f(x).
\end{aligned}
\]
Summing the one-step identities over \(n=0,\ldots,N-1\) gives
\[
\begin{aligned}
    \sum_{n=0}^{N-1}
    \sum_{x\in\mathcal{S}}
    \sum_{u\in\mathcal U(x)}
    m_n(x,u)(P^u-I)f(x)
    &=
    \sum_{n=0}^{N-1}
    \langle f,\pi_{n+1}-\pi_n\rangle\\
    &=
    \langle f,\pi_N\rangle-\langle f,\pi_0\rangle
    =
    0,
\end{aligned}
\]
where the last equality uses the endpoint condition \(\pi_0=\pi_N=\tilde\pi\).

Then, the time-averaged state--action marginal,
\[
    \bar m_N(x,u)
    :=
    \frac{1}{N}
    \sum_{n=0}^{N-1}m_n(x,u),
\]
 satisfies the stationary discrete {\it basic adjoint relationship} (BAR) condition
\[
    \sum_{x\in\mathcal{S}}
    \sum_{u\in\mathcal U(x)}
    \bar m_N(x,u)(P^u-I)f(x)
    =
    0.
\]

The $N$-step self-transport objective can now be written as
\[
    \sum_{x\in\mathcal{S}}
    \sum_{u\in\mathcal U(x)}
    c(x,u)\bar m_N(x,u).
\]
Averaging in time therefore maps every feasible $N$-step flow to a feasible point of the occupation-measure linear program~\eqref{eq:occupation-LP}; the stationary discrete BAR is the balance constraint of~\eqref{eq:occupation-LP} tested against every $f$, and the two are equivalent on the finite state space (take $f=\mathbf 1_{\{y\}}$).

{
For $\tilde\pi\in\mathcal P(\mathcal S)$, let $g_N(\tilde\pi)$ denote the optimal value of the $N$-step self-transport problem with endpoint law $\tilde\pi$.

\begin{proposition}[The $N$-step value is the ergodic value]\label{prop:discrete-value}
Let \Cref{assump:unichain} hold. Then, for every $N\ge1$,
\[
\inf_{\tilde\pi\,\in\,\mathcal P(\mathcal S)}\ g_N(\tilde\pi)\ =\ g^{\mathrm{erg}} .
\]
\end{proposition}
The proof of \Cref{prop:discrete-value} is given in \Cref{ec:proofs}.
}

{
\begin{remark}[Comparison with the occupation-measure linear program]\label{rem:lp-comparison}
The formulations differ in their decision variables. The linear program has a single stationary measure, with the dynamics compressed into the balance constraint. The self-transport problem has $N$ state--action measures and the endpoint law $\tilde\pi$, imposes the dynamics step by step, and requires stationarity only of the time average, through the matched endpoints. The self-transport formulation yields a novel structural identification and a computational route through the entropic regularization of the next subsection. 
\end{remark}
}

{
\begin{remark}[Horizon independence]\label{rem:horizon-independence}
The equality in \Cref{prop:discrete-value} holds for every horizon, including $N=1$; it is an identity, not an asymptotic. In the classical problem~\eqref{eq:mdp-ergodic}, the horizon must grow so that the endpoint contribution $\langle f,\pi_N-\pi_0\rangle/N$ of the telescoped flow identity vanishes. The loop condition cancels this term exactly, at every $N$. At $N=1$ the meta-SOT problem coincides with the linear program. Observe that the constraints fix the state marginal of $m_0$ and its pushforward to the common value $\tilde\pi$, and optimizing over $\tilde\pi$ leaves exactly the balance constraint. 
\end{remark}
}

\subsection{Entropic Regularization and Schr\"odinger Bridge Form}
Now, we introduce an entropic regularization of the self-transport problem that facilitates computation. Fix a reference control policy \(r(u\mid x)\in\mathcal P(\mathcal U(x))\) and $\varepsilon > 0$. The regularized \(N\)-step objective becomes
\[
    \frac{1}{N}
    \sum_{n=0}^{N-1}
    \left[
    \sum_{x\in\mathcal{S}}
    \sum_{u\in\mathcal U(x)}
    c(x,u)m_n(x,u)
    +
    \varepsilon
    \sum_{x\in\mathcal{S}}
    \sum_{u\in\mathcal U(x)}
    m_n(x,u)
    \log
    \frac{m_n(x,u)}
    {\pi_n(x)r(u\mid x)}
    \right],
\]
subject to the same controlled flow constraints and endpoint condition $\pi_0=\pi_N=\tilde\pi$. Since \(m_n(x,u)=\pi_n(x)q_n(\cdot\mid x)\), the entropy term is
\[
    \sum_{x\in\mathcal{S}}\pi_n(x)
    D_{\mathrm{KL}}
    \bigl(q_n(\cdot\mid x)\Vert r(u\mid x)\bigr).
\]
To exhibit the Schr\"odinger-bridge structure, absorb the running cost into the reference. Let $R$ be the law of the controlled path started from \(\tilde\pi\) and evolving under the reference rule $r(\cdot\mid x)$, with one-step kernel $\sum_u r(u\mid x)P(\cdot\mid x,u)$, and define the \emph{cost-tilted} reference by the Gibbs reweighting
\begin{equation}\label{eq:cost-tilted-reference}
    \frac{dR^{c}}{dR}(x_0,u_0,\dots,x_{N})\ \propto\ \exp\!\Big(-\tfrac1\varepsilon\textstyle\sum_{n=0}^{N-1}c(x_n,u_n)\Big).
\end{equation}
Then the regularized objective equals $\tfrac{\varepsilon}{N}\,D_{\mathrm{KL}}(P^{q}\Vert R^{c})-\tfrac{\varepsilon}{N}\log Z_N$, where $P^{q}$ is the state--action path law of the controlled flow started at $\tilde\pi$ under the decision rules $q=(q_n)$, and $Z_N:=\mathbb E_R\big[e^{-\frac{1}{\varepsilon}\sum_{n}c(X_n,U_n)}\big]$ is the normalizing constant of~\eqref{eq:cost-tilted-reference}. Observe that the transition factors given the actions are common to $P^{q}$ and $R^{c}$ and cancel in the relative entropy, which is why the path-level quantity reduces to the accumulated control penalty. For fixed \(\tilde\pi\), dropping the constant and the horizon normalization, the regularized self-transport problem is exactly the Schr\"odinger bridge for the tilted reference,
\begin{equation}\label{eq:finite-SB}
    \min_{q:\,\pi_0=\pi_N=\tilde\pi}\ D_{\mathrm{KL}}\big(P^{q}\,\Vert\,R^{c}\big),
\end{equation}
the cheapest path law, in relative entropy to $R^{c}$, that leaves $\tilde\pi$ and returns to it. Posed on path space, the Schr\"odinger bridge problem ranges over all path laws with endpoint marginals $\tilde\pi$, whereas \eqref{eq:finite-SB} ranges only over the decision rules $q$. The restriction does not change the value, since the bridge of a Markov reference is Markov~\citep{leonard2014}. Locally, the cost tilts the reference control to $r^{\varepsilon}(u\mid x)\propto r(u\mid x)\,e^{-c(x,u)/\varepsilon}$, and \eqref{eq:finite-SB} is of the form addressed by Sinkhorn/iterative proportional fitting; the scope of that identification is given in \Cref{prop:ipfp-convergence} and \Cref{rem:ipfp-scope}. Recall that the classical Schr\"odinger problem matches endpoint marginals for a fixed reference dynamics with \emph{no} running cost, whereas here the reference is reweighted by $e^{-c/\varepsilon}$, the Gibbs factor of the cost. Hence, we identify this as a \emph{cost-tilted} bridge. It is the discrete-time, finite-state counterpart of the continuous-time entropic regularization of \Cref{sec:entropy-reg}.

A final note on controlled Markov jump processes on finite state/action spaces. Let $\mathcal G^u$ be the generator of such a process with bounded total jump rate. Uniformizing {the process} at any rate $\Lambda$ no smaller than the total jump rate gives the one-step kernel $P^u = I + \Lambda^{-1}\mathcal{G}^u$ and the corresponding one-step cost \(c/\Lambda\), which embeds it as a discrete-time MDP on the same state space; here $P^u-I = \Lambda^{-1}\mathcal{G}^u$, so the increment operator of the finite MDP is exactly the continuous-time generator up to the time scale $\Lambda$. The discrete BAR and self-transport endpoint condition carry over verbatim with $\mathcal{G}^u$ in place of $P^u-I$.

The models of interest in the sequel have a countably infinite state space and an unbounded running cost. The finite linear-programming arguments of this section do not apply to such models. Uniformization does not help, since the uniformized chain retains the countably infinite state space. In \Cref{sec:bar-sot}, therefore, we develop BAR-SOT in continuous time and without truncating the state space, with Lyapunov conditions in place of finiteness.

\section{Long-term Average Cost Control as a Stochastic Optimal Transport Problem}\label{sec:bar-sot}
We formulate the control problem for a general continuous-time controlled Markov jump processes. Let the state space $\mathcal{S}$ be a locally compact, separable metric space{; for the duality and verification results of this section and the next, we take $\mathcal S$ countable with the discrete topology, which is the setting of every example in this paper. The duality and verification results extend to general $\mathcal S$ under additional regularity, such as petiteness of compact sets and continuity of $V$. We leave this extension to general $\mathcal S$ to a future work.} Let the \emph{action set} $\mathcal{U}$ be a compact metric space (more generally, a state-dependent compact $\mathcal{U}(x)$, as in \Cref{sec:finite-mdp}, and we define $\mathcal{U} := \cup_{x\in\mathcal{S}}\mathcal{U}(x)$). In the remainder of this paper we will focus on \emph{Markov controls}, defined as a measurable function $\mu:[0,1]\times\mathcal{S}\to\mathcal{U}$, with $\mu(t,x)$ the action taken at time $t$ in state $x$; it is \emph{stationary} if it does not depend on $t$, and we then write $\mu(x)$. Note that, while stationary controls are the natural class for the ergodic problem~\eqref{eq:ergodic-cost}, the finite-horizon self-transport formulations below require the time-dependent class, since a flow that leaves $\tilde\pi$ and returns to it is generically driven by a time-varying control. The controlled dynamics are given by a family of generators $\{\mathcal{G}^u\}_{u\in\mathcal{U}}$ acting on a common core $\mathcal{D}\subseteq C_b(\mathcal{S})${, where $C_b(\mathcal{S})$ denotes the space of bounded continuous real-valued functions on $\mathcal{S}$}. A process--control pair $(Q,\mu)$ follows the dynamics if it solves the following martingale problem for every $f\in\mathcal{D}$, 
\begin{equation}\label{eq:mart-problem}
    f(Q(t))-f(Q(0))-\int_0^t \mathcal{G}^{\mu(s,Q(s))}f(Q(s))\,ds \quad\text{is a martingale.}
\end{equation}
This is the setting of the controlled martingale problem of \citet{kurtz1998existence,kurtz2001stationary}. In the finite setting of \Cref{sec:finite-mdp}, a measure annihilating the increment operator is the occupation measure of a stationary control by an elementary linear-programming argument (see the proof of \Cref{prop:discrete-value} and \Cref{rem:lp-comparison}). Beyond a finite state space this realization is no longer elementary. Producing a genuine process--control pair whose occupation measure is a prescribed annihilating measure is supplied by a forward-equation theorem of \citet{kurtz2001stationary} (quoted as \Cref{prop:KS-superposition}), which we invoke in the proof of \Cref{thm:equivalence}.

We take $\mathcal{G}^u$ to be a controlled jump generator,
\begin{equation}\label{eq:generator}
    \mathcal{G}^{u}f(x)=\int_{\mathcal{S}}\big[f(y)-f(x)\big]\,q(x,dy;u),
\end{equation}
where $q(x,\cdot\,;u)$ is a controlled jump-rate kernel. Our running example of a multiclass queue is the nearest-neighbor case with reflection at the boundary, but \eqref{eq:generator} equally covers controlled arrivals, routing and transfers between classes, and batch transitions.

The running cost is a measurable function $c:[0,1]\times\mathcal{S}\times\mathcal{U}\to[0,\infty)$, for which we assume the following.
\begin{assumption}\label{assump:cost}
    For each $(t,x)$, the map $u\mapsto c(t,x,u)$ is lower semicontinuous, and convex whenever $\mathcal{U}$ is convex.
\end{assumption}
As noted before, the object of study is the long-run average-cost control problem
\begin{equation}%
    g^{\mathrm{erg}} := \inf_{\mu}\ \limsup_{T\to\infty}\frac{1}{T}\,\mathbb{E}\!\left[\int_0^T c\big(s,Q(s),\mu(Q(s))\big)\,ds\right],
\end{equation}
the infimum taken over admissible stationary control maps $\mu$. Under an optimal $\mu$ the process reaches equilibrium as $T\to\infty$: its marginal law converges to the invariant law $\tilde\pi_\mu$ of $\mathcal{G}^\mu$ and is thereafter constant in time. Equivalently, the controlled dynamics leave $\tilde\pi_\mu$ unchanged, a flow started at $\tilde\pi_\mu$ returning to it over any finite time horizon.

For the long-run average cost to be well defined, we require each admissible stationary control to be stable, certified by a Foster--Lyapunov function of its own, with a cost that does not outgrow its stability margin. A single weight function $V$, common to all controls, is used only to measure growth. {For a measurable function $\phi$ and a measure $\nu$ on the same space, we write $\langle\phi,\nu\rangle:=\int\phi\,d\nu$ whenever the integral is defined, and for vectors $a,b\in\mathbb R^k$ with $k\ge1$, the same bracket $\langle a,b\rangle:=\sum_{i=1}^ka_ib_i$ denotes the Euclidean inner product.}
\begin{assumption}\label{assump:stability}
    There is a weight $V:\mathcal{S}\to[0,\infty)$ with relatively compact sublevel sets. A stationary map $\mu$ is \emph{admissible} if it generates an irreducible (Harris) process, if there exist $V_\mu:\mathcal{S}\to[0,\infty)$, $W_\mu:\mathcal{S}\to[1,\infty)$, a compact set $C_\mu\subset\mathcal{S}$ and a constant $b_\mu<\infty$ such that
    \begin{equation}\label{eq:foster-lyapunov}
        \mathcal{G}^{\mu}V_\mu(x)\le -W_\mu(x)+b_\mu\,\mathbf{1}_{C_\mu}(x),\qquad 0\le c\bigl(x,\mu(x)\bigr)\le W_\mu(x),\qquad x\in\mathcal{S},
    \end{equation}
    and if its invariant law $\tilde\pi_\mu$ satisfies $\langle V,\tilde\pi_\mu\rangle<\infty$.
\end{assumption}
The certificate $(V_\mu,W_\mu,C_\mu,b_\mu)$ may depend on $\mu$; nothing below uses constants that are uniform across controls. By the Foster--Lyapunov criterion~{\color{red}\citep[Thm.~4.2]{meyntweedie1993}}, each admissible $\mu$ is positive Harris recurrent with $\langle W_\mu,\tilde\pi_\mu\rangle<\infty$, so its long-run average cost exists and equals $\langle c(\cdot,\mu),\tilde\pi_\mu\rangle$ from every initial state.

\begin{example}[Multiclass reflected queue]\label{ex:queue}
 Take $\mathcal{S}=\mathbb{Z}^k_+$. Class-$i$ jobs arrive at rate $\lambda_i$ and are served at a controlled rate $\mu_i(x)$, giving the reflected queue-length process
\begin{equation}\label{eq:queue process}
    Q_i(t):=X_i(t)-\Big(\inf_{0\le s\le t}X_i(s)\Big)\wedge 0,\qquad
    X_i(t):=Q_i(0)+N_i^+(t\lambda_i)-N_i^-\!\Big(\int_0^t\mu_i(Q(s))\,ds\Big),
\end{equation}
$i\in[k]$, where the $N_i^\pm$ are independent unit-rate Poisson processes and $Q(0)\in\mathbb{Z}^k_+$ is the (possibly random) initial state. The generator \eqref{eq:generator} specializes to the nearest-neighbor kernel $q(x,\cdot\,;u)=\sum_i\lambda_i\,\delta_{x+e_i}+\sum_i u_i\,\delta_{x-e_i}\mathbf{1}_{x_i>0}$, i.e.
\begin{equation}\label{eq:queue-generator}
    \mathcal{G}^{\mu}f(x)=\sum_{i=1}^k\Big\{\lambda_i\big[f(x+e_i)-f(x)\big]+\mu_i(x)\big[f(x-e_i)-f(x)\big]\mathbf{1}_{x_i>0}\Big\}.
\end{equation}
With the forward/backward differences $f_+(x):=\big(f(x+e_i)-f(x)\big)_{i\in[k]}$ and $f_-(x):=\big(f(x-e_i)-f(x)\big)_{i\in[k]}${, with entries written $f_\pm^i$,} this reads $\mathcal{G}^{\mu}f(x)=\langle\lambda,f_+(x)\rangle+\sum_{i=1}^k\mu_i(x)\,[f_-(x)]_i\,\mathbf{1}_{x_i>0}$. \emph{Scheduling} is the finite action set $\mathcal{U}=\{\mu_i e_i\}_{i\in[k]}$, where serving class $i$ devotes rate $\mu_i$ to it and nothing to the others,  {for which, under the load condition $\rho:=\sum_{i\in[k]}\lambda_i/\mu_i<1$, the stationary maps that serve a class attaining $\max_i\mu_ix_i$ satisfy the drift of \Cref{assump:stability} with $V_\mu(x)=\frac{k}{(1-\rho)\mu_{\min}}\sum_ix_i^2$ and $W_\mu(x)=1+\sum_ix_i$. These are not the only admissible maps: any stationary map certified by its own Lyapunov function, such as a priority rule, is admissible as well. Indeed, for such a map, $\sum_j\lambda_jx_j\le\rho\max_i\mu_ix_i$ and $\max_i\mu_ix_i\ge k^{-1}\mu_{\min}\sum_ix_i$, and the drift follows outside a finite set.} Throughout, we take the action set to be the convex hull $\mathcal{U}=\mathrm{conv}\{\mu_ie_i\}_{i\in[k]}$, whose vertices are the pure schedules and whose remaining points are randomized schedules. %
\end{example}

\subsection{Process Formulation of Self-Transport Problem}
We now present a reformulation of the ergodic control problem as a self-transport problem. First, define an admissible class of stochastic processes as follows:
\begin{definition}
    We say that the process--control pair $(Q,\mu)$ is in the admissible class $\mathcal{A}$ if $\mu$ is a {Markov control}, the initial law $\mathcal{L}(Q(0))$ is prescribed, $(Q,\mu)$ solves the controlled martingale problem~\eqref{eq:mart-problem} on $[0,1]${, and has finite energy, $\mathbb{E}\int_0^1 V(Q(t))\,dt<\infty$, with $V$ the weight of \Cref{assump:stability}}.
\end{definition}

 Note that the Foster--Lyapunov drift of \Cref{assump:stability} plays no role for admissibility over a finite horizon. However, it defines the admissible stationary class of the ergodic problem~\eqref{eq:ergodic-cost}. For the queue of \Cref{ex:queue} this is the reflected process~\eqref{eq:queue process} with a control $\mu(t,\cdot)$ that is Lipschitz for each $t$, of finite variation.

\begin{remark}[{Queueing Revisited}]
    We can think of $\mu$ as a control on the rate of the decreasing process $N^-$, which is interpreted as the server processing a task. Indeed, when $\mu$ takes values in the discrete set $\mathcal{U} := \{\mu_ie_i\}_{i\in[k]}$, only one of the Poisson processes $N_i^-$ has a non-zero rate $\mu_i$ at any given time. This would then correspond to assigning a task to that server. In this way, $\mu$ plays the role of a control function controlling the intensity of the downward jump process.
\end{remark}

This motivates us to define the following self-transport problem for a fixed $\tilde\pi\in\mathcal P(\mathcal S)$ on a finite horizon: 
\begin{equation}
    \label{eq:process formulation cost}
    G(\tilde{\pi}) := \inf_{(Q,\mu)\in\mathcal{A}}\mathbb{E}\left[\int_0^1c(s,Q(s),\mu(s,Q(s)))\,ds\right],
\end{equation}
such that $\mathcal{L}(Q(0)) = \mathcal{L}(Q(1)) = \tilde{\pi}$.  This is a stochastic self-transport problem. To solve the scheduling problem, we still have to infimize over endpoint laws and take the ergodic limit. \Cref{thm:ergodic-value} shows that the finite-horizon problem has the same value as the ergodic problem when the cost does not have time dependence. 

\subsection{Weak Basic Adjoint Relationship (BAR) Formulation of Self-Transport Problem} 
We define another SOT problem related to the one stated above. However, this formulation does not use stochastic processes to optimize over. We consider a flow of probability measures $\{\rho_t\}_{0\leq t \leq 1}$ on the state space $\mathcal{S}$ that forms a loop, i.e., $\rho_0 = \rho_1 = \tilde{\pi}$, and then look at an admissible class of control functions for a given flow and optimize over both admissible controls and the flows of probability measures. Throughout, $UC_b^1$ denotes the class of functions $f:[0,1]\times\mathcal{S}\to\mathbb{R}$ such that $f$ and the time derivative $\partial_tf$ are bounded and uniformly continuous on $[0,1]\times\mathcal{S}$; on a countable state space this reduces to boundedness and $C^1$ time dependence, uniformly over states. For a horizon $[0,T]$ the same class is understood with $[0,1]$ replaced by $[0,T]$.
\begin{definition}
    Given a flow $\{\rho_t\}_{0\leq t \leq 1}\subset\mathcal{P}(\mathcal{S})$ of probability measures which form a loop, i.e., $\rho_0= \rho_1=\tilde{\pi}$, and a Markov control $\mu:[0,1]\times\mathcal{S}\to\mathcal{U}$, we say that $\mu$ is admissible for the flow $\{\rho_t\}_{0\leq t\leq 1}$ if the flow has finite energy, $\int_0^1\langle 1+V,\rho_t\rangle\,dt<\infty$, and, for all $t\in[0,1]$ and all $f\in UC_b^1$, the weak BAR identity holds:
    \begin{equation}
        \label{eq:weak FP}
        \int_{\mathcal{S}}f(t,x)\,\rho_t(dx)-\int_{\mathcal{S}}f(0,x)\,\tilde{\pi}(dx) = \int_0^t\!\int_{\mathcal{S}}(\partial_sf(s,x)+\mathcal{G}^{\mu(s,x)} f(s,x))\,\rho_s(dx)\,ds.
    \end{equation}
    $\mathcal{G}^\mu$ is given in equation \eqref{eq:generator}. 
\end{definition}

\begin{remark}
    Identity \eqref{eq:weak FP} is the weak (Kolmogorov) forward equation for the flow, called the Fokker--Planck constraint in the stochastic optimal transport literature; we refer to it as the weak BAR identity and to \eqref{eq:SOT weak FP} as the weak BAR formulation; at $t=1$, since $\rho_1=\tilde\pi$, it reduces to the loop (BAR) constraint $\int_{\mathcal{S}}(f(1,x)-f(0,x))\,\tilde{\pi}(dx)=\int_0^1\!\int_{\mathcal{S}}(\partial_sf+\mathcal{G}^\mu f)\,\rho_s(dx)\,ds$.
\end{remark}

We are now ready to define the SOT problem,
\begin{equation}
    \label{eq:SOT weak FP}
    \tilde{G}(\tilde{\pi}) := \inf_{\{\rho_t\},\mu}\int_0^1\!\int_{\mathcal{S}}c(t,x,\mu(t,x))\,\rho_t(dx)\,dt,
\end{equation}
where the infimum is taken over all flows of measures which form a loop with $\rho_0=\rho_1=\tilde{\pi}$, and admissible maps $\mu$ for each flow.

\subsection{Relaxed Marginal Measure Formulation of Self-Transport Problem}
We define yet another SOT problem using marginal measures. This formulation enjoys the property that the controls are randomized in addition to the state. %
\begin{definition}\label{def:relaxed-admissible}
    We say that a probability measure $\pi\in\mathcal{P}([0,1]\times\mathcal{S}\times\mathcal{U})$ is in an admissible set $\mathcal{M}$ if the following conditions hold:
    \begin{enumerate}
        \item The measure disintegrates in time against Lebesgue measure on $[0,1]$, i.e., $\pi(dt,dx,du) = \pi_t(dx,du)\,dt$.
        \item The state marginal $\pi_{1,t}(dx):= \int_{\mathcal{U}}\pi_t(dx,du)$ admits a weakly continuous version $t\mapsto\pi_{1,t}$, with $\pi_{1,0} = \pi_{1,1} = \tilde{\pi}\in\mathcal{P}(\mathcal{S})$. We write the disintegration $\pi_t(dx,du)=q_t(du\mid x)\,\pi_{1,t}(dx)$, where $q_t(\cdot\mid x)\in\mathcal{P}(\mathcal{U})$ is the \emph{randomized control kernel}.
        \item The measure satisfies the BAR relationship: For all $f\in UC^1_b$,
        \begin{equation}
            \label{eq:BAR constraint}
            \int_{\mathcal{S}}(f(1,x)-f(0,x))\,\tilde{\pi}(dx) = \int_{[0,1]\times\mathcal{S}\times\mathcal{U}}\{\partial_tf(t,x) + \mathcal{G}^uf(t,x)\}\,\pi(dt,dx,du).
        \end{equation}
        \item The measure has finite energy, $\int_{[0,1]\times\mathcal{S}\times\mathcal{U}}(1+V(x))\,\pi(dt,dx,du)<\infty$, with $V$ the weight of \Cref{assump:stability}. This normalization makes the dual pairing below well defined; the stationary competitors used in \Cref{thm:ansatz} satisfy it.
    \end{enumerate}
\end{definition}
Here~\eqref{eq:BAR constraint} is the continuous-time counterpart of the discrete-time BAR constraint of \Cref{sec:dbar}, with $\mathcal{G}^u$ in the role of $P^u-I$ and the relaxed marginal measure $\pi$ in the role of the time-averaged occupation $\bar m_N$.

The self-transport problem is:
\begin{equation}
    \label{eq:relaxed marginal measures SOT}
    g(\tilde{\pi}) := \inf_{\pi\in\mathcal{M}}\int_{[0,1]\times\mathcal{S}\times\mathcal{U}}c(t,x,u)\,\pi(dt,dx,du).
\end{equation}

We use this formulation to construct the dual problem in the sense of \citet[Theorem 2.4]{mikami2021stochastic}. To this end, consider the Lagrangian
\begin{multline}\label{eq:lagrangian}
   L(\pi,f) := \int_{[0,1]\times\mathcal{S}\times\mathcal{U}} \bigg[c(t,x,u)-\big\{\partial_tf(t,x)+\mathcal{G}^uf(t,x)\big\}\bigg]\pi(dt,dx,du) \\+ \int_{\mathcal{S}}(f(1,x)-f(0,x))\,\tilde\pi(dx).
\end{multline}
From weak duality, we have
\begin{equation}
    \label{eq:weak duality}
    g(\tilde{\pi}) = \inf_{\pi}\sup_fL(\pi,f) \geq \sup_f\inf_\pi L(\pi,f).
\end{equation}
Since $\pi$ is a probability measure, the first term in \eqref{eq:lagrangian} must be nonnegative; otherwise, the infimum over $\pi$ can drive the value to be unbounded. Thus, for all $(t,x,u)\in [0,1]\times\mathcal{S}\times\mathcal{U}$,
\begin{equation*}
    -c(t,x,u)+\partial_tf(t,x)+\mathcal{G}^uf(t,x) \leq 0.
\end{equation*}
Since this holds for all $u\in\mathcal{U}$, it holds for the supremum as well:
\begin{equation}\label{eq:hjb-inequality}
    \partial_tf(t,x) + \mathcal{H}f(t,x)\leq 0,
\end{equation}
where the Hamiltonian is
\begin{equation}
    \label{eq:Hamiltonian}
    \mathcal{H}f(t,x):= \sup_{u\in\mathcal{U}}\{\mathcal{G}^uf(t,x)-c(t,x,u)\}.
\end{equation}
For the queue of \Cref{ex:queue}, $\mathcal{G}^uf=\langle\lambda,f_+\rangle+\sum_{i}u_if_-^i\,\mathbf 1_{x_i>0}$, so $\mathcal{H}f=\langle\lambda,f_+\rangle+H(t,x;f_-)$ with the control Hamiltonian $H(t,x;z):=\sup_{u\in\mathcal{U}}\{\langle z,u\rangle-c(t,x,u)\}$.

\begin{example}[Queue scheduling and singular structure]\label{ex:singular}
Consider the $k=2$ case of \Cref{ex:queue}, with the two pure schedules $(\mu_1,0)$ and $(0,\mu_2)$, corresponding to the server devoting its effort to class $1$ at rate $\mu_1$ or to class $2$ at rate $\mu_2$; assume the cost has no control dependence, and let the load condition of \Cref{ex:queue} hold. The service part of the generator is linear in $u$, so the control Hamiltonian reduces to
\begin{equation}\label{eq:singular-two-action-H}
    H(t,x;f_-)=\max\big\{\mu_1 f_-^1(t,x)\,\mathbf 1_{x_1>0},\ \mu_2 f_-^2(t,x)\,\mathbf 1_{x_2>0}\big\}-c(t,x).
\end{equation}
The comparison between the two service directions is the switching function
\begin{equation}\label{eq:singular-switching-function}
    \Delta_f(t,x):=\mu_1 f_-^1(t,x)-\mu_2 f_-^2(t,x),
\end{equation}
and the maximizing control is bang-bang, $u^*(t,x)=(\mu_1,0)$ when $\Delta_f(t,x)>0$ and $u^*(t,x)=(0,\mu_2)$ when $\Delta_f(t,x)<0$. On the switching set
\begin{equation}\label{eq:singular-set}
    \Sigma:=\{(t,x):\ x_1\wedge x_2>0,\ \Delta_f(t,x)=0\},
\end{equation}
the Hamiltonian is flat between the two actions and does not determine a unique decision, and any deterministic tie-break is external to it. Let $\pi\in\mathcal M$, with randomized control kernel $q_t(du\mid x)$, and let $f$ satisfy~\eqref{eq:hjb-inequality}, with $\int c\,d\pi=\int_{\mathcal S}(f(1,x)-f(0,x))\,\tilde\pi(dx)$, so that the pair has no duality gap in~\eqref{eq:weak duality}. The integrand of the Lagrangian~\eqref{eq:lagrangian} is nonnegative by~\eqref{eq:hjb-inequality} and, by the BAR identity~\eqref{eq:BAR constraint}, integrates to zero against $\pi$, so $q_t(\cdot\mid x)$ is carried by the maximizing actions of $u\mapsto\mathcal G^uf(t,x)-c(t,x,u)$ for $\pi$-almost every $(t,x)$. Away from $\Sigma$, the maximizing action is unique, and $q_t(\cdot\mid x)$ is the point mass at $u^*(t,x)$. When $\pi$ places no mass on $\Sigma$, the deterministic rule $u^*$ is therefore the optimal control. Mass of $\pi$ on $\Sigma$ is precisely the case in which the relaxed formulation is needed. On $\Sigma$, $q_t(\cdot\mid x)$ may split mass, and its barycenter $\bar u(t,x):=\int_{\mathcal U}u\,q_t(du\mid x)$ interpolates the two service directions.
\end{example}

\subsection{Equivalence of the Three Formulations}
We now show that the values of the three SOT problems are the same, which rests on the following structural hypothesis: the control enters the dynamics affinely, so that randomizing the action and then averaging it leaves the generator unchanged.
\begin{assumption}\label{assump:affine}
The action set $\mathcal{U}$ is convex, and for every $f\in\mathcal{D}$ and $x\in\mathcal{S}$ the map $u\mapsto\mathcal{G}^uf(x)$ is affine. Consequently, for any control kernel $q(du\mid x)$ on $\mathcal{U}$ with barycenter $\bar u(x):=\int_{\mathcal{U}}u\,q(du\mid x)\in\mathcal{U}$,
\begin{equation}\label{eq:affine-average}
    \int_{\mathcal{U}}\mathcal{G}^uf(x)\,q(du\mid x)=\mathcal{G}^{\bar u(x)}f(x).
\end{equation}
\end{assumption}
Rate controls satisfy this condition, and so do scheduling controls under the convex-hull convention of \Cref{ex:queue}. For instance, in \Cref{ex:queue}, $\mathcal{G}^uf=\langle\lambda,f_+\rangle+\sum_{i}u_if_-^i\,\mathbf 1_{x_i>0}$ is linear in $u$, so a randomized service rule has the generator of a point of $\mathcal{U}=\mathrm{conv}\{\mu_ie_i\}_{i\in[k]}$, namely its barycenter. When the cost does not depend on the control, the maximand of the Hamiltonian~\eqref{eq:Hamiltonian} is linear in $u$, so a maximizer can be taken at a vertex, which is the bang-bang policy analyzed in \Cref{ex:singular}.

The realization step in the proof below rests on a forward-equation theorem of \citet{kurtz2001stationary}. We restate it in the form used here, specialized to the case of no singular component and with the time dependence carried by the flow. Recall the core $\mathcal{D}\subseteq C_b(\mathcal{S})$ of the martingale problem~\eqref{eq:mart-problem}.
\begin{proposition}[{\citealp[Corollary~1.12]{kurtz2001stationary}}]\label{prop:KS-superposition}
Let $\{\nu_t\}_{t\in[0,1]}\subset\mathcal{P}(\mathcal{S})$ be a measurable flow, let $q_t(du\mid x)$ be a jointly measurable control kernel, and set $\hat{\mathcal{A}}_tf(x):=\int_{\mathcal{U}}\mathcal{G}^uf(x)\,q_t(du\mid x)$ for $f\in\mathcal{D}$. Assume the following.
\begin{enumerate}
\item[(i)] There exist $\psi:\mathcal{S}\times\mathcal{U}\to[1,\infty)$ and constants $a_f<\infty$ such that $|\mathcal{G}^uf(x)|\le a_f\,\psi(x,u)$ for every $f\in\mathcal{D}$, and
$\int_0^1\!\int_{\mathcal{S}\times\mathcal{U}}\psi(x,u)\,q_t(du\mid x)\,\nu_t(dx)\,dt<\infty$.
\item[(ii)] $\mathcal{D}$ contains a countable subfamily that contains the constants, is closed under multiplication, and separates points, and each $\mathcal{G}^u$ is a pre-generator.
\item[(iii)] For every $f\in\mathcal{D}$ and $t\in[0,1]$,
\[
\langle f,\nu_t\rangle=\langle f,\nu_0\rangle+\int_0^t\langle\hat{\mathcal{A}}_sf,\nu_s\rangle\,ds .
\]
\end{enumerate}
Then there exist a process $X$ and a filtration $\{\mathcal{F}_t\}$ such that $f(X(t))-f(X(0))-\int_0^t\hat{\mathcal{A}}_sf(X(s))\,ds$ is an $\{\mathcal{F}_t\}$-martingale for every $f\in\mathcal{D}$, and $\mathcal{L}(X(t))=\nu_t$ for every $t\in[0,1]$.
\end{proposition}

\begin{theorem}[Equivalence of the SOT formulations]\label{thm:equivalence}
    Let \Cref{assump:cost} (lower semicontinuous, convex cost) and \Cref{assump:affine} (convex action set, control-affine generator) hold, and assume the total jump intensity is dominated by the Lyapunov function, $q(x,\mathcal{S};u)\le\kappa_V(1+V(x))$ for all $(x,u)$ and some $\kappa_V<\infty$. For all $\tilde{\pi}\in\mathcal{P}(\mathcal{S})$, the value of the three SOT formulations is the same,
    \begin{equation}
        \label{eq:SOT equivalence}
        G(\tilde{\pi}) = \tilde{G}(\tilde{\pi}) = g(\tilde{\pi}).
    \end{equation}
\end{theorem}

\begin{proof}
    We prove the equality of the three SOT values by proving the following inequalities:
    \begin{equation*}
        G(\tilde{\pi}) \geq \tilde{G}(\tilde{\pi}) \geq g(\tilde{\pi}) \geq G(\tilde{\pi}).
    \end{equation*}
    In this proof, we denote the control function for the process formulation by $\mu^Q(t,x)$ to distinguish it from the control function in the weak BAR formulation, which we continue to denote by $\mu(t,x)$.
    \begin{enumerate}
        \item $G(\tilde{\pi})\geq \tilde{G}(\tilde{\pi})$:\\
        Given $(Q,\mu^Q)\in\mathcal{A}$ with $\mathcal{L}(Q(0))=\mathcal{L}(Q(1))=\tilde\pi$, define the map:
        \[
        \mu(t,x) := \mathbb{E}[\mu^Q(t,Q(t))\mid Q(t)=x].
        \]
        From It\^o's lemma,
        \[
        \mathbb{E}[f(t,Q(t)) - f(0,Q(0))] = \int_0^t\mathbb{E}[\partial_sf(s,Q(s)) + \mathcal{G}^{\mu^Q}f(s,Q(s))]\,ds.
        \]
        Use the fact that
        \begin{align*}
            \mathbb{E}[\mathcal{G}^{\mu^Q}f(s,Q(s))] &= \mathbb{E}\big[\mathbb{E}[\mathcal{G}^{\mu^Q}f(s,Q(s))\mid Q(s)]\big]\\
            &= \mathbb{E}[\mathcal{G}^{\mu}f(s,Q(s))],
        \end{align*}
        where the inner step is the affine-averaging identity~\eqref{eq:affine-average} with $\mu(t,x)=\mathbb{E}[\mu^Q(t,Q(t))\mid Q(t)=x]$. Using the convexity of the map $u\mapsto c(t,x,u)$ and Jensen's inequality, we have
        \begin{equation*}
            \mathbb{E}\left[\int_0^tc(s,Q(s),\mu^Q(s,Q(s)))ds\right]\geq\mathbb{E}\left[\int_0^tc(s,Q(s),\mu(s,Q(s)))ds\right],
        \end{equation*}
        and thus, $G\geq \tilde{G}$.
         \item $\tilde{G}\geq g$:
        For an admissible pair $(\{\rho_t\},\mu)$, define
        \begin{equation*}
            \nu(dt,dx,du) : = \rho_t(dx)\,\delta_{\mu(t,x)}(du)\,dt.
        \end{equation*}
        We verify that $\nu$ belongs to the admissible set $\mathcal{M}$ of \Cref{def:relaxed-admissible}. The disintegration condition holds by construction. For the weak continuity of the state marginal, fix a bounded uniformly continuous $\psi$ and apply~\eqref{eq:weak FP} to the time-independent test function $f(s,x)=\psi(x)$. Observe that $t\mapsto\langle\psi,\rho_t\rangle$ is then an integral in time, hence continuous. Bounded uniformly continuous functions determine weak convergence, so the state marginal admits a weakly continuous version, and the endpoint condition $\rho_0=\rho_1=\tilde\pi$ is the loop condition of the flow. The BAR constraint of $\mathcal{M}$ is the $t=1$ case of~\eqref{eq:weak FP}, as noted after that definition. The finite-energy condition for $\nu$ reads $\int_0^1\langle 1+V,\rho_t\rangle\,dt<\infty$, which is the finite-energy condition of the flow. Finally, the costs coincide,
        \begin{equation*}
            \int c\,d\nu \;=\; \int_0^1\!\int_{\mathcal{S}} c\big(t,x,\mu(t,x)\big)\,\rho_t(dx)\,dt .
        \end{equation*}
        Since $g(\tilde\pi)$ is the infimum over $\mathcal{M}$, we conclude $\tilde G\ge g$.

        \item $g\geq G$:
        Given $\pi\in\mathcal{M}$, disintegrate $\pi(dt,dx,du)=q_t(du\mid x)\,\pi_{1,t}(dx)\,dt$, with a jointly measurable version of the kernel, and define the barycentric control
        \begin{equation*}
            \bar u(t,x) := \int_{\mathcal{U}}u\,q_t(du\mid x)\ \in\ \mathcal{U},
        \end{equation*}
        a Markov control. We first localize the BAR constraint in time. Testing~\eqref{eq:BAR constraint} with $f(s,x)=\varphi(s)\psi(x)$, where $\varphi\in C^1([0,1])$ and $\psi$ is bounded and uniformly continuous, and using the weak continuity of $\pi_{1,t}$, we obtain
        \begin{equation}\label{eq:localized-forward}
            \langle\psi,\pi_{1,t}\rangle \;=\; \langle\psi,\tilde\pi\rangle+\int_0^t\!\int_{\mathcal{S}\times\mathcal{U}}\mathcal{G}^u\psi(x)\,q_s(du\mid x)\,\pi_{1,s}(dx)\,ds\qquad\text{for all }t\in[0,1].
        \end{equation}
        By the affine-averaging identity~\eqref{eq:affine-average}, the inner integral equals $\mathcal{G}^{\bar u(s,x)}\psi(x)$, so~\eqref{eq:localized-forward} is condition (iii) of \Cref{prop:KS-superposition}. \Cref{prop:KS-superposition} then supplies a process $X$ and a filtration for which
        \[
        f(X(t))-f(X(0))-\int_0^t\mathcal{G}^{\bar u(s,\cdot)}f(X(s))\,ds
        \]
        is a martingale for every $f$ in the core, with $\mathcal{L}(X(t))=\pi_{1,t}$ for every $t$. Observe that condition (i) of the proposition is verified by $\psi(x,u)=1+q(x,\mathcal{S};u)$, since $|\mathcal{G}^uf(x)|\le 2\|f\|_\infty\,q(x,\mathcal{S};u)$, with the required integrability following from the intensity bound in the statement of the theorem together with the finite-energy condition of $\mathcal{M}$. Condition (ii) holds on a countable state space with the core spanned by the constant function and the singleton indicators, a countable family closed under multiplication that separates points, each $\mathcal{G}^u$ being a pre-generator as a jump generator. The pair $(X,\bar u)$ therefore solves the controlled martingale problem, satisfies the loop condition $\mathcal{L}(X(0))=\mathcal{L}(X(1))=\tilde\pi$, and has finite energy, $\mathbb{E}\int_0^1V(X(t))\,dt=\int_0^1\langle V,\pi_{1,t}\rangle\,dt<\infty$. Hence $(X,\bar u)\in\mathcal{A}$. Convexity of $c$ and Jensen's inequality give
        \begin{equation*}
            \int_{\mathcal{S}\times\mathcal{U}} c(t,x,u)\,\pi_t(dx,du)\ \geq\
            \int_{\mathcal{S}} c(t,x,\bar u(t,x))\,\pi_{1,t}(dx),
        \end{equation*}
        so the cost of $(X,\bar u)$ is at most $\int c\,d\pi$. Taking the infimum over $\pi\in\mathcal{M}$ yields $g\geq G$. 
    \end{enumerate}
\end{proof}

\subsection{The Dual Problem and Weak Duality}

The Lagrangian~\eqref{eq:lagrangian} suggests the dual problem
\begin{equation}\label{eq:SOT Dual}
    \sup_{f\in\mathcal{F}}\int_{\mathcal{S}}(f(1,x)-f(0,x))\,\tilde{\pi}(dx),
\end{equation}
the BAR-SOT analogue of the stochastic-transport duality of \citet[Theorem 2.4]{mikami2021stochastic}. \Cref{lem:weak-duality} below establishes weak duality, bounding the primal value $g(\tilde\pi)$ from below by the dual value. In \Cref{thm:ansatz} we establish equality at the optimizing marginal $\tilde\pi^*$. 

The supremum in~\eqref{eq:SOT Dual} is taken over the weighted class
\begin{equation}\label{eq:dual-class}
\begin{aligned}
\mathcal{F} := \Big\{f:[0,1]\times\mathcal{S}\to\mathbb{R}\ :\ f(\cdot,x)\in C^1 &\text{ for each }x,\ \ \sup_{t}\|f(t,\cdot)\|_V+\sup_{t}\|\partial_tf(t,\cdot)\|_V<\infty,\\ &\text{ and $f$ satisfies the HJB}\Big\},	
\end{aligned}
\end{equation}
where $\|\phi\|_V:=\sup_{x\in\mathcal{S}}|\phi(x)|/(1+V(x))$ is the weighted supremum norm associated with the weight $V$ of \Cref{assump:stability}, and the HJB is defined as
\begin{equation}
    \label{eq:HJB}
    \partial_tf(t,x) + \mathcal{H}f(t,x) = 0,
\end{equation}
with the Hamiltonian $\mathcal{H}$ of~\eqref{eq:Hamiltonian}.
The Lagrangian argument requires only the subsolution inequality $\partial_tf+\mathcal Hf\le0$. We impose equality in $\mathcal F$ because \Cref{lem:weak-duality} below holds verbatim for subsolutions. The maximizer exhibited in \Cref{thm:ansatz} satisfies the equality, so the restriction does not lose anything at $\tilde\pi^*$. The soft-feasible class of \Cref{sec:consistency} is defined with the inequality.
 The weighted class (rather than $UC_b^1$, for instance) is required because the dual optimizer grows with the state whenever the running cost does. Indeed, \Cref{thm:ansatz} exhibits the maximizer as $h(x)+(t-T)g^*$ with $h$ the Poisson potential of~\eqref{eq:Poisson}, which is unbounded for the queue of \Cref{ex:queue}. The pairing in~\eqref{eq:SOT Dual} is finite for every $\tilde\pi$ with $\langle V,\tilde\pi\rangle<\infty$, to which the dual statement is henceforth restricted. 

\Cref{thm:equivalence} requires only that the Lyapunov function dominate the total jump intensity. The duality results from this point on assume in addition that the total jump intensity is bounded, as it is in every model considered in this paper, together with a compatibility bound between the jump kernel and $V$.
\begin{lemma}[Weak duality]\label{lem:weak-duality}
Suppose the total jump intensity is bounded, $\sup_{x,u}q(x,\mathcal{S};u)<\infty$, and that the jump kernel satisfies the compatibility bound $\sup_{u\in\mathcal{U}}\int_{\mathcal{S}}(1+V(y))\,q(x,dy;u)\le\kappa_V\,(1+V(x))$ for some $\kappa_V<\infty$. Then for every $\tilde\pi\in\mathcal{P}(\mathcal{S})$ with $\langle V,\tilde\pi\rangle<\infty$,
\begin{equation}\label{eq:weak-duality}
    g(\tilde{\pi})\ \ge\ \sup_{f\in\mathcal{F}}\int_{\mathcal{S}}(f(1,x)-f(0,x))\,\tilde{\pi}(dx),
\end{equation}
and the analogous inequality holds on every horizon $[0,T]$.
\end{lemma}
Since $V\ge0$, the compatibility bound implies the intensity domination of \Cref{thm:equivalence} with the same constant. Both hypotheses hold for the nearest-neighbor kernel of \Cref{ex:queue} with the quadratic $V$, where the total intensity is at most $\sum_i\lambda_i+\max_i\mu_i$.
The proof of \Cref{lem:weak-duality} is given in \Cref{ec:proofs}.

\subsection{The Ergodic Problem}
To obtain the value function for the long-term average control problem, we still need to solve the meta-SOT problem by infimizing over all endpoint laws $\tilde\pi$. Thus, we are actually looking for,
\begin{equation}\label{eq:meta-SOT}
    g^*_1 := \inf_{\tilde{\pi}} g(\tilde{\pi}).
\end{equation}
The ergodic problem is when the optimization is done after taking an average time cost and then looking at the infinite-horizon problem, i.e., 
\begin{equation}\label{eq:ergodic SOT}
    \inf_{\tilde{\pi}}\liminf_{T\to\infty}\frac{1}{T}g_T(\tilde{\pi}),
\end{equation}
where $g_T(\tilde\pi)$ is the SOT problem with horizon $[0,T]$ instead of $[0,1]$, and $g^*_T:=\inf_{\tilde\pi}g_T(\tilde\pi)$. Before showing that the ergodic SOT problem and the SOT problem on the horizon $[0,1]$ have the same value, we prove some useful results.

In this section, we consider only time-independent costs, i.e., $c(t,x,u) = c(x,u)$. Before stating the theorem, we note that the pointwise maximizer appearing in it is well defined. In particular, observe that for each $x$, the map $u\mapsto\mathcal{G}^uh(x)-c(x,u)$ is upper semicontinuous on the compact set $\mathcal{U}$ ($u\mapsto\mathcal{G}^uh(x)$ is affine on the convex $\mathcal{U}\subset\mathbb{R}^k$, hence continuous, and $c$ is lower semicontinuous), so a maximizer exists, and it admits a measurable version by the measurable maximum theorem~\citep[Theorem~18.19]{aliprantisborder2006}. Sufficient conditions for the existence of the Poisson pair $(h,g^*)$ hypothesized below are given by \citet{guo2009}.

\begin{assumption}[Coercive cost]\label{assump:coercive}
The map $x\mapsto\inf_{u\in\mathcal{U}}c(x,u)$ has relatively compact sublevel sets, and $c$ is bounded on compact subsets of $\mathcal{S}\times\mathcal{U}$. Equivalently, for every $M<\infty$ there is a compact set $K\subset\mathcal{S}$ such that $\inf_{u\in\mathcal{U}}c(x,u)\ge M$ for all $x\notin K$. On $\mathcal{S}=\mathbb{Z}^k_+$ this states that $\inf_{u\in\mathcal{U}}c(x,u)\to\infty$ as $|x|\to\infty$.
\end{assumption}
Since $\inf_{u}c(x,u)\le c(x,\mu(x))\le W_\mu(x)$ by the cost bound of \Cref{assump:stability}, coercivity forces every certificate $W_\mu$ to be coercive; a model with bounded cost can satisfy \Cref{assump:stability} but not this assumption.
\begin{lemma}[Recurrence of the maximizer]\label{lem:maximizer-recurrence}
Let \Cref{assump:coercive} hold. Let $h:\mathcal{S}\to\mathbb{R}$ and $g^*\in\mathbb{R}$ satisfy $g^*+\mathcal{G}^{u^{**}}h(x)=c\big(x,u^{**}(x)\big)$ for all $x$, where $u^{**}(x)\in\arg\max_u\{\mathcal{G}^uh(x)-c(x,u)\}$ is a measurable maximizer. Suppose that $-h$ is bounded below with relatively compact sublevel sets, and that the process generated by $u^{**}$ is irreducible. Then this process is positive Harris recurrent with a unique invariant measure $\tilde\pi^*$, and $\langle c(\cdot,u^{**}),\tilde\pi^*\rangle<\infty$.
\end{lemma}
The proof of \Cref{lem:maximizer-recurrence} is given in \Cref{ec:proofs}.
For \Cref{ex:queue}, \Cref{assump:coercive} is immediate. The sublevel condition requires $-h$ to grow without bound, consistent with the quadratic growth of the relative value function of a stable queue.

\begin{theorem}[Strong duality via the Poisson ansatz]\label{thm:ansatz}
Let \Cref{assump:affine}, \Cref{assump:cost}, and the hypotheses of \Cref{lem:weak-duality} hold. Suppose there exist a function $h:\mathcal{S}\to\mathbb{R}$ with $\|h\|_V<\infty$ and a scalar $g^*$ such that, with the pointwise maximizer
\[
    u^{**}(x) := \arg\max_u\{\mathcal{G}^uh(x) - c(x,u)\},
\]
the pair $(h,g^*)$ solves the Poisson equation,
\begin{equation}\label{eq:Poisson}
    g^* + \mathcal{G}^{u^{**}}h(x) = c(x,u^{**}) \qquad\text{for all }x.
\end{equation}
Assume the conditions of \Cref{lem:maximizer-recurrence}. The controlled process generated by $u^{**}$ is then positive Harris recurrent with a unique invariant measure $\tilde\pi^*$, $(\mathcal{G}^{u^{**}})^*\tilde\pi^*=0$; we assume in addition that $\langle V,\tilde\pi^*\rangle<\infty$. Then the ansatz 
\begin{equation}\label{eq:ansatz}
    f^*(t,x):=h(x)+(t-T)g^*
\end{equation} 
satisfies:
\begin{enumerate}
\item[(i)] $f^*$ solves the finite-horizon HJB~\eqref{eq:HJB}, hence $f^*\in\mathcal{F}$;
\item[(ii)] $f^*$ attains the dual supremum at $\tilde\pi^*$, with $\langle f^*_T-f^*_0,\tilde\pi^*\rangle = Tg^* = \sup_{f\in\mathcal{F}}\langle f_T-f_0,\tilde\pi^*\rangle$;
\item[(iii)] for every $\tilde\pi$, $g_T(\tilde\pi)\ge Tg^*$; hence $g^*_T=\inf_{\tilde\pi}g_T(\tilde\pi)=Tg^*$, attained at $\tilde\pi^*$.
\end{enumerate}
\end{theorem}
\begin{proof}
\emph{(i)} Since $f^*(t,\cdot)$ and $h$ differ by the spatially constant $(t-T)g^*$, which every generator $\mathcal{G}^u$ annihilates, $\mathcal{G}^uf^*(t,x)=\mathcal{G}^uh(x)$ for all $t$ and $u$. By the definition of $u^{**}$ as the maximizer of the Hamiltonian,
\[
\mathcal{H}f^*(t,x) = \sup_u\{\mathcal{G}^uh(x)-c(x,u)\} = \mathcal{G}^{u^{**}}h(x) - c(x,u^{**}),
\]
so, substituting and using the Poisson equation,
\[
\begin{aligned}
\partial_t f^*(t,x) + \mathcal{H}f^*(t,x)
&= g^* + \mathcal{G}^{u^{**}}h(x) - c(x,u^{**}) = 0.
\end{aligned}
\]

Thus $f^*\in\mathcal{F}$.

\emph{(ii)} Let $f\in\mathcal{F}$ be any HJB solution and set $w:=f^*-f$. Subtracting the two HJBs,
\[
\partial_t w + \big(\mathcal{H}f^*(t,x)-\mathcal{H}f(t,x)\big)=0.
\]
The control $u^{**}$ maximizes the Hamiltonian at $f^*$, whereas for $f$ it is merely feasible, not necessarily optimal. Therefore
\[
\mathcal{H}f^*(t,x)=\mathcal{G}^{u^{**}}f^*(t,x) - c(x,u^{**}),
\qquad
\mathcal{H}f(t,x)\ge \mathcal{G}^{u^{**}}f(t,x) - c(x,u^{**}),
\]
and subtracting, by the linearity of the operator $f\mapsto\mathcal{G}^{u^{**}}f$, gives $\mathcal{H}f^*-\mathcal{H}f\le \mathcal{G}^{u^{**}}f^*-\mathcal{G}^{u^{**}}f=\mathcal{G}^{u^{**}}w$, whence
\[
\partial_t w + \mathcal{G}^{u^{**}}w \ge 0.
\]

Now, consider the \emph{constant} flow $\rho_t\equiv\tilde\pi^*$, which is admissible under $u^{**}$ since $\tilde\pi^*$ is $u^{**}$-invariant ($\partial_t\rho_t=(\mathcal{G}^{u^{**}})^*\tilde\pi^*=0$), and it is periodic with $\rho_0=\rho_T=\tilde\pi^*$. The weak BAR~\eqref{eq:weak FP} for this flow yields
\[
\langle w_T-w_0,\tilde\pi^*\rangle
= \int_0^T\!\int_{\mathcal{S}}\big\{\partial_t w + \mathcal{G}^{u^{**}}w\big\}\,\tilde\pi^*(dx)\,dt \ \ge\ 0,
\]
so $\langle f^*_T-f^*_0,\tilde\pi^*\rangle\ge\langle f_T-f_0,\tilde\pi^*\rangle$ for every $f\in\mathcal{F}$. (Since $w$ has the weighted growth of the class~\eqref{eq:dual-class} rather than lying in $UC_b^1$, the identity~\eqref{eq:weak FP} is applied to $w$ via a routine truncation, using $\langle V,\tilde\pi^*\rangle<\infty$ and dominated convergence.) Since $f^*_T-f^*_0=Tg^*$ is constant in $x$, the attained value is $\langle Tg^*,\tilde\pi^*\rangle = Tg^*$.

\emph{(iii)} For any $\tilde\pi$, pairing $f^*$ against an arbitrary admissible $\pi$ via the truncation argument of \Cref{lem:weak-duality} gives $g_T(\tilde\pi)\ \ge\ \langle f^*_T-f^*_0,\tilde\pi\rangle = Tg^*$; no moment condition on $\tilde\pi$ is needed, since the endpoint increment $f^*_T-f^*_0\equiv Tg^*$ is constant in $x$, so its truncated pairings are bounded, while the interior domination uses only the finite-energy condition on $\pi$. If no admissible $\pi$ exists, $g_T(\tilde\pi)=+\infty$ and the bound is trivial. For the reverse inequality at $\tilde\pi^*$, the stationary pair $\pi^*(dt,dx,du):=\delta_{u^{**}(x)}(du)\,\tilde\pi^*(dx)\,dt$ is admissible for the relaxed problem: it has finite energy since $\langle V,\tilde\pi^*\rangle<\infty$, and after integrating $\partial_tf$ in time the BAR constraint reduces to $\int_0^T\langle\mathcal{G}^{u^{**}}f(t,\cdot),\tilde\pi^*\rangle\,dt=0$, the invariance of $\tilde\pi^*$. Its cost is
\[
\int c\,d\pi^* \;=\; T\,\langle c(\cdot,u^{**}),\tilde\pi^*\rangle \;=\; T\big(g^*+\langle\mathcal{G}^{u^{**}}h,\tilde\pi^*\rangle\big)\;=\;Tg^*,
\]
by the Poisson equation~\eqref{eq:Poisson} and invariance, the latter applied to $h$ via the truncation argument of \Cref{lem:weak-duality}. Hence $g_T(\tilde\pi^*)\le Tg^*$, so $g_T(\tilde\pi^*)=Tg^*$ and $\inf_{\tilde\pi}g_T(\tilde\pi)=Tg^*$, attained at $\tilde\pi^*$.
\end{proof}

By \Cref{thm:ansatz}(iii), $g^*_T=Tg^*$ for every horizon $T>0$, so the horizon enters only as a scale,
\[
g^*_1 = g^* = \frac{g^*_T}{T}.
\]
It remains to connect the finite-horizon values to the ergodic problem itself.
\begin{theorem}[The meta-SOT value is the ergodic value]\label{thm:ergodic-value}
    Let the hypotheses of \Cref{thm:ansatz} and \Cref{lem:weak-duality} hold.
    \begin{enumerate}
    \item[(i)] The value of the ergodic SOT problem \eqref{eq:ergodic SOT} is $g^*$.
    \item[(ii)] The ergodic control value $g^{\mathrm{erg}}$ of \eqref{eq:ergodic-cost} also equals $g^*$, attained at $u^{**}$, which is admissible. The long-run average-cost control problem and the ergodic self-transport problem thus have the same value.
    \end{enumerate}
\end{theorem}
The proof of \Cref{thm:ergodic-value} is given in \Cref{ec:proofs}.

\begin{remark}[Regulated processes and costed boundary control]\label{rem:regulated}
Reflected dynamics are already within the framework. Observe that the reflection of \Cref{ex:queue} enters through the generator itself, via the boundary indicator in~\eqref{eq:queue-generator}, so the results of this and the following sections apply to regulated processes without modification. Charging the boundary is a different matter. If the local time accumulated at the boundary carries a proportional capacity cost, the problem acquires a singular control component. Its BAR identity picks up an expected boundary occupation measure, and the dual acquires a gradient constraint at the boundary. This extension requires the singular part of the martingale-problem framework of \citet{kurtz2001stationary} and a duality theory of its own, and we will develop it in separate work.
\end{remark}

\section{The Outer Problem as a Periodic Boundary-Value Problem}\label{sec:periodic-bvp}

\Cref{sec:bar-sot} reduced the self-transport problem at a fixed endpoint law $\tilde\pi$ to its dual~\eqref{eq:SOT Dual}, a problem in a single BAR multiplier $f$, and \Cref{thm:ergodic-value} identified the ergodic cost rate with the meta-SOT value~\eqref{eq:meta-SOT}, in which the endpoint law $\tilde\pi$ is itself optimized. We now observe that the dual already established in \Cref{sec:bar-sot} makes this outer optimization a \emph{periodic boundary-value problem} (BVP) for the BAR multiplier alone, with no reference to $\tilde\pi$ and, at this stage, with no regularization of any kind. This periodic BVP is the object we compute with for the remainder of the paper. \Cref{sec:optimality-system} explains why it cannot be solved as it stands, and \Cref{sec:entropy-reg} smooths its Hamiltonian so that it can be solved by gradient methods.

\subsection{The Optimal Dual is Time-Periodic}\label{sec:periodic-dual}

It is tempting to attack the meta-SOT problem~\eqref{eq:meta-SOT} using block alternation, by fixing the control, updating $\tilde\pi$, and repeating. This is ill-posed, however, since $\tilde\pi$ enters the relaxed marginal-measure problem only through the self-transport endpoint constraint $\rho_0 = \rho_T = \tilde\pi$ and carries no running-cost gradient. Therefore, an ``update $\tilde\pi$'' step would have no variational content, and it merely projects $\tilde\pi$ onto the current endpoint marginal $\rho_T$, the stationary law of the \emph{current}, and not the optimal, policy.

Now, observe that weak duality (\Cref{lem:weak-duality}), taken on the horizon $[0,T]$, gives
\begin{equation}\label{eq:horizon-T-dual}
    g_T(\tilde\pi) \;\geq\; \sup_{f\in\mathcal F}\,\langle f_T - f_0,\,\tilde\pi\rangle,
\end{equation}
with equality at $\tilde\pi^*$, where \Cref{thm:ansatz} shows the supremum is attained at the lift ansatz~\eqref{eq:ansatz}, $f^*(t,x) = h(x) + (t-T)g^*$, with $h$ the Poisson potential of~\eqref{eq:Poisson}. The endpoint gap of the optimal dual is therefore
\begin{equation}\label{eq:periodic-dual}
    f^*_T(x) - f^*_0(x) \;=\; T g^*, \qquad \text{independent of } x.
\end{equation}
The optimal BAR multiplier is \emph{time-periodic up to an additive constant}. In particular, its non-constant part, i.e., the Poisson potential $h$, returns to itself over $[0,T]$. At the same time, the constant $Tg^*$ is the meta-SOT value, recovered as $\langle f^*_T - f^*_0,\,\tilde\pi\rangle$ paired against any probability measure. This is the dual content of $g^*_T = Tg^*$ and of \Cref{thm:ergodic-value}.

Equation~\eqref{eq:periodic-dual} removes $\tilde\pi$ from the outer optimization. Because the endpoint gap of $f^*$ is constant in $x$, its pairing $\langle f^*_T-f^*_0,\tilde\pi\rangle = Tg^*$ is the same for every $\tilde\pi$. By \Cref{thm:ansatz}(iii), this constant is a lower bound on $g_T(\tilde\pi)$ for every admissible $\tilde\pi$ and equals $g_T(\tilde\pi^*)$; the meta-infimum is therefore $Tg^*$. Indeed, the endpoint law no longer appears as a variable. The rate $g^*$, in turn, is not solved for separately, and instead it can be read off the endpoint gap of the periodic dual. The outer problem is thus solved by finding the dual $f^*$ whose endpoint gap is $x$-independent, viz., a \emph{periodic boundary-value problem} for the backward HJB~\eqref{eq:HJB},
\begin{equation}\label{eq:periodic-bvp-hard}
    \partial_t f_t(x) + \mathcal H f_t(x) = 0 \quad (0 < t < T),
    \qquad f_T - f_0 = \text{const in } x,
\end{equation}
with the Hamiltonian $\mathcal H$ of~\eqref{eq:Hamiltonian}. The periodicity constraint stands instead of the terminal condition $f_T=\psi$ for an ordinary finite-horizon problem. Substituting the ansatz~\eqref{eq:ansatz} into~\eqref{eq:periodic-bvp-hard} solves the latter and yields $\partial_tf^*=g^*$. The Poisson equation~\eqref{eq:Poisson} gives $\mathcal Hf^*=\mathcal Hh=\mathcal G^{u^{**}}h-c(\cdot,u^{**})=-g^*$, so the residual vanishes and the endpoint gap is $Tg^*$. Conversely, any solution of~\eqref{eq:periodic-bvp-hard} in $\mathcal F$ with endpoint gap $Tg$ is a feasible dual, so by~\eqref{eq:horizon-T-dual} its constant satisfies $Tg\le g_T(\tilde\pi)$ for every $\tilde\pi$, hence $g\le g^*$; the periodic BVP is solved for the solution of maximal endpoint gap, which is attained by the ansatz.

\begingroup\color{red}
\subsection{Existence and Uniqueness of the Periodic Solution}\label{sec:periodic-wellposed}

We now give conditions under which~\eqref{eq:periodic-bvp-hard} has a solution and all of its solutions in $\mathcal F$ agree up to an additive constant. Assume that the cost $c(x,u)$ does not depend on time. For a control kernel $q(du\mid x)$ on $\mathcal U(x)$ we write $\mathcal G^q\phi(x):=\int\mathcal G^u\phi(x)\,q(du\mid x)$ and $c^q(x):=\int c(x,u)\,q(du\mid x)$, and a \emph{policy} is a rule that selects, at each time $t$, a probability measure $q_t$ on $\mathcal U(X_{t-})$ as a function of the past of the process. Recall \Cref{assump:cost,assump:coercive}, and suppose the jump kernel satisfies the hypotheses of \Cref{lem:weak-duality}, with total intensity bounded by $\bar Q:=\sup_{x,u}q(x,\mathcal S;u)$. In addition, we further stipulate that:
\begin{assumption}\label{assump:standing}
\begin{enumerate}
\item[(i)] For each $x$, the map $u\mapsto q(x,\cdot\,;u)$ is continuous for the weighted total-variation norm $\sum_{y}(1+V(y))\,|\nu(\{y\})|$.
\item[(ii)] There is a stationary control kernel $q^\circ$ whose chain is irreducible, with $c^{q^\circ}\le C_c(1+V)$, together with $W^\circ:\mathcal S\to[0,\infty)$, $W^\circ\le C_W(1+V)$, a finite set $C$ and $b<\infty$ such that $\mathcal G^{q^\circ}W^\circ\le-c^{q^\circ}+b\,\mathbf 1_C$.
\end{enumerate}
\end{assumption}
\Cref{assump:standing}\,(i) allows a measureable selection of an action that attains the maximum in the Hamiltonian. On the other hand, \Cref{assump:standing}\,(ii) guarantees that there exists at least one stable policy, ensuring the problem is well-posed. \Cref{assump:standing} will be a standing assumption from this point out in the paper. 

\begin{assumption}\label{assump:uniform}
\begin{enumerate}
\item[(i)] There exists $U:\mathcal S\to[0,\infty)$ with $\sum_yq(x,\{y\};u)U(y)<\infty$ for all $(x,u)$, a finite set $F$ and $b_U<\infty$ such that $\mathcal G^uU\le-(1+V)+b_U\,\mathbf 1_F$ for every $x\in\mathcal S$ and every $u\in\mathcal U(x)$.
\item[(ii)] There exist a state $o\in F$, $t_0>0$ and $\delta>0$ such that, from every $x\in F$ and under every policy, the hitting time $\tau_o$ of $o$ satisfies $\mathbb P_x(\tau_o\le t_0)\ge\delta$.
\end{enumerate}
\end{assumption}
\Cref{assump:uniform} is uniform over all actions. \Cref{assump:uniform}\,(i) ensures every policy, including the unknown one from $f$, is stable with moment bounds that don’t depend on the policy. This justifies stopping the processes and passing to the limit in the proof. \Cref{assump:uniform}\,(ii) is a uniqueness assumption that excludes multichain behavior, since there is one state that every chain is assumed to reach. It is used only in \Cref{thm:periodic-wellposed}, where it takes the place of the recurrence hypotheses of \Cref{lem:maximizer-recurrence}.
\begin{theorem}[Existence and uniqueness of the periodic solution]\label{thm:periodic-wellposed}
Let \Cref{assump:standing,assump:uniform} hold.
\begin{enumerate}
\item[(i)] There exists a solution $(h,g^*)$ of $g^*+\mathcal Hh=0$ with $\|h\|_V<\infty$ and $-h$ bounded below, and $g^*$ is the optimal long-run average cost over all policies, from every initial state. Every stationary control kernel $q$ with $q(A(x)\mid x)=1$ for all $x$, where $A(x):=\arg\max_{u\in\mathcal U(x)}\{\mathcal G^uh(x)-c(x,u)\}$, is optimal from every initial state, and its chain has a unique invariant law $\pi_q$, with $\langle V,\pi_q\rangle<\infty$.
\item[(ii)] Every solution $f\in\mathcal F$ of~\eqref{eq:periodic-bvp-hard} has the form $f(t,x)=h(x)+(t-T)g^*+k$ for a constant $k$. In particular its endpoint gap is $Tg^*$, and every pair $(h',g)$ with $\|h'\|_V<\infty$ and $g+\mathcal Hh'=0$ satisfies $g=g^*$ and $h'=h+k$ for a constant $k$.
\item[(iii)] For every $q$ as in (i), $\pi_q$ minimizes $g_T$ over endpoint laws $\tilde\pi$ with $\langle V,\tilde\pi\rangle<\infty$, with $g_T(\pi_q)=Tg^*$. The set of these laws $\pi_q$ is convex, it coincides with the set of invariant laws of all optimal stationary control kernels, and it is a single law when $A(x)$ is a singleton for every $x$.
\end{enumerate}
\end{theorem}
The proof of \Cref{thm:periodic-wellposed} is given in \Cref{ec:proofs}.
\begin{proposition}\label{prop:queue-periodic}
In \Cref{ex:queue}, let $\lambda_i>0$ for every $i$, $\rho<1$, $c(x)=\sum_ic_ix_i$ with $c_i>0$, $V(x)=|x|^2$, and take the non-idling action sets $\mathcal U(x)=\mathrm{conv}\{\mu_ie_i:\ x_i>0\}$ for $x\ne0$ and $\mathcal U(0)=\{0\}$. Then \Cref{assump:standing,assump:uniform} hold with $o=0$. Moreover, the solution $(h,g^*)$ of \Cref{thm:periodic-wellposed} satisfies the hypotheses of \Cref{thm:ansatz}: $-h$ has finite sublevel sets, and the chain of every stationary control supported on $A$ is irreducible.
\end{proposition}
The proof of \Cref{prop:queue-periodic} is given in \Cref{ec:proofs}.
\endgroup

Everything about the ergodic problem is therefore encoded in one periodic solution of the periodic BVP for the backward HJB. The value $g^*$ is its endpoint gap, the optimal control $u^{**}$ is the maximizer in its Hamiltonian~\eqref{eq:Hamiltonian}, and the optimal endpoint law $\tilde\pi^*$ is recovered afterwards as the stationary law of the chain that $u^{**}$ controls. It remains to show how to compute this solution.

\section{The Optimality System and the Difficulty of a Direct Solve}\label{sec:optimality-system}
The periodic BVP~\eqref{eq:periodic-bvp-hard} of \Cref{sec:periodic-bvp} is a complete characterization of the ergodic problem, but it is not yet a computable one. We now assemble the coupled optimality system behind it and explain why a direct solve of this system is difficult. This motivates the smoothing of the Hamiltonian developed in \Cref{sec:entropy-reg}.

\subsection{The optimality system}

The dual~\eqref{eq:SOT Dual} reduces the self-transport problem to finding a single dual potential $f$, and the lift of \Cref{thm:ansatz} reduces that potential to a relative value $h$ and a scalar rate $g^*$ through the ansatz $f^*(t,x)=h(x)+(t-T)g^*$ of~\eqref{eq:ansatz}. Recovering the value and the optimal control therefore amounts to solving a coupled forward--backward system with a periodic boundary condition. Specifically, the system comprises a backward equation for $f$ (HJB), a forward equation for $\rho_t$ (Kolmogorov), and the periodic-in-time closure~\eqref{eq:optimality-system-bc} established in \Cref{sec:periodic-bvp}. To be precise, 
\begin{equation}\label{eq:optimality-system-backward}
    \partial_t f(t,x) + \mathcal{H}f(t,x) = 0,
    \qquad u^{f}(t,x)\in\arg\max_{u\in\mathcal{U}}\{\mathcal{G}^uf(t,x) - c(x,u)\},
\end{equation}
is the backward HJB~\eqref{eq:HJB} with Hamiltonian~\eqref{eq:Hamiltonian} together with the forward Kolmogorov (master) equation for the marginal flow $\rho_t$ on $\mathcal{S}$ under the control it induces,
\begin{equation}\label{eq:optimality-system-forward}
    \partial_t \rho_t = \big(\mathcal{G}^{u^{f}}\big)^{*}\rho_t,
\end{equation}
with $\mathcal{G}^{u}$ the generator~\eqref{eq:generator}, and the whole system closed by the periodicity condition on the dual,
\begin{equation}\label{eq:optimality-system-bc}
    f_T - f_0 = \text{const in } x,
\end{equation}
together with the self-transport condition $\rho_0=\rho_T$ on the forward flow.

Observe that since the running cost here is independent of $\rho$, the backward HJB~\eqref{eq:optimality-system-backward} does \emph{not} depend on the marginal flow $\rho_t$. Instead, $f$ determines the optimal control $u^f$, which in turn drives $\rho_t$ through~\eqref{eq:optimality-system-forward}, but $\rho_t$ never feeds back into the equation for $f$. The forward equation is a passive readout once $f$ is known: its periodic solution is the stationary law of the control $u^{f}$, which at the optimum is the endpoint law $\tilde\pi^*$ of \Cref{thm:ansatz}. The entire difficulty in this computation lies on the dual side. Under the ansatz~\eqref{eq:ansatz}, the maximizer $u^{f^*}$ is time-independent and coincides with the $u^{**}$ of \Cref{thm:ansatz}.

\subsection{Why a direct solve is hard}

Two structural features make~\eqref{eq:optimality-system-backward}--\eqref{eq:optimality-system-bc} resist a direct solution. They are properties of the problem itself, independent of any particular method for solving it.

\smallskip\noindent\textbf{(a) The rate $g^*$ is an unknown one must find along with the solution.}
The periodic closure~\eqref{eq:optimality-system-bc} forces the endpoint gap $f_T-f_0$ to be constant in $x$, equal to $Tg^*$. The lift in~\eqref{eq:ansatz} then requires $h$ to solve the Poisson equation~\eqref{eq:Poisson}, $g^* + \mathcal{G}^{u^{**}}h = c(x,u^{**})$. This is a nonlinear eigenvalue problem. A consistent time-periodic potential exists for exactly one value of $g^*$, and at that value $h$ is determined only modulo constants, since the generator annihilates constants and is therefore non-coercive. See \citet{guo2009} for the corresponding theory of the average-cost optimality equation on unbounded state spaces.
An \emph{ordinary} finite-horizon problem, with a prescribed terminal condition (or a discounted problem, with its discount factor), carries a terminal or zeroth-order term that makes the operator invertible, {so that the solution is unique}. Our periodic closure~\eqref{eq:optimality-system-bc} does not prescribe terminal data, and instead it asks only that the dual return to itself up to a constant that is not known in advance. {Solutions are therefore unique up to an additive constant (\Cref{thm:periodic-wellposed}\,(ii)).} In technical terms, the periodicity reinstates exactly the eigenvalue structure of the ergodic problem. Lifting to a finite horizon, therefore, changes the \emph{form} of the problem, not its difficulty.

\smallskip\noindent\textbf{(b) The HJB is fully nonlinear and, in general, degenerate.}
It is well-known that the equation maximizes over controls, and is therefore nonlinear. Furthermore, the maximizer often flips abruptly and the value develops a crease that need not be smooth. Precisely, the Hamiltonian $\mathcal{H}f=\sup_u\{\mathcal{G}^uf-c\}$ is a supremum over controls, so~\eqref{eq:optimality-system-backward} is a fully nonlinear equation in $f$. When the generator is control-affine, the cost control-independent and the set of admissible controls finite, i.e., the scheduling case, the control-dependent part of $\mathcal{H}f$ reduces to the support function of the action set, $\max_i \mu_i f^i_-$ over the vertices $\{\mu_ie_i\}$, whose maximizer is a discontinuous bang-bang switch (cf.\ \Cref{ex:singular}, switching on the sign of $\mu_1 f^1_- - \mu_2 f^2_-$). The relative value $h$ is then generally non-differentiable, with a crease along a switching surface whose location is itself part of the unknown. When $\mathcal{S}$ is a continuum, no classical solution to the HJB need exist and one is in the viscosity-solution regime. On the lattice $\mathbb{Z}^k_+$ the equation is a difference equation and smoothness is not at issue, but the abrupt switch of the maximizer across an unknown surface remains the operative difficulty.

For computation, (b) is the decisive obstacle. The natural way to solve the periodic BVP at scale is to parametrize $f$, drive the residual $\partial_tf+\mathcal Hf$ to zero by gradient descent, and impose~\eqref{eq:optimality-system-bc} as a constraint. But the map $f\mapsto\mathcal Hf$ is only piecewise smooth; observe that its derivative with respect to $f$ is the bang-bang maximizer $u^f$ itself, which jumps across the switching surface and is not defined on it. A gradient method applied to the raw residual therefore receives a discontinuous search direction precisely on the set where the solution is being decided.

\smallskip
Finally, $f$ and $\rho_t$ live on the state space $\mathcal{S}$ (the lattice $\mathbb{Z}^k_+$ for the queue), which is large and in principle unbounded, so the problem is also high-dimensional. We flag this as a matter of scale rather than structure. It does not change the nature of~\eqref{eq:optimality-system-backward}--\eqref{eq:optimality-system-bc}, and it is addressed separately by the neural representation of the dual in \Cref{sec:numerical}.

\section{Entropy Regularization: Smoothing the Hamiltonian}\label{sec:entropy-reg}
Difficulty (b) in the previous section blocks a gradient-based numerical solution. It is therefore natural to replace $\mathcal H$ by a smooth approximation $\mathcal H^\varepsilon$ and to solve the periodic BVP for the smoothed equation instead. There is more than one way to smooth a supremum. In this section we achieve this through a relative-entropy penalty on the control, which replaces the hard supremum in the Hamiltonian by a log-sum-exp whose derivative in $f$ is a Gibbs kernel over the admissible actions, a smooth randomized control in place of the bang-bang switch. This choice has two advantages: first, it has a variational meaning on the primal side, as an entropic regularization of the self-transport problem in the sense of the Schr\"odinger bridge, and, second, its error relative to the unregularized value can be quantified (\Cref{sec:consistency}). Note that difficulty (a) is not resolved by the smoothing itself. It is handled by the periodic closure of the dual, which is retained unchanged in the smoothed formulation. The endpoint gap of the regularized dual is again constant in $x$, and the rate is read off from it rather than solved for as a separate eigenvalue.
\subsection{Entropy Regularization}

We now regularize the relaxed marginal-measure problem~\eqref{eq:relaxed marginal measures SOT}. The purpose is to remove the switching discontinuity identified as difficulty (b) in \Cref{sec:optimality-system}, so that the periodic BVP of \Cref{sec:periodic-bvp} can be solved by gradient methods. The regularization is introduced at the level of the primal problem at a fixed endpoint law $\tilde\pi$, where its effect on the dual is transparent; by \Cref{sec:periodic-bvp} the endpoint law drops out once periodicity is imposed, and we return to the periodic problem in \Cref{sec:soft-periodic-bvp}. For $\pi\in\mathcal M$, conditions 1 and 2 of \Cref{def:relaxed-admissible} give $\pi(dt,dx,du)=\rho_t(dx)\,q_t(du\mid x)\,dt$, with $\rho_t:=\pi_{1,t}$ the state marginal, $\rho_0=\rho_1=\tilde\pi$, and $q_t(\cdot\mid x)\in\mathcal P(\mathcal U)$ the randomized control kernel; we write $(\rho,q)\in\mathcal A_{\mathcal M}(\tilde\pi)$ when the measure so built lies in $\mathcal M$. Fix a reference control kernel $r_{t,x}\in\mathcal P(\mathcal U)$, $(t,x)\in[0,1]\times\mathcal S$. The KL-regularized problem is
\begin{equation}
\begin{aligned}
g^\varepsilon(\tilde{\pi})
&:= \inf_{(\rho,q)\in\mathcal A_{\mathcal M}(\tilde{\pi})}
\Bigg\{
\int_0^1 \int_{\mathcal{S}} \int_{\mathcal{U}}
c(t,x,u)\,q_t(du\mid x)\rho_t(dx)\,dt \\
&\qquad {}+ \varepsilon \int_0^1 \int_{\mathcal{S}}\int_\mathcal{U}\,
D_\mathrm{KL}\bigl(q_t(\cdot\mid x)\,\|\,r_{t,x}\bigr)\rho_t(dx)\,dt
\Bigg\},
\end{aligned}
\end{equation}
The entropy term does not replace the BAR constraint. It penalizes the deviation of the control kernel from the reference and leaves the endpoint and BAR structure of $\mathcal M$ untouched; the only change is how the conditional control law is selected once the state marginal and the BAR multiplier are fixed, which is what makes the problem suitable for an iterative construction.

\subsection{BAR Dual and Gibbs Kernel}\label{sec:gibbs-kernel}

With the same dual multiplier that enforces the weak BAR constraint, the entropically regularized Lagrangian becomes:
\begin{equation}
\begin{aligned}
\mathcal L^\varepsilon(\rho,q;f)
&= \int_{\mathcal{S}}\bigl(f(1,x)-f(0,x)\bigr)\tilde{\pi}(dx) \\
&\quad {}+ \int_0^1 \int_{\mathcal{S}}
\Bigg[ -\partial_t f(t,x) \\
&\qquad {}+ \int_{\mathcal{U}}
\bigl(c(t,x,u)-\mathcal{G}^uf(t,x)\bigr)\,q_t(du\mid x) \\
&\qquad {}+ \varepsilon\,D_\mathrm{KL}\bigl(q_t(\cdot\mid x)\,\|\,r_{t,x}\bigr)
\Bigg]\rho_t(dx) dt.
\end{aligned}
\end{equation}

For fixed $(t,x,\rho,f)$, the minimization over the conditional control law is pointwise:
\begin{equation}
q_f^\varepsilon(\cdot\mid t,x)
:=
\arg\min_{q\in\mathcal P(\mathcal{U})}
\left\{
\int_{\mathcal{U}}\bigl(c(t,x,u)-\mathcal{G}^uf(t,x)\bigr)\,q(du)
+
\varepsilon\,D_\mathrm{KL}\bigl(q\,\|\,r_{t,x}\bigr)
\right\}.
\end{equation}

This minimizer has the Gibbs form:
\begin{equation}\label{eq:gibbs-kernel}
q_f^\varepsilon(du\mid t,x)
=
\frac{
\exp\!\left(\dfrac{\mathcal{G}^uf(t,x)-c(t,x,u)}{\varepsilon}\right)\,r_{t,x}(du)
}{
\mathcal{Z}^\varepsilon f(t,x)
},
\end{equation}
with normalization given by the operator $\mathcal{Z}^\varepsilon:\mathcal{D}\to\mathcal{D}$ with action given by,
\begin{equation}
\mathcal{Z}^\varepsilon f(t,x)
:=
\int_{\mathcal{U}}
\exp\!\left(\dfrac{\mathcal{G}^uf(t,x)-c(t,x,u)}{\varepsilon}\right)\,r_{t,x}(du).
\end{equation}

Define the soft-Hamiltonian operator by:
\begin{equation}
\mathcal H^\varepsilon f(t,x)
:=
\varepsilon \log (\mathcal{Z}^\varepsilon f(t,x)).
\end{equation}

Equivalently, we can define the soft-Hamiltonian operator directly as,
\begin{equation}
\mathcal{H}^\varepsilon f(t,x):=-\inf_{q\in\mathcal P(\mathcal{U})}
\left\{
\int_{\mathcal{U}}\bigl(c(t,x,u)-\mathcal{G}^uf(t,x)\bigr)\,q(du)
+
\varepsilon\,D_\mathrm{KL}\bigl(q\,\|\,r_{t,x}\bigr)
\right\}.
\end{equation}

At each point $(t,x)$, the dual BAR multiplier induces a locally optimal randomized control law obtained by exponentially tilting the reference kernel. This is the controlled-process analog of the Gibbs variational structure in the Schr\"odinger bridge problem.

The Gibbs kernel solves the local control randomization problem. It does not by itself determine the evolution of the state marginal. To obtain a self-transport bridge, this local Gibbs kernel must be converted into a forward transition-rate kernel on the state space.

\ifmor
\else
\subsection{Forward Self-Transport Under the Gibbs Kernel}

For each dual BAR multiplier $f$, the previous subsection gives the pointwise Gibbs kernel
\begin{equation}
q_f^\varepsilon(du\mid t,x).
\end{equation}
This is the local probability law on the admissible controls $u$ at time $t$ and state $x$. It resolves the entropy-regularized minimization over the conditional control law. The remaining step is to convert this local randomized control into a forward law for the state process.

Define the averaged jump-rate kernel induced by the Gibbs kernel $q^\varepsilon_f$:
\begin{equation}
\bar q_f^\varepsilon(t,x,dy)
:=
\int_{\mathcal{U}} q(x,dy;u)\,q_f^\varepsilon(du\mid t,x).
\end{equation}
This is the conditional mean transition intensity under the randomized law $q_f^\varepsilon(\cdot\mid t,x)$. Since the Gibbs kernel is supported on the admissible relaxed control class inherited from $r_{t,x}$, the kernel $\bar q_f^\varepsilon(t,x,dy)$ belongs to the corresponding relaxed admissible hull at state $x$.

For each fixed $f$, define the time-inhomogeneous transition-rate kernel $K_f^\varepsilon(t;x,y)$ on $\mathcal{S}$ through the Radon--Nikodym derivative of the averaged jump-rate kernel with respect to some reference $\sigma$-finite measure $m(dy)$ on $\mathcal{S}$
\begin{equation}
K_f^\varepsilon(t;x,y)=\frac{d\bar q_f^\varepsilon(t,x,\cdot)}{dm}(y),
\qquad y\neq x,
\end{equation}
and
\begin{equation}
K_f^\varepsilon(t;x,x)
:=
-\int_{\mathcal{S}\setminus\{x\}}K^\varepsilon_f(t;x,y)m(dy),
\end{equation}
so that mass is conserved and we have:
\begin{equation*}
    \int_{\mathcal{S}}K^\varepsilon_f(t;x,y)m(dy) = 0.
\end{equation*}

\begin{remark}
    The theory presented in this section does not rely on the existence of the reference measure $m(dy)$; however, it is more convenient to state the algorithm in terms of densities rather than measures.
\end{remark}

The primitive transition rates are those of the controlled generator $\mathcal{G}^u$, while the averaged tilted jump-rate kernel gives the effective rates. The diagonal entry is defined so that each row of $K_f^\varepsilon$ sums to zero, so the kernel conserves total probability in the forward equation. All other transition rates are zero whenever they are excluded by the controlled jump kernel.

Given the common endpoint law $\tilde\pi$, the forward marginal law generated by $f$ is the family
\begin{equation}
\rho^{f,\varepsilon}
=
\{\rho_t^{f,\varepsilon}:0\leq t\leq 1\}
\subset \mathcal P(\mathcal{S}),
\end{equation}
where $\rho_t^{f,\varepsilon}$ is the probability measure of the controlled process. Let the density of these marginal measures with respect to the reference measure $m(dx)$ be $\rho_t^{f,\varepsilon}(x)$. This density solves the forward Kolmogorov equation:
\begin{equation}
\frac{d}{dt}\rho_t^{f,\varepsilon}(y)
=
\int_{\mathcal{S}\setminus \{y\}}
[\rho_t^{f,\varepsilon}(x)K_f^\varepsilon(t;x,y) - \rho_t^{f,\varepsilon}(y)K^\varepsilon_f(t;y,x)]m(dx),
\qquad y\in \mathcal{S},
\end{equation}
with initial condition
\begin{equation}
\rho_0^{f,\varepsilon}=\tilde\pi.
\end{equation}
Equivalently, the forward equation is:
\begin{equation*}
    \frac{d}{dt}\rho_t^{f,\varepsilon}(y)
=
\int_{\mathcal{S}}
\rho_t^{f,\varepsilon}(x)K_f^\varepsilon(t;x,y)m(dx),
\qquad y\in \mathcal{S}.
\end{equation*}

The first term records transition inflow. The second term records transition outflow. Hence, once the backward potential $f$ is fixed, the Gibbs kernel induces a full forward law on $\mathcal{S}$.

The entropy-regularized self-transport bridge is obtained by requiring the forward law generated by the Gibbs kernel to return to the same endpoint law:
\begin{equation}
\rho_1^{f,\varepsilon}=\tilde\pi.
\end{equation}
This is the global self-transport condition. The Gibbs kernel describes local control randomization, while $K_f^\varepsilon$ converts that local information into probability flow across the state space.

Therefore, the regularized BAR-SOT problem can be written as the coupled forward--backward system:
\begin{equation}\label{eq:soft hjb}
\partial_t f^\varepsilon(t,x)
+
\mathcal H^\varepsilon f^\varepsilon(t,x)
=
0,
\qquad (t,x)\in[0,1]\times \mathcal{S},
\end{equation}
\begin{equation}
\frac{d}{dt}\rho_t^{f^\varepsilon,\varepsilon}(y)
=
\int_{\mathcal{S}}
\rho_t^{f^\varepsilon,\varepsilon}(x)
K_{f^\varepsilon}^\varepsilon(t;x,y)m(dx),
\qquad y\in \mathcal{S},
\end{equation}
together with the two-point marginal constraint
\begin{equation}
\rho_0^{f^\varepsilon,\varepsilon}
=
\rho_1^{f^\varepsilon,\varepsilon}
=
\tilde\pi.
\end{equation}

The function $f^\varepsilon$ plays the role of a dual potential or value function. The backward equation determines the Gibbs kernel through the soft Hamiltonian. The forward equation transports probability under that kernel. The endpoint identity closes the problem by requiring that the resulting law returns to the same marginal $\tilde\pi$ at time $1$.

\subsection{Iterative Endpoint Scaling}

The forward--backward system suggests an endpoint-fitting procedure. The relaxed primal object is the measure
\begin{equation}
\pi(dt,x,du)=dt\,\rho_t(x)\,q_t(du\mid x),
\end{equation}
which records time, state, and randomized control. Instead of updating this full measure directly, the iteration updates the terminal dual potential. Once the terminal potential is fixed, the backward equation determines the time-dependent BAR multiplier, the Gibbs formula determines the Gibbs kernel, and the forward equation determines the resulting marginal law.

Let
\begin{equation}
\psi^{(0)}:\mathcal{S}\to\mathbb R
\end{equation}
be an initial terminal potential, and define the corresponding terminal factor by
\begin{equation}
\Psi^{(0)}(x):=\exp\!\left(\frac{\psi^{(0)}(x)}{\varepsilon}\right).
\end{equation}
This exponential change of variables puts the terminal potential in multiplicative form, which is the natural form for a Sinkhorn-type scaling step.

Given $\psi^{(m)}$, solve the backward soft Hamilton--Jacobi equation
\begin{equation}
\partial_t f^{(m)}(t,x)
+
\mathcal H^\varepsilon f^{(m)}(t,x)
=
0,
\qquad 0\leq t<1,
\end{equation}
with terminal condition
\begin{equation}
f^{(m)}(1,x)=\psi^{(m)}(x).
\end{equation}
This step propagates the current terminal guess backward through time and produces the time-dependent dual BAR multiplier $f^{(m)}$.

For notational convenience, define
\begin{equation}
\rho_t^{(m)}:=\rho_t^{f^{(m)},\varepsilon}.
\end{equation}
The forward generator at iteration $m$ is the same kernel from the previous subsection, evaluated at the current potential:
\begin{equation}
K_{f^{(m)}}^\varepsilon(t;x,y).
\end{equation}

Propagate the forward marginal law from the prescribed initial marginal:
\begin{equation}
\rho_0^{(m)}=\tilde\pi,
\end{equation}
\begin{equation}
\frac{d}{dt}\rho_t^{(m)}(y)
=
\int_{\mathcal{S}}
\rho_t^{(m)}(x)K_{f^{(m)}}^\varepsilon(t;x,y)m(dx),
\qquad y\in \mathcal{S}.
\end{equation}
This produces a terminal marginal $\rho_1^{(m)}$. The initial marginal $\rho_0^{(m)}=\tilde\pi$ is fixed in every cycle. What changes from one step to the next is the terminal marginal generated by the current tilted dynamics.

The endpoint mismatch is corrected by multiplicatively rescaling the terminal factor. On states where both $\tilde\pi(x)$ and $\rho_1^{(m)}(x)$ are positive, set
\begin{equation}
\Psi^{(m+1)}(x)
=
\Psi^{(m)}(x)
\left(
\frac{d\tilde\pi}{\rho_1^{(m)}}(x)
\right)^\eta,
\qquad 0<\eta\leq 1.
\end{equation}
When $\tilde\pi(x)=0$, set $\Psi^{(m+1)}(x):=0$, since such states carry no mass under the endpoint constraint $\rho_1=\tilde\pi$. The ratio $\tilde\pi(x)/\rho_1^{(m)}(x)$ increases the terminal factor at states where the current forward pass places too little mass and decreases it at states where the current forward pass places too much mass. The parameter $\eta$ is a relaxation parameter. When $\eta=1$, the full correction is applied. When $0<\eta<1$, the update is damped.

Since
\begin{equation}
\Psi^{(m)}(x)=\exp\!\left(\frac{\psi^{(m)}(x)}{\varepsilon}\right),
\end{equation}
the multiplicative update becomes an additive update for the terminal potential:
\begin{equation}
\psi^{(m+1)}(x)
=
\psi^{(m)}(x)
+
\eta\varepsilon
\left[
\log \tilde\pi(x)-\log \rho_1^{(m)}(x)
\right].
\end{equation}
Thus, the update adds a positive term when the current terminal law is too small at $x$ and a negative term when it is too large.

The iteration has the form
\begin{equation}
\psi^{(m)}
\longrightarrow
f^{(m)}
\longrightarrow
q_{f^{(m)}}^\varepsilon
\longrightarrow
\rho_1^{(m)}
\longrightarrow
\psi^{(m+1)}.
\end{equation}
The current terminal potential determines the backward dual potential. That potential determines the Gibbs kernel. The Gibbs kernel generates the forward law, and the endpoint mismatch updates the terminal potential for the next cycle.

This is the BAR-SOT analog of Sinkhorn scaling. The scaling is performed on the terminal dual factor that generates the Gibbs kernel. The backward equation converts the terminal factor into a time-dependent dual BAR multiplier, and the forward Kolmogorov equation tests whether the resulting tilted dynamics return the marginal law to $\tilde\pi$.

A natural stopping criterion is
\begin{equation}
\left\|\rho_1^{(m)}-\tilde\pi\right\|_1\leq\delta,
\end{equation}
for a prescribed tolerance $\delta>0$. This criterion measures how closely the terminal marginal produced by the current iterate matches the required endpoint law.

At convergence, the pair
\begin{equation}
\bigl(f^\varepsilon,\rho^{f^\varepsilon,\varepsilon}\bigr)
\end{equation}
satisfies the regularized self-transport bridge system, and the Gibbs kernel $q_{f^\varepsilon}^\varepsilon$ gives the associated entropy-regularized relaxed control law.

\subsection{Convergence of Endpoint Scaling}

Given that \(r_{t,x}(du)\) is the reference control kernel used in the Gibbs formula, let
\(R^{\varepsilon,r}\) denote the \emph{cost-tilted reference path law}: the law of the state--action process driven by the reference control \(r\), reweighted by the Gibbs cost factor \(e^{-\frac{1}{\varepsilon}\int_0^1 c\,dt}\) and normalized. It is the continuous-time counterpart of the cost-tilted reference \(R^c\) of \Cref{sec:finite-mdp}. The two ways of writing the regularized objective agree on the admissible class: for the path law \(P^q\) generated by an admissible control kernel \(q\), the transition law given the action is common to \(P^q\) and \(R^{\varepsilon,r}\), those terms cancel in the path relative entropy, and
\(D_\mathrm{KL}(P^q\,\|\,R^{\varepsilon,r})\) equals the accumulated control penalty \(\int_0^1\int_{\mathcal{S}}D_\mathrm{KL}\big(q_t(\cdot\mid x)\,\|\,r_{t,x}\big)\rho_t(dx)\,dt\) plus the normalized cost term. Thus the KL penalty is applied only to the control, while the path law \(R^{\varepsilon,r}\) is the object on which endpoint operations below are performed. Let $e_t(\omega) = X_t(\omega)$ be the evaluation map at time $t$ and define
\[
K^{\varepsilon,r} := (e_0,e_1)_\#R^{\varepsilon,r}\in\mathcal{P}(\mathcal{S}\times\mathcal{S}),
\]
as the endpoint kernel. Equivalently, for all $C\in\mathcal{B}(\mathcal{S}\times\mathcal{S})$,
\[
    K^{\varepsilon,r}(C) := R^{\varepsilon,r}((X_0,X_1)\in C).
\]
Define the set $\mathbf{S}_{\tilde\pi} = \mathrm{supp}(\tilde\pi)\subseteq \mathcal{S}$. Let
\[
\mathcal{B}(\mathbf{S}_{\tilde\pi}) := \{A\cap\mathbf{S}_{\tilde\pi}: A\in\mathcal{B}(\mathcal{S})\}
\]
be the Borel $\sigma$-algebra. Assume that 
\[
K^{\varepsilon,r}(A,B) >0, \text{ for every } A,B\in\mathcal{B}(\mathbf{S}_{\tilde\pi})\text{ with }\tilde\pi(A)\tilde\pi(B)>0.
\]
For measures $\sigma,\rho \in\mathcal{P}(\mathbf{S}_{\tilde\pi})$, define the set
\[
\Pi(\sigma,\rho) := \{\gamma\in\mathcal{P}(\mathbf{S}_{\tilde\pi}\times \mathbf{S}_{\tilde\pi}): (\mathrm{pr}_0)_{\#}\gamma = \sigma,(\mathrm{pr}_1)_{\#}\gamma = \rho\},
\]
and assume that the set
\[
\{\gamma\in\Pi(\tilde\pi,\tilde\pi):D_\mathrm{KL}(\gamma\|K^{\varepsilon,r})<\infty\} \neq\emptyset.
\]
For $\gamma\in\mathcal{P}(\mathbf{S}_{\tilde\pi}\times\mathbf{S}_{\tilde\pi})$, write $\gamma_0:=(\mathrm{pr}_0)_{\#}\gamma$ and $\gamma_1:=(\mathrm{pr}_1)_{\#}\gamma$ for its two marginals.

The endpoint Schr\"odinger problem is:
\begin{equation}
\label{eq:static-endpoint-schrodinger-short}
\gamma^\varepsilon
\in
\arg\min_{\gamma\in\Pi(\tilde\pi,\tilde\pi)}
D_\mathrm{KL}\bigl(\gamma\,\|\,K^{\varepsilon,r}\bigr).
\end{equation}

We call each of the transition types that make up the jump kernel~\eqref{eq:generator} a \emph{jump channel}; for the queue of \Cref{ex:queue} the channels are the arrivals $x\to x+e_i$ and the service completions $x\to x-e_i$. The \emph{fully controlled} (free-rate) case is the one in which the control sets the rate of every channel, so that the Gibbs kernel factors across channels; in scheduling only the service channels are controlled. Once the dynamic problem reduces to \eqref{eq:static-endpoint-schrodinger-short}, the convergence of the endpoint-scaling iteration is classical.
\begin{proposition}[Convergence of endpoint scaling, fully controlled case]\label{prop:ipfp-convergence}
Suppose the endpoint-conditioned tilted reference remains admissible, as in the fully controlled case defined above; the regularized inner problem at a fixed $\tilde\pi$ then coincides with \eqref{eq:static-endpoint-schrodinger-short}. Under the positivity and finite-entropy hypotheses above, the endpoint-scaling iteration with $\eta=1$ is the iterative proportional fitting procedure for \eqref{eq:static-endpoint-schrodinger-short}, and its iterates converge to the unique minimizer $\gamma^\varepsilon$ \citep{csiszar1975,ruschendorf1993}. If moreover $K^{\varepsilon,r}$ {has a density with respect to $\tilde\pi\otimes\tilde\pi$ and the density is} bounded above and bounded below away from zero on ${\mathbf S_{\tilde\pi}\times\mathbf S_{\tilde\pi}}$ (as when $\mathbf S_{\tilde\pi}$ is finite), the convergence is linear in the Hilbert projective metric with a contraction factor $\theta<1$ depending only on those bounds \citep{franklinlorenz1989}, and the damped update $0<\eta<1$ contracts with factor $1-\eta(1-\theta)$.
\end{proposition}
\begin{proof}
Conditioning on the endpoints decomposes the relative entropy as
\[
D_\mathrm{KL}\big(P\,\|\,R^{\varepsilon,r}\big)=D_\mathrm{KL}\big(\gamma\,\|\,K^{\varepsilon,r}\big)+\mathbb{E}_\gamma\Big[D_\mathrm{KL}\big(P(\cdot\mid X_0,X_1)\,\|\,R^{\varepsilon,r}(\cdot\mid X_0,X_1)\big)\Big],
\]
and under the stated admissibility the bridge term vanishes at the optimum within the admissible class, reducing the dynamic problem to \eqref{eq:static-endpoint-schrodinger-short}. The two steps of the iteration are then alternating I-projections onto the marginal constraints $\{\gamma:\gamma_0=\tilde\pi\}$ and $\{\gamma:\gamma_1=\tilde\pi\}$, where the I-projection onto a convex set of measures is the minimizer over that set of the relative entropy to $K^{\varepsilon,r}$. The forward pass re-imposes the initial marginal, and the terminal rescaling imposes the final one. Convergence of alternating I-projections under the two hypotheses is \citet{csiszar1975} (see also \citealt{ruschendorf1993}), and the Hilbert-metric contraction under two-sided kernel bounds is \citet{franklinlorenz1989}.
{For the damped update, write $\phi=\log\Psi$. The damped update is the convex combination $\phi^{(m+1)}=(1-\eta)\,\phi^{(m)}+\eta\,\phi^{(m)}_{\mathrm{IPF}}$, where $\phi^{(m)}_{\mathrm{IPF}}$ is the $\eta=1$ step, and the Hilbert projective metric $d(\Psi,\Psi^\star)$ is the oscillation seminorm of $\log\Psi-\log \Psi^\star$. By the triangle inequality and homogeneity of a seminorm, and the contraction $d(\Psi^{(m)}_{\mathrm{IPF}},\Psi^\star)\le\theta\,d(\Psi^{(m)},\Psi^\star)$ with $\Psi^\star$ the fixed point, $d(\Psi^{(m+1)},\Psi^\star)\le(1-\eta)\,d(\Psi^{(m)},\Psi^\star)+\eta\theta\,d(\Psi^{(m)},\Psi^\star)=\big(1-\eta(1-\theta)\big)\,d(\Psi^{(m)},\Psi^\star)$.} 
\end{proof}
\begin{remark}[Scope]\label{rem:ipfp-scope}
The admissibility hypothesis of \Cref{prop:ipfp-convergence} is a genuine restriction. In the decomposition displayed in the proof, the bridge term vanishes only at the cost-tilted reference path law $R^{\varepsilon,r}$ conditioned on its endpoints, which is a Doob $h$-transform of $R^{\varepsilon,r}$: every jump rate $r(x,y)$ is multiplied by the ratio $h(t,y)/h(t,x)$. For the scheduling problem, this modifies the arrival rates as well as the service rates, whereas an admissible control reallocates service while leaving the arrival rates fixed. The conditioned law is therefore not attainable in general; the bridge term cannot be driven to zero, and the value of the dynamic problem may strictly exceed that of \eqref{eq:static-endpoint-schrodinger-short}. In this partially controlled case we do not claim the classical guarantee. The computations of \Cref{sec:numerical} do not rely on it, as they solve the periodic dual of \Cref{sec:periodic-bvp} directly by residual collocation.
\end{remark}

\fi

\subsection{The Regularized Periodic BVP: Existence and Uniqueness of Solutions}\label{sec:soft-periodic-bvp}
The soft Hamiltonian $\mathcal H^\varepsilon f=\varepsilon\log\mathcal Z^\varepsilon f$ is a log-partition function, hence smooth in $f$, and its derivative is the Gibbs kernel. In the finite-action setting, with $z_u:=\mathcal G^uf(t,x)-c(t,x,u)$, one has $\partial\mathcal H^\varepsilon f/\partial z_u=q^\varepsilon_f(u\mid t,x)$, the softmax of~\eqref{eq:gibbs-kernel}, and the same holds in the sense of a Gateaux derivative for a general action set. The randomized control thus replaces the bang-bang maximizer $u^f$ as the gradient of the Hamiltonian, and it varies continuously across the switching surface instead of jumping. As $\varepsilon\to0$, $\mathcal H^\varepsilon f\to\mathcal Hf$ pointwise (\Cref{sec:consistency}), so $\varepsilon$ is a temperature that trades smoothness against fidelity.

We therefore return to the periodic BVP~\eqref{eq:periodic-bvp-hard} of \Cref{sec:periodic-bvp} with $\mathcal H$ replaced by $\mathcal H^\varepsilon$. The regularized dual $f^\varepsilon$ solves the backward soft Hamilton--Jacobi equation with the periodicity constraint,
\begin{equation}\label{eq:periodic-bvp}
    \partial_t f_t(x) + \mathcal H^\varepsilon{f_t}(x) = 0 \quad (0 < t < T),
    \qquad f_T - f_0 = \text{const in } x.
\end{equation}
{\Cref{thm:soft-periodic} below shows that~\eqref{eq:periodic-bvp} has a solution when the reference kernel does not depend on time.} Note that $\mathcal H^\varepsilon f$ already carries the running cost $c$ and the arrival drift inside the tilted log-partition $\varepsilon\log\mathcal{Z}^\varepsilon f$. The optimal control is then the Gibbs kernel $q^\varepsilon_{f^\varepsilon}$ of \Cref{sec:gibbs-kernel}, and the endpoint law is the stationary law $\tilde\pi^*_\varepsilon$ of the chain that this kernel controls. The entropy regularization enters~\eqref{eq:periodic-bvp} only through $\mathcal H^\varepsilon f$. The periodicity boundary condition is inherited from the unregularized dual, and therefore~\eqref{eq:periodic-dual} holds at every $\varepsilon>0$ with $g^*$ replaced by the regularized rate $g^*_\varepsilon$.
Difficulty (a) of \Cref{sec:optimality-system} is handled exactly as before, and read off the endpoint gap $(f^\varepsilon_T-f^\varepsilon_0)$.

For a time-independent reference kernel $r_x$, the following theorem makes precise the statements that the optimal control is the Gibbs kernel and that the endpoint law is the stationary law of its chain. In the theorem, $c^q_\varepsilon:=c^q+\varepsilon D_{\mathrm{KL}}(q\,\|\,r)$ denotes the regularized running cost of a control kernel $q$ and $q^\circ$ is the stationary control kernel of \Cref{assump:standing}.
\begin{theorem}[The regularized periodic BVP]\label{thm:soft-periodic}
Let \Cref{assump:standing} hold, let $\varepsilon>0$, and let the reference kernel $r_x$ not depend on time and satisfy $K_r:=\sup_xD_{\mathrm{KL}}(q^\circ(\cdot\mid x)\,\|\,r_x)<\infty$.
\begin{enumerate}
\item[(i)] There exists a solution $(h_\varepsilon,g^*_\varepsilon)$ of $g^*_\varepsilon+\mathcal H^\varepsilon h_\varepsilon=0$ with $\|h_\varepsilon\|_V<\infty$ and $-h_\varepsilon$ bounded below, so that $f^\varepsilon:=h_\varepsilon+(t-T)g^*_\varepsilon$ solves~\eqref{eq:periodic-bvp}. The Gibbs kernel $q^*_\varepsilon:=q^\varepsilon_{h_\varepsilon}$ defines an irreducible, positive recurrent chain with invariant law $\tilde\pi^*_\varepsilon$, and $g^*_\varepsilon=\langle c^{q^*_\varepsilon}_\varepsilon,\tilde\pi^*_\varepsilon\rangle$ is the optimal regularized long-run average cost over all policies, from every initial state.
\item[(ii)] Every solution $f$ of~\eqref{eq:periodic-bvp} with $f(\cdot,x)\in C^1$, $\sup_t(\|f(t)\|_V+\|\partial_tf(t)\|_V)<\infty$ and $-f$ bounded below has endpoint gap at least $Tg^*_\varepsilon$.
\item[(iii)] If, in addition, $\langle V,\tilde\pi^*_\varepsilon\rangle<\infty$, then every $f$ as in (ii) has the form $f(t,x)=h_\varepsilon(x)+(t-T)g^*_\varepsilon+k$ for a constant $k$. In particular its endpoint gap is $Tg^*_\varepsilon$, and its Gibbs kernel is $q^*_\varepsilon$, whose stationary law is $\tilde\pi^*_\varepsilon$.
\end{enumerate}
\end{theorem}
The proof of \Cref{thm:soft-periodic} in \Cref{ec:proofs}, differs from that of \Cref{thm:periodic-wellposed}, since it does not use \Cref{assump:uniform}, which fails whenever some non-idling policy is unstable. Instead, we exploit the Gibbs kernel and a discounting argument. The Gibbs kernel $q^*_\varepsilon$ charges every action that the reference kernel charges, and hence, since $K_r<\infty$, every action that $q^\circ$ charges, so that its chain inherits the irreducibility of the chain of $q^\circ$ and $\tilde\pi^*_\varepsilon$ charges every state. Moreover, $g^*_\varepsilon$ is the limit, along a sequence $\beta_n\to0$, of $\beta_n$ times the optimal $\beta_n$-discounted regularized cost, and therefore it bounds from below the regularized long-run average cost of every policy, including the time-periodic Gibbs policy of an arbitrary solution of~\eqref{eq:periodic-bvp}. For this reason, part (ii) holds for every solution in the class described there, with no stability hypothesis on its Gibbs policy, and part (iii) requires a moment condition on the single chain of $q^*_\varepsilon$ rather than on every policy.

Because the residual $\partial_tf+\mathcal H^\varepsilon f$ is now a smooth functional of $f$, \eqref{eq:periodic-bvp} can be attacked by gradient descent over any differentiable parametrization of the dual. Here we minimize the squared residual over collocation points in $(t,x)$ subject to the periodicity constraint, and read $g^*_\varepsilon$ from the converged endpoint gap. This is the scheme we implement in \Cref{sec:numerical} under a neural parametrization of $f$. Since $f^\varepsilon$ is soft-feasible, the consistency theorems of the next subsection give $g^*_\varepsilon-g^*\le\mathcal{O}(\varepsilon\log\tfrac1\varepsilon)$; the matching lower bound $g^*_\varepsilon\ge g^*$ holds whenever $f^\varepsilon$ attains the supremum over the soft-feasible class.

\subsection{Consistency of the Regularized Value}\label{sec:consistency}
In this subsection we show that the value of the entropy-regularized dual converges to the unregularized value $g^*T$ as $\varepsilon\to0$, and we quantify the rate. The convergence is at the level of the $\tilde\pi^*$-weighted endpoint gap; we do not claim convergence of the maximizers themselves. The pointwise limit of the soft Hamiltonian is the starting point. From the Laplace--Varadhan theorem, we have:
\begin{equation*}
    \lim_{\varepsilon\to0}\mathcal{H}^\varepsilon f(t,x) = \mathcal{H}f(t,x),
\end{equation*}
for each fixed $f$ and $(t,x)$, the supremum in $\mathcal{H}f$ taken over the support of $r_{t,x}$. The pointwise limit makes the result unsurprising. It is still not obvious that the optimal value over the soft-feasible class defined below converges to $g^*T$, the value attained by the solution of the HJB \eqref{eq:HJB}. We require the following assumption on the action set $\mathcal{U}$, the reference measure $r_{t,x}$, and the cost function $c(t,x,u)$.
\begin{assumption}\label{assump:smooth-density}
    $\mathcal{U}\subset\mathbb{R}^k$ is compact with nonempty interior. The reference measure $r_{t,x}$ admits a continuous density $\rho_{t,x}$ with respect to Lebesgue measure on $\mathcal{U}$, bounded above and bounded below away from zero{, and uniformly equicontinuous in $u$,} uniformly in $(t,x)$. The map $u\mapsto c(t,x,u)$ is $C^3$ on $\mathcal{U}$ with $mI\preceq\partial_u^2c(t,x,u)\preceq MI$ for constants $0<m\le M<\infty$ uniform in $(t,x)$, and $u\mapsto\mathcal{G}^uf(t,x)$ is affine (\Cref{assump:affine}).
\end{assumption}
The assumption is tailored to continuous action sets with interior optimizers, as in rate control. The scheduling example, whose maximizer sits at a vertex, is covered instead by the finite-action estimate of \Cref{thm:finite-action-consistency}.

We now state the estimate:
\begin{theorem}[Continuous-action consistency]\label{thm:cont-action-consistency}
    Let the hypotheses of \Cref{thm:ansatz}, \Cref{lem:weak-duality}, and \Cref{assump:smooth-density} hold, and let $\tilde\pi^*$ be the invariant law of $u^{**}$. Define the soft-feasible class
    \begin{equation}\label{eq:soft-class}
    \begin{aligned}
        \mathcal{F}^\text{soft} := \Big\{&f:[0,T]\times\mathcal{S}\to\mathbb{R}\ :\
        f(\cdot,x)\in C^1, \\
        &\sup_{t}\|f(t,\cdot)\|_V
        +\sup_{t}\|\partial_tf(t,\cdot)\|_V<\infty, \\
        &\partial_tf + \mathcal{H}^\varepsilon f\le 0\Big\},
    \end{aligned}
    \end{equation}
    and suppose $f^\varepsilon$ attains the supremum of $\langle f_T-f_0,\tilde{\pi}^*\rangle$ over $\mathcal{F}^\text{soft}$. Suppose further that there is $d_0>0$ such that, for $f\in\{f^*,f^\varepsilon\}$ and every $(t,x)$, the maximizer of $u\mapsto\mathcal{G}^uf(t,x)-c(t,x,u)$ lies at distance at least $d_0$ from the boundary of $\mathcal{U}$. Here $f^*=h+(t-T)g^*$ is the ansatz of \Cref{thm:ansatz}. Then, for fixed $T>0$,
    \begin{equation}
        0\ \le\ \langle f^\varepsilon_T-f^\varepsilon_0,\tilde{\pi}^*\rangle - g^*T\ =\ \mathcal{O}\Big(T\varepsilon\log\tfrac{1}{\varepsilon}\Big),
    \end{equation}
    as $\varepsilon\to0$.
\end{theorem}

\ifmor
\else
\begin{proof}

Throughout, $f^*=h+(t-T)g^*$ is the ansatz of \Cref{thm:ansatz}, so $\langle f^*_T-f^*_0,\tilde\pi^*\rangle=g^*T$ and $(\mathcal{G}^{u^{**}})^*\tilde\pi^*=0$; pairings of weighted functions against $\tilde\pi^*$ extend from bounded functions by the truncation argument of \Cref{lem:weak-duality}, using $\langle V,\tilde\pi^*\rangle<\infty$.

\emph{Step 1 (two-sided Laplace estimate).} Fix $(t,x)$ and $f\in\{f^*,f^\varepsilon\}$, and write $\Phi(u):=\mathcal{G}^uf(t,x)-c(t,x,u)$. Since $r_{t,x}$ is a probability measure, $\int_{\mathcal{U}} e^{\Phi/\varepsilon}\,dr_{t,x}\le e^{\sup_{\mathcal{U}}\Phi/\varepsilon}$, whence $\mathcal{H}^\varepsilon f\le\mathcal{H}f$ pointwise. Under \Cref{assump:smooth-density}, $\Phi$ is $C^3$ and strongly concave, with $-MI\preceq\partial_u^2\Phi\preceq-mI$, so its maximizer $u^*$ is unique and interior, and Laplace's method \citep[Ch.~IX]{wong2001asymptotic} gives
\[
  \mathcal{H}^\varepsilon f(t,x) = \mathcal{H}f(t,x) + \tfrac{\varepsilon k}{2}\log(2\pi\varepsilon)
  - \tfrac{\varepsilon}{2}\log\det\big(-\partial_u^2\Phi(u^*)\big)
  + \varepsilon\log\rho_{t,x}(u^*) + {o(\varepsilon)}.
\]
The uniform bounds on $\partial_u^2c$, on $\rho_{t,x}$, and on the distance of $u^*$ from the boundary{, together with the equicontinuity of $\rho_{t,x}$,} make the error terms uniform in $(t,x)$. Hence
\[
  0\ \le\ \mathcal{H}f(t,x)-\mathcal{H}^\varepsilon f(t,x)\ =\ \tfrac{\varepsilon k}{2}\log\tfrac{1}{2\pi\varepsilon}+\mathcal{O}(\varepsilon)\ =:\ \delta_\varepsilon(t,x)\ =\ \mathcal{O}\big(\varepsilon\log\tfrac1\varepsilon\big),
\]
uniformly in $(t,x)$.

\emph{Step 2 (upper bound).} Set $w^\varepsilon:=f^\varepsilon-f^*$. Soft-feasibility of $f^\varepsilon$ and the HJB for $f^*$ give
\[
  \partial_tw^\varepsilon + \mathcal{H}^\varepsilon f^\varepsilon-\mathcal{H}f^*\ \le\ 0 .
\]
Observe that $\mathcal{H}f^\varepsilon\ge\mathcal{G}^{u^{**}}f^\varepsilon-c(\cdot,u^{**})$, while $\mathcal{H}f^*=\mathcal{G}^{u^{**}}f^*-c(\cdot,u^{**})$ by the definition of $u^{**}$, so $\mathcal{H}f^\varepsilon-\mathcal{H}f^*\ge\mathcal{G}^{u^{**}}w^\varepsilon$. Adding and subtracting $\mathcal{H}f^\varepsilon$ and applying Step 1,
\[
  \partial_tw^\varepsilon + \mathcal{G}^{u^{**}}w^\varepsilon\ \le\ \mathcal{H}f^\varepsilon-\mathcal{H}^\varepsilon f^\varepsilon\ =\ \delta_\varepsilon .
\]
Pairing against the constant flow $\rho_t\equiv\tilde\pi^*$, admissible by the invariance of $\tilde\pi^*$, and integrating over $[0,T]$,
\[
  \langle w^\varepsilon_T-w^\varepsilon_0,\tilde\pi^*\rangle\ \le\ T\,\sup_{t,x}\delta_\varepsilon\ =\ \mathcal{O}\big(T\varepsilon\log\tfrac1\varepsilon\big),
\]
and since $\langle f^*_T-f^*_0,\tilde\pi^*\rangle=g^*T$, this reads $\langle f^\varepsilon_T-f^\varepsilon_0,\tilde\pi^*\rangle-g^*T\le\mathcal{O}(T\varepsilon\log\tfrac1\varepsilon)$.

\emph{Step 3 (lower bound).} By Step 1, $\partial_tf^*+\mathcal{H}^\varepsilon f^*=\mathcal{H}^\varepsilon f^*-\mathcal{H}f^*\le0$, so $f^*\in\mathcal{F}^\text{soft}$. Since $f^\varepsilon$ attains the supremum over $\mathcal{F}^\text{soft}$,
\[
  \langle f^\varepsilon_T-f^\varepsilon_0,\tilde\pi^*\rangle\ \ge\ \langle f^*_T-f^*_0,\tilde\pi^*\rangle\ =\ g^*T .
\]
Steps 2 and 3 sandwich the regularized value, proving the claim. In particular, entropy regularization \emph{over}-estimates the value, consistent with $\mathcal{H}^\varepsilon\le\mathcal{H}$ enlarging the soft-feasible set. 

\end{proof}
\fi

The next theorem shows that the gap between the optimal values of the HJB and the soft HJB is $\mathcal{O}(\varepsilon)$ as $\varepsilon\to0$ when both the set of admissible controls and the state space are discrete, in the setting of \Cref{ex:queue}.

\begin{theorem}[Finite-action consistency]\label{thm:finite-action-consistency}
    Consider the single-server $k$-class scheduling problem on $\mathcal S=\mathbb Z^k_+$: finite action set $\mathcal{U}= \{\mu_ie_i\}_{i=1}^k$, state-only cost $c(x)$, and the reference kernel $r_{t,x}\in\mathcal{P}(\mathcal{U})$ uniform over the feasible actions. Let $F(x):=\{i:x_i>0\}$ be the set of nonempty classes and $n(x):=|F(x)|$; the reference places mass $1/n(x)$ on each $\mu_ie_i$ with $i\in F(x)$, and at $x=0$ the server idles. The soft Hamiltonian acts on the set $\mathcal{F}^\text{soft}$ defined in \Cref{thm:cont-action-consistency}. With the control-independent cost terms omitted, the soft and hard Hamiltonians are, for $x\neq0$,
    \[
    \mathcal{H}^\varepsilon f(t,x) = \varepsilon\log\left(\frac{1}{n(x)}\sum_{i\in F(x)} e^{\mu_if_-^i(t,x)/\varepsilon}\right), \qquad \mathcal{H}f(t,x) = \max_{i\in F(x)}\mu_if_-^i(t,x),
    \]
    and $\mathcal{H}^\varepsilon f(t,0)=\mathcal{H}f(t,0)=0$; the restriction to $F(x)$ is the boundary indicator of the queue generator~\eqref{eq:queue-generator}, so $f$ is never evaluated outside $\mathbb Z^k_+$. Let the hypotheses of \Cref{thm:ansatz} and \Cref{lem:weak-duality} hold, the hypotheses of \Cref{thm:ansatz} taken under the convex-hull convention of \Cref{ex:queue}, with the stability condition $\rho=\sum_i\lambda_i/\mu_i<1$, and let $\tilde\pi^*$ be the invariant law of $u^{**}$. Suppose $f^\varepsilon$ attains the supremum of $\langle f_T-f_0,\tilde\pi^*\rangle$ over the soft-feasible class~\eqref{eq:soft-class}, with $\mathcal{H}^\varepsilon$ in the finite-action form above. Then, for every $\varepsilon>0$ and $T>0$,
\begin{equation}\label{eq:finite-action-bound}
  0\;\le\;\langle f^\varepsilon_T-f^\varepsilon_0,\tilde\pi^{*}\rangle - g^*T
   \;\le\; T\,\varepsilon\,\mathbb E_{\tilde\pi^*}\!\big[\log n(X)\big]\;\le\;T\,\varepsilon\log k,
\end{equation}
with the convention $\log n(0):=0$, so that the Hamiltonian gap $\mathcal{H}f-\mathcal{H}^\varepsilon f$ is zero at $x=0$.
\end{theorem}

\ifmor
\else
\begin{proof}
    Fix $x\neq0$ and write $F=F(x)$, $n=n(x)$; write $z_i=f^i_-(t,x)$ for $i\in F$, and the Hamiltonian operators acting on the functions as
    \[
    H^\varepsilon(x;z):=\mathcal{H}^\varepsilon f(t,x), \qquad H(x;z) := \mathcal{H}f(t,x),
    \]
    and set $M:= H = \max_{i\in F} \mu_iz_i$. At $x=0$ both Hamiltonians vanish and the gap below is zero, consistent with the convention $\log n(0)=0$.

    \noindent\emph{(i) $H^\varepsilon\le H$.} As $r_{t,x}$ is a probability measure and $e^{\mu_i z_i/\varepsilon}\le e^{M/\varepsilon}$, we have $H^\varepsilon=\varepsilon\log(\tfrac1n\sum_{i\in F}e^{\mu_i z_i/\varepsilon})\le\varepsilon\log(e^{M/\varepsilon})=M=H$.

    \noindent\emph{(ii) The gap.} Since $1\le\sum_{i\in F}e^{(\mu_i z_i-M)/\varepsilon}\le n$ (the largest term is $e^0=1$; there are $n$ terms, each $\le1$),
    \[
    0\;\le\;H(x;z)-H^\varepsilon(x;z)
    =\varepsilon\log n-\varepsilon\log\!\Big(\!\sum_{i\in F}e^{(\mu_i z_i-M)/\varepsilon}\Big)
    \;\le\;\varepsilon\log n(x)\;\le\;\varepsilon\log k,
    \]
    \emph{exactly}, for every $\varepsilon$. Call $x$ a \emph{switching state} if the maximum $M$ is attained by two or more indices of $F(x)$, and call the set of switching states the switching set; for $k=2$ it is the tie set $\Sigma$ of \Cref{ex:singular}. Off the switching set, let the \emph{runner-up margin} be $\delta(x):=M-\max\{\mu_iz_i:\ i\in F(x),\ \mu_iz_i<M\}>0$, with $\delta(x):=\infty$ when $n(x)=1$. Off the switching set the subtracted term is $\varepsilon\log\!\big(1+\mathcal O(e^{-\delta/\varepsilon})\big)$, so $H-H^\varepsilon=\varepsilon\log n(x)+\mathcal O(\varepsilon e^{-\delta/\varepsilon})$.

    \noindent\emph{(iii) Sandwich.} By (i), $H^\varepsilon\le H$, so the hard-optimal $f^*$ (the ansatz~\eqref{eq:ansatz}) is a soft \emph{sub}-solution, hence belongs to the soft-feasible class~\eqref{eq:soft-class}; thus $\langle f^\varepsilon_T-f^\varepsilon_0,\tilde\pi^*\rangle\ge\langle f^*_T-f^*_0,\tilde\pi^*\rangle=g^*T$. For the upper bound, set $w^\varepsilon=f^\varepsilon-f^*$, subtract the two HJBs, bound $H(f^\varepsilon_-)-H(f^*_-)\ge\langle u^{**},w^\varepsilon_-\rangle$ by convexity, and pair against the invariant flow $\rho_t\equiv\tilde\pi^*$ (weak BAR; the pairing extends to the weighted class by the truncation argument of \Cref{lem:weak-duality}, using $\langle V,\tilde\pi^*\rangle<\infty$) to obtain
    \[
    \langle f^\varepsilon_T-f^\varepsilon_0,\tilde\pi^{*}\rangle-g^*T
    \;\le\;\int_0^T\!\sum_x\big(H-H^\varepsilon\big)(x;f^\varepsilon_-)\,\tilde\pi^*(x)\,dt
    \;\le\;T\,\varepsilon\,\mathbb E_{\tilde\pi^*}\!\big[\log n(X)\big]\;\le\;T\,\varepsilon\log k,
    \]
    where the middle inequality follows from (ii). 
\end{proof}
\fi

\begin{remark}[Why $\mathcal O(\varepsilon)$, not $\mathcal O(\varepsilon\log\frac1\varepsilon)$]
     The continuous consistency theorem expands a log-partition around an \emph{interior, nondegenerate} maximizer, whose Laplace correction is $\frac{\varepsilon k}{2}\log(2\pi\varepsilon)$. The single-server problem has a \emph{finite} action set and a \emph{linear} Hamiltonian, so the maximizer sits at a vertex and the ``Laplace integral'' is a discrete sum. Observe from step (ii) of the proof of \Cref{thm:finite-action-consistency} that off the switching set the gap is $\varepsilon\log n(x)+\mathcal O(\varepsilon e^{-\delta/\varepsilon})$: the $\mathcal O(\varepsilon)$ rate is the reference normalization $1/n(x)$. \Cref{fig:consistency-rate} confirms $\mathbb E_{\tilde\pi}[H-H^\varepsilon]=\mathcal O(\varepsilon)$ over three orders of $\varepsilon$, with empirical slope $\mathbb E_{\tilde\pi}[\log n]$ lying well below $\log k$.
\end{remark} 

\ifmor
\else
\begin{figure}[htb]
\centering
\includegraphics[width=0.6\linewidth]{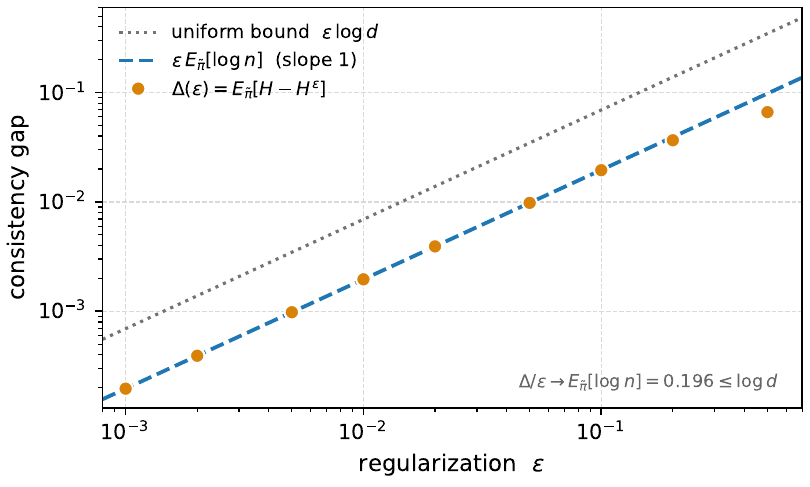}
\caption{{Finite-action consistency rate (single-server $c\mu$ regime, $\lambda=(0.5,0.5)$, $\mu=(1,1.5)$, $c=(1,1)$). The $\tilde\pi$-weighted Hamiltonian gap $\Delta(\varepsilon)=\mathbb E_{\tilde\pi}[H-H^\varepsilon]$ --- which bounds the consistency error in~\eqref{eq:finite-action-bound} --- is $\mathcal O(\varepsilon)$ with slope $\mathbb E_{\tilde\pi}[\log n]=0.196$, well under the uniform bound $\varepsilon\log k$, and with no $\log\frac1\varepsilon$ correction.}}
\label{fig:consistency-rate}
\end{figure}
\fi

 In \Cref{sec:numerical} next, we operationalize the periodic-BVP approach by solving it numerically using collocation \citep[see][]{betts2010practical,raissi2019pinn} of the soft-HJB residual under a neural parametrization of the dual, on a multiclass reentrant network scheduling problem. In Section~\ref{sec:switch}, we present an application to scheduling input-queued switches.

\section{Numerical Examples}\label{sec:numerical}

We demonstrate the framework on two scheduling problems under complementary scaling conditions. The first is a six-class reentrant network of \citet{daigluzman2022}, a standard heavy-traffic benchmark. As the system approaches full load, the part of the dual function that separates admissible actions becomes a vanishing fraction of its magnitude, and a neural network architecture that exploits Harrison's equivalent-workload representation makes the dual resolvable. The second is an input-queued switch, where the obstruction is the combinatorial size of the control space. The state is the full matrix of queue lengths and the control is a matching drawn from a super-exponentially sized set. We address this scale issue with a graph-attention parametrization of the dual function on the bipartite structure of the switch. The switch example is reported in \Cref{sec:switch}. In both problems, we solve the periodic dual of \Cref{sec:periodic-bvp}, take its entropy-regularized Gibbs control, and evaluate that control by long-run average holding cost in a single validated simulator against the strongest available baselines for either problem.

\subsection{Heavy Traffic and the Equivalent-Workload Architecture}\label{sec:heavy-traffic}

The periodic BVP~\eqref{eq:periodic-bvp} reduces long-run average-cost control to computing the dual function $f$, the optimal control (in the entropy-regularized setting) being the Gibbs kernel $q^\varepsilon_{f}$. This section stress-tests that program in the heavy-traffic regime, where it is both most useful and hardest, on a standard multiclass reentrant benchmark. Of particular interest is the neural parametrization of $f$ that makes it succeed. Specifically, we use a direct encoding of Harrison's \emph{equivalent workload formulation}~\citep{harrisonvanmieghem1997,harrison2000} into the neural architecture, and this inductive bias makes the inferred policy competitive with the state-of-the-art policies obtained using reinforcement learning~\citep{daigluzman2022}; once the sub-dominant transverse correction is resolved (\Cref{sec:deep-ht}), the dual improves on them.

\subsubsection{Benchmark and the heavy-traffic obstruction}\label{sec:ht-obstruction}

We use the six-class, two-station reentrant network of \citet{daigluzman2022} (\Cref{fig:network}), identical to the \texttt{reentrant} instance of the QGym benchmark suite~\citep{qgym2024}. Classes $\{1,2,3\}$ are served by station~$1$ and $\{4,5,6\}$ by station~$2$, with service rates $\mu=(\tfrac18,\tfrac12,\tfrac14,\tfrac16,\tfrac17,1)$, exogenous Poisson arrivals to classes~$1$ and~$3$ at rate $9/140$, and deterministic routing $1\!\to\!4\!\to\!2\!\to\!5\!\to\text{out}$ and $3\!\to\!6\!\to\text{out}$. The holding cost is the number in system, $c(x)=\sum_k x_k$. Both stations then have traffic intensity $\rho=0.9$.

This network satisfies the hypotheses of \Cref{thm:soft-periodic}\,(i)--(ii) with the weight $V(x)=|x|^2$, where $|x|:=\sum_kx_k$. The cost is time-independent, lower semicontinuous and coercive, and the total jump intensity is at most $\sum_k\mu_k+2\cdot\tfrac{9}{140}$. Since every transition moves one job into, out of, or between buffers, $V$ increases by at most $2|x|+2$ in a jump, which gives the compatibility bound of \Cref{lem:weak-duality}; and each station chooses among finitely many buffers, so \Cref{assump:standing}\,(i) holds. For the baseline kernel $q^\circ$ of \Cref{assump:standing}\,(ii) we take the proportionally randomized policy of \citet[App.~C]{daigluzman2022}, under which a nonempty buffer $k$ at station $s$ is served at rate $\mu_kx_k/\sum_{i\in\mathcal C_s}x_i$, where $\mathcal C_s$ is the set of classes served by station $s$. As a continuous-time chain, this is head-of-the-line proportional processor sharing, whose fluid model drains whenever the traffic intensity of every station is below one \citep{bramson1996hlpps}. The moment bound of \citet[Prop.~5.3]{daimeyn1995}, with $p=1$, then shows that $W^\circ(x):=\mathbb E_x\int_0^{\tau_C}\big(1+|X(s)|\big)\,ds$, where $\tau_C$ is the hitting time of a suitable finite set $C$, satisfies $\mathcal G^{q^\circ}W^\circ\le-c+b\,\mathbf 1_C$ for a constant $b$ and $W^\circ\le C_W(1+V)$. It is straightforward to see that the chain of $q^\circ$ is irreducible. Finally, the Gibbs control used below takes the reference kernel $r_x$ uniform over the nonempty buffers of each station, so that $D_{\mathrm{KL}}(q^\circ(\cdot\mid x)\,\|\,r_x)\le2\log3$ and $K_r<\infty$. By \Cref{thm:soft-periodic}\,(i)--(ii), the regularized periodic BVP~\eqref{eq:periodic-bvp} has the stationary solution $h_\varepsilon+(t-T)g^*_\varepsilon$, its Gibbs kernel $q^*_\varepsilon$ is optimal for the regularized problem and positive recurrent, and every solution in the class of \Cref{thm:soft-periodic}\,(ii) has endpoint gap at least $Tg^*_\varepsilon$. We have not been able to verify the moment condition $\langle V,\tilde\pi^*_\varepsilon\rangle<\infty$ in \Cref{thm:soft-periodic}\,(iii) for this network, so the uniqueness of the periodic solution is not established here. 

\begin{figure}[htb]
\centering
\begin{tikzpicture}[>=stealth,
  buf/.style={draw, thick, minimum size=7mm, rounded corners=1.5pt},
  s1/.style={buf, fill=blue!10}, s2/.style={buf, fill=red!10},
  rt/.style={->, thick}]
  \node[s1] (b1) at (0,1.4) {$1$};
  \node[s2] (b4) at (2,1.4) {$4$};
  \node[s1] (b2) at (4,1.4) {$2$};
  \node[s2] (b5) at (6,1.4) {$5$};
  \node[s1] (b3) at (2,0)   {$3$};
  \node[s2] (b6) at (4,0)   {$6$};
  \draw[rt] (-1.4,1.4) -- (b1) node[midway,above,font=\scriptsize]{$\lambda$};
  \draw[rt] (b1) -- (b4);
  \draw[rt] (b4) -- (b2);
  \draw[rt] (b2) -- (b5);
  \draw[rt] (b5) -- ++(1.4,0) node[right,font=\scriptsize]{out};
  \draw[rt] (0.6,0) -- (b3) node[midway,above,font=\scriptsize]{$\lambda$};
  \draw[rt] (b3) -- (b6);
  \draw[rt] (b6) -- ++(1.4,0) node[right,font=\scriptsize]{out};
  \foreach \n/\m in {b1/{\frac18}, b2/{\frac12}, b3/{\frac14}, b4/{\frac16}, b5/{\frac17}, b6/{1}}
     \node[font=\scriptsize] at ([yshift=-6.5mm]\n) {$\mu=\m$};
\end{tikzpicture}
\caption{The six-class reentrant network of \citet{daigluzman2022}. Jobs route $1\to4\to2\to5\to\text{out}$ and $3\to6\to\text{out}$, with service rates $\mu$ below each buffer and exogenous Poisson arrivals at rate $\lambda=9/140$. Station~$1$ serves buffers $\{1,2,3\}$ (blue) and station~$2$ serves $\{4,5,6\}$ (red); each station is thus a shared resource visited more than once along a route, the defining feature of a reentrant line and the source of the workload coupling in~\eqref{eq:workload-matrix}.\label{fig:network}}

\end{figure}
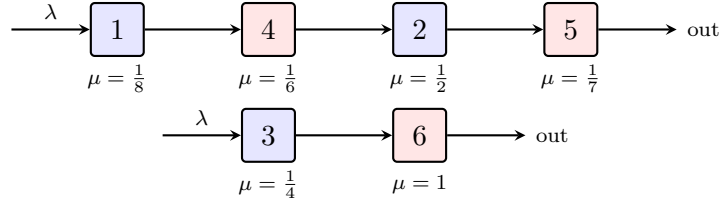

The difficulty of computing the optimal policy for reentrant networks is structural, and not a matter of optimization effort. Recall that the greedy content of the Gibbs kernel~\eqref{eq:gibbs-kernel} is the service `advantage' $\mu_k\big(f(x^{(k)})-f(x)\big)$, the class-$k$ term of the controlled generator $\mathcal G^u f$, where the post-service state $x^{(k)}:=x-e_k+e_{s(k)}$ routes the completed job from buffer $k$ to its successor $s(k)$ on the reentrant line; here $e_{s(k)}=0$ at a route's exit. As $\rho\to1$ the dual $f$ grows without bound: it is the relative value function of the ergodic problem, and both the mean cost it centers and this relative value diverge as the network saturates into heavy traffic, whose diffusion limit is a Brownian control problem in the workload~\citep{harrison2000,harrisonvanmieghem1997}. The differences $f(x^{(k)})-f(x)$ that separate admissible actions, by contrast, grow far more slowly. Heuristically, the ratio of the decision-relevant advantage to the magnitude of $f$ from which it must be extracted is of order $(1-\rho)^2$, which is about $1\%$ at $\rho=0.9$. A dual function $f$ parametrized directly in the queue coordinate $x$ must therefore resolve a decision-relevant signal that is a vanishing fraction of the magnitude it is fitting. Network-agnostic parametrizations, such as a multilayer perceptron in $x$ or the permutation-invariant pooling of \Cref{sec:pooling-baseline}, can solve the problem at light load. On the other hand, the induced greedy policy degrades sharply for $\rho\gtrsim0.85$ and typically diverges beyond $\rho=0.9$ in our experiments; see \Cref{fig:pooling-vs-workload}(a). The only mechanism to mitigate this is to introduce the `right' inductive bias into the neural architecture.

\subsubsection{The equivalent-workload architecture}\label{sec:ssc-arch}

The inductive bias comes from the static planning problem (SPP)~\citep{harrison2000} for the network. The SPP identifies, through the dual of its capacity constraints, a \emph{workload matrix} $M=R\,(I-P)^{-1}\in\mathbb R^{2\times6}$, where $P$ is the routing matrix and $R_{sk}=1/\mu_k$ when station $s$ serves class $k$~\citep{harrison2000}. Its rows are the nominal workload each class imposes on each station over its entire remaining route:
\begin{equation}\label{eq:workload-matrix}
    M=\begin{pmatrix} 10 & 2 & 4 & 2 & 0 & 0 \\ 13 & 7 & 1 & 13 & 7 & 1 \end{pmatrix}.
\end{equation}
The two-dimensional \emph{workload} $W=Mx$ is precisely the heavy-traffic `sufficient statistic' needed to be built into the neural architecture. Indeed, under state-space collapse~\citep{harrisonvanmieghem1997}, the diffusion-scale relative value depends on the state only through $W$, and the deviation of the queue from the workload cone is asymptotically negligible.

The equivalent-workload architecture compiles this structure into the dual function approximator. Let $z=(I-M^{+}M)\,x$ be the projection of the queue vector onto $\ker M$, the null space of the workload map, so that $x\mapsto(W,z)$ is a change of coordinates with $Mz=0$; here $M^{+}$ is the Moore--Penrose inverse. We parametrize
\begin{equation}\label{eq:ssc-net}
    f_\theta(t,x)\;=\;V_\theta\big(W_1,W_2,t\big)\;+\;T_\theta(z),
    \qquad W=Mx,\quad z=(I-M^{+}M)\,x,
\end{equation}
with $V_\theta,T_\theta$ multilayer perceptrons that we term the `workload value' and the `transverse correction' respectively. Both are two-hidden-layer perceptrons of width $64$ with sigmoid linear unit (SiLU) activations. $V_\theta$ takes the quadratic workload features $(W_1,W_2,W_1^2,W_2^2,W_1W_2)$ together with $t/T$. $T_\theta$ takes the transverse coordinate $z=(I-M^{+}M)x\in\mathbb{R}^6$; since $\operatorname{rank}M=2$, $z$ is confined to the four-dimensional null space $\ker M$ (equivalently $Mz=0$), and is supplied to $T_\theta$ in its ambient $\mathbb{R}^6$ coordinates rather than in a basis of $\ker M$.

The two components play asymmetric roles. The workload value $V_\theta$ carries the diverging part of $f$ and the gain $g^*$; being a function of only the two-dimensional $W$, it can be resolved sharply even at a high load of $\rho=0.9$. The transverse correction $T_\theta$ is initialized to zero and, as a function of $z\in\ker M$ alone, it supplies only the sub-dominant, finite-$\rho$ correction to the state-space collapse. The advantage $\mu_k\big(f_\theta(x^{(k)})-f_\theta(x)\big)$ is therefore dominated by $V_\theta(Mx^{(k)})-V_\theta(Mx)$, the discrete gradient of a well-conditioned two-dimensional function, and the $(1-\rho)^2$ conditioning obstruction of \Cref{sec:ht-obstruction} is confined to the vanishing transverse term rather than corrupting the entire advantage. In this sense~\eqref{eq:ssc-net} encodes the equivalent workload formulation directly in the hypothesis class of $f$. \Cref{fig:ssc-arch} depicts the construction.

\begin{figure}[htb]
\centering
\begin{tikzpicture}[>=stealth,
  io/.style={draw, thick, rounded corners=1.5pt, fill=gray!8, minimum size=7mm},
  wl/.style={draw, thick, rounded corners=1.5pt, fill=blue!10, minimum height=8mm},
  tv/.style={draw, thick, rounded corners=1.5pt, fill=red!10, minimum height=8mm},
  op/.style={draw, thick, circle, minimum size=6mm},
  fl/.style={->, thick}]
  \node[io] (x) at (0,0.3) {$x\in\mathbb Z^6_+$};
  \node[wl] (W) at (2.9,1.3) {$W=Mx$};
  \node[tv] (z) at (2.9,-0.9) {$z=(I-M^{+}M)x$};
  \node[font=\scriptsize] (t) at (4.7,2.6) {$t$};
  \node[wl] (V) at (6.2,1.5) {$V_\theta(W_1,W_2,t)$};
  \node[tv] (T) at (6.2,-0.9) {$T_\theta(z)$};
  \node[op] (sum) at (8.3,0.3) {$+$};
  \node[io] (f) at (9.9,0.3) {$f_\theta(t,x)$};
  \draw[fl] (x) -- (W) node[pos=0.62, above, font=\scriptsize, fill=white, inner sep=1pt]{$M$};
  \draw[fl] (x) -- (z) node[pos=0.5, below, font=\scriptsize, fill=white, inner sep=1pt]{$I-M^{+}M$};
  \draw[fl] (t) -- (V);
  \draw[fl] (W) -- (V);
  \draw[fl] (z) -- (T);
  \draw[fl] (V) -- (sum);
  \draw[fl] (T) -- (sum);
  \draw[fl] (sum) -- (f);
  \node[font=\scriptsize, blue!55!black] at (6.2,0.62) {dominant: carries $g^*$};
  \node[font=\scriptsize, red!55!black] at (6.2,-1.72) {init $0$: sub-dominant};
\end{tikzpicture}
\caption{The equivalent-workload architecture~\eqref{eq:ssc-net}. The queue vector $x\in\mathbb Z^6_+$ is split by the workload map $M$ of~\eqref{eq:workload-matrix} into the two-dimensional workload $W=Mx$ and the transverse coordinate $z=(I-M^{+}M)x\in\ker M$. The workload value $V_\theta(W_1,W_2,t)$ depends on only two spatial variables and time; it carries the diverging part of the dual and the gain $g^*$ and is resolved sharply through $\rho=0.9$. The transverse correction $T_\theta(z)$ is initialized to zero and, as a function of $z$ alone, is structurally unable to represent the workload, supplying only the sub-dominant finite-$\rho$ correction. Their sum is the dual $f_\theta(t,x)$.\label{fig:ssc-arch}}

\end{figure}
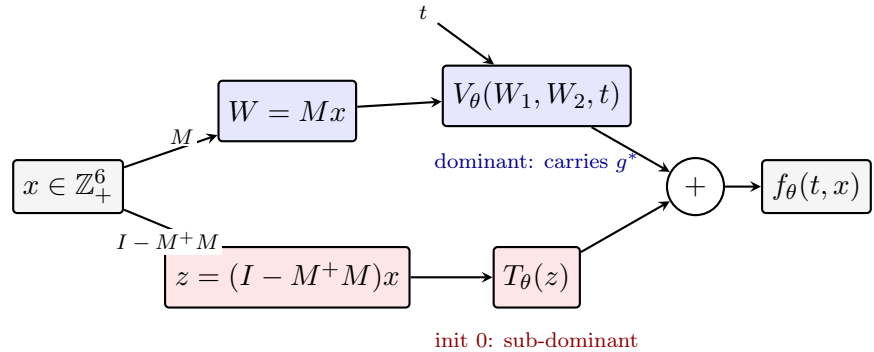

\subsubsection{An architecture-agnostic baseline}\label{sec:pooling-baseline}

We compare the workload architecture against a permutation-invariant pooling network of the DeepSets type~\citep{zaheer2017deepsets}, which uses the same raw per-class data. Each class $k$ carries a feature vector
\begin{equation}\label{eq:pool-feat}
  \xi_k=\big(x_k,\ 1/\mu_k,\ w^{(1)}_k,\ w^{(2)}_k,\ \lambda_k,\ \chi_k\big),
\end{equation}
recording its queue length, its immediate service requirement $1/\mu_k$, the nominal work $w^{(s)}_k$ it places on each station $s$ (the columns of the workload matrix $M$), its arrival rate $\lambda_k$, and an exit indicator $\chi_k$. A shared encoder $\phi_\theta$ is applied to every class, its outputs are summed within each station, and a decoder $D_\theta$ reads the pooled statistics together with time:
\begin{equation}\label{eq:pool-net}
  P_s=\sum_{k\in\mathcal C_s}\phi_\theta(\xi_k)\ \in\mathbb R^{m},\quad s\in\{1,2\},
  \qquad
  f_\theta(t,x)=D_\theta\big(P_1,P_2,\,n_1,n_2,\,t/T\big),
\end{equation}
where $\mathcal C_s$ is the set of classes served by station $s$, $n_s=\#\{k\in\mathcal C_s:x_k>0\}$ counts its nonempty buffers, and $\phi_\theta,D_\theta$ are multilayer perceptrons. Summing over each station makes $f_\theta$ invariant to the labeling of classes within a station, the symmetry a parametrization with no privileged coordinate should respect; \Cref{fig:pooling-arch} shows the construction.

The single structural knob is the pooling dimension $m$, the width of the per-station statistic the network is permitted to learn, and it is precisely the object the equivalent-workload architecture fixes by hand. Where~\eqref{eq:ssc-net} hard-codes the two-dimensional linear pool $W=Mx$, the pooling network is handed each class's nominal work $w^{(s)}_k$ only as one input feature and must \emph{discover}, through the nonlinear encoder and the sum, which statistic of the queue is sufficient in heavy traffic. This makes it the right control for the inductive-bias question, as it shares the workload architecture's information and its permutation symmetry, and differs only in whether the sufficient statistic is imposed or learned. \Cref{sec:workload-results} scans $m$ and compares the two parametrizations as a function of load.

\begin{figure}[htb]
\centering
\begin{tikzpicture}[>=stealth,
  cb/.style={draw, thick, circle, minimum size=6.5mm, font=\small, fill=blue!10},
  cr/.style={draw, thick, circle, minimum size=6.5mm, font=\small, fill=red!10},
  plb/.style={draw, thick, rounded corners=1.5pt, minimum height=8mm, fill=blue!10},
  plr/.style={draw, thick, rounded corners=1.5pt, minimum height=8mm, fill=red!10},
  io/.style={draw, thick, rounded corners=1.5pt, fill=gray!8, minimum size=7mm},
  fl/.style={->, thick}]
  \node[cb] (x1) at (0,2.3) {$x_1$};
  \node[cb] (x2) at (0,1.7) {$x_2$};
  \node[cb] (x3) at (0,1.1) {$x_3$};
  \node[cr] (x4) at (0,0.1) {$x_4$};
  \node[cr] (x5) at (0,-0.5) {$x_5$};
  \node[cr] (x6) at (0,-1.1) {$x_6$};
  \node[plb] (P1) at (4.2,1.7) {$P_1=\sum_{k\in\mathcal C_1}\phi_\theta(\xi_k)$};
  \node[plr] (P2) at (4.2,-0.5) {$P_2=\sum_{k\in\mathcal C_2}\phi_\theta(\xi_k)$};
  \node[io] (rho) at (7.8,0.6) {$D_\theta$};
  \node[io] (f) at (9.7,0.6) {$f_\theta(t,x)$};
  \node[font=\scriptsize] (t) at (6.6,2.1) {$t$};
  \foreach \x in {x1,x2,x3} \draw[fl, blue!55] (\x) -- (P1);
  \foreach \x in {x4,x5,x6} \draw[fl, red!55] (\x) -- (P2);
  \draw[fl] (P1) -- (rho); \draw[fl] (P2) -- (rho);
  \draw[fl] (t) -- (rho); \draw[fl] (rho) -- (f);
  \node[font=\scriptsize] at (2.05,2.62) {shared encoder $\phi_\theta$};
  \node[font=\scriptsize] at (4.2,-1.55) {station pools, $P_s\in\mathbb R^m$};
\end{tikzpicture}
\caption{The architecture-agnostic pooling baseline~\eqref{eq:pool-net}. A shared encoder $\phi_\theta$ maps each class $k$ to an $m$-dimensional embedding of its features $\xi_k$ (queue length, service rate, nominal work on each station, arrival rate, exit flag), and the embeddings are summed within each station to form the pooled statistics $P_1,P_2\in\mathbb R^m$ (station~$1$ serves classes $\{1,2,3\}$, station~$2$ serves $\{4,5,6\}$). A decoder $D_\theta$ maps the pools, the nonempty-buffer counts $n_1,n_2$, and time to the dual $f_\theta(t,x)$.\label{fig:pooling-arch}}

\end{figure}
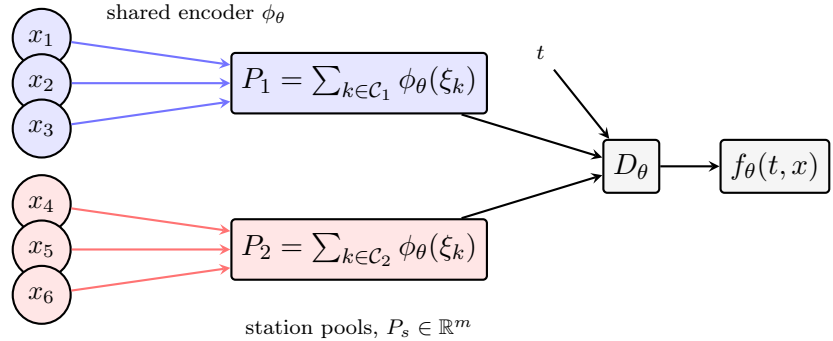

One might ask whether a richer \emph{network-agnostic} aggregator, such as self-attention over the classes in place of the plain sum, which subsumes~\eqref{eq:pool-net} and can weight classes adaptively, would close the gap. Such a network would indeed remain agnostic, but the heavy-traffic obstruction of \Cref{sec:ht-obstruction} is one of conditioning, rather than expressivity. The decision-relevant advantage $\mu_k(f(x^{(k)}) - f(x))$ is a vanishing $\mathcal O\!\big((1-\rho)^2\big)$ fraction of the value regardless of how the sufficient statistic is aggregated. A learned aggregator must still extract that statistic from the same ill-conditioned signal. We believe attention could be a good vehicle for the corrections once the workload coordinate is built into the attention mechanism, but not a network-agnostic substitute. We explore this in \Cref{sec:switch}.

\subsubsection{Numerical results}\label{sec:workload-results}

We solve the periodic BVP~\eqref{eq:periodic-bvp} for $f_\theta$ by \emph{collocation} of the soft-HJB residual. Specifically, we fit $f_\theta$ by driving the residual to zero in mean square over a sample of states, rather than discretizing the equation on a dense grid over the state space. Collocation is a standard device for boundary-value and optimal-control problems, where it reduces the problem to a finite set of residual conditions~\citep{betts2010practical}; we use its physics-informed neural-network form~\citep{raissi2019pinn}, in which a network is fitted to the residual of the equation at sampled points. The gain is read from the endpoint gap $g^*=(f_T-f_0)/T$ of~\eqref{eq:periodic-dual}. Fitting the dual directly at $\rho=0.9$ is hard: the relative value $h$ it must represent grows like $(1-\rho)^{-2}$, so a network started cold at heavy load must learn a large, stiff target at once. We instead grow the load along a homotopy (\Cref{fig:homotopy}): the dual is fit at light load, where $h$ is small, and warm-started from one load to the next through $\rho=0.5,0.6,\dots,0.9$, annealing $\varepsilon$ at each stage, so the network only ever adds the increment in $h$ between neighboring loads. Policies are scored by long-run average holding cost in a continuous-time simulator of the network; every policy we consider below is evaluated in that same simulator, eliminating any convention mismatch between methods. The published QGym numbers are not reproduced by its released simulator in our runs, so we re-run every baseline, including PPO, in a single validated simulator.

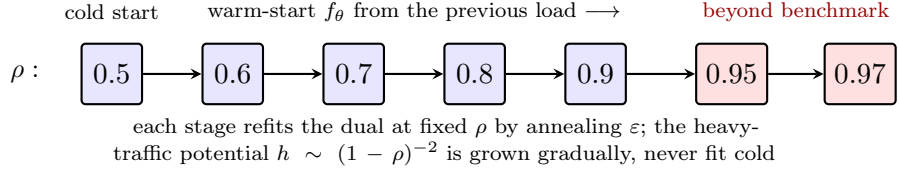
\begin{figure}[htb]
\centering
\begin{tikzpicture}[>=stealth,
  ld/.style={draw, thick, rounded corners=1.5pt, minimum size=8mm, fill=blue!10},
  ldx/.style={draw, thick, rounded corners=1.5pt, minimum size=8mm, fill=red!12},
  fl/.style={->, thick}]
  \node[font=\small] at (-1.15,0) {$\rho:$};
  \node[ld] (a) at (0,0) {$0.5$};
  \node[ld] (b) at (1.6,0) {$0.6$};
  \node[ld] (c) at (3.2,0) {$0.7$};
  \node[ld] (d) at (4.8,0) {$0.8$};
  \node[ld] (e) at (6.4,0) {$0.9$};
  \node[ldx] (f) at (8.2,0) {$0.95$};
  \node[ldx] (g) at (9.9,0) {$0.97$};
  \foreach \x/\y in {a/b,b/c,c/d,d/e,e/f,f/g} \draw[fl] (\x)--(\y);
  \node[font=\scriptsize] at (0,0.82) {cold start};
  \node[font=\scriptsize] at (4.0,0.82) {warm-start $f_\theta$ from the previous load $\longrightarrow$};
  \node[font=\scriptsize, red!60!black] at (9.05,0.82) {beyond benchmark};
  \node[font=\scriptsize, text width=10.4cm, align=center] at (4.4,-0.9)
    {each stage refits the dual at fixed $\rho$ by annealing $\varepsilon$; the heavy-traffic potential $h\sim(1-\rho)^{-2}$ is grown gradually, never fit cold};
\end{tikzpicture}
\caption{The load homotopy used to train the workload dual. The dual is fit cold at light load and warm-started up through increasing $\rho$, annealing $\varepsilon$ at each stage, so the network never learns the large heavy-traffic potential at once; \Cref{sec:deep-ht} continues the homotopy past the benchmark to $\rho=0.95$ and $0.97$.\label{fig:homotopy}}

\end{figure}

Training uses Adam (learning rate $10^{-3}$, batch $4096$ collocation states), with $\varepsilon$ annealed along the geometric schedule $0.5,0.25,0.1,0.05,0.025$ at each load stage; the residual-conditioning stage of \Cref{sec:deep-ht} runs a further $8\times10^{4}$ steps at fixed $\varepsilon=0.05$. A trained dual is scored as the Gibbs control at the terminal $\varepsilon$, serving each server's class with probability proportional to $\exp\!\big(\mu_k(f_\theta(x^{(k)})-f_\theta(x))/\varepsilon\big)$ over its nonempty buffers (work-conserving, non-idling); each evaluation averages holding cost over a long run after warm-up. Proximal policy optimization (PPO) is a work-conserving actor--critic trained on the uniformized chain and scored in the same simulator. Throughout the numerical sections, a reported value for a learned controller is the mean over independent training runs, each run scored by averaging its evaluation seeds; the quoted $\pm$ is the standard error across training runs, which dominates the within-run evaluation error. Percentage comparisons against a baseline are computed per run under common random numbers and then averaged.

\Cref{fig:pooling-vs-workload}(a) traces the induced policy cost against $\rho$ for the architecture-agnostic pooling baseline of \Cref{sec:pooling-baseline} and for~\eqref{eq:ssc-net}. The workload architecture tracks the last-buffer-first-served (LBFS) heuristic through $\rho=0.9$, precisely where the pooling baseline diverges. \Cref{fig:pooling-vs-workload}(b) scans the pooling dimension $m$ of~\eqref{eq:pool-net} and locates the empirical collapse dimension: a scalar per-station statistic ($m=1$) is insufficient, whereas an $m\approx4$--$8$ learned statistic suffices and larger $m$ over-parametrizes---measuring the sufficiency of the workload coordinate directly from data.

At $\rho=0.9$, \Cref{tab:head-to-head} places the workload dual against the classical priority rules and against work-conserving PPO trained in the same simulator. The dual, sharpened by annealing the regularization $\varepsilon$ toward zero, attains an average holding cost of $14.04$, comparable to work-conserving PPO ($14.10$) and well below LBFS ($14.32$); both learned controllers dominate the classical rules ($c\mu$ $17.78$, MaxWeight $17.58$, MaxPressure $23.40$) by a wide margin, with standard errors reported in \Cref{tab:head-to-head}. A principled, low-dimensional value-based dual is thus competitive with state-of-the-art policy-gradient reinforcement learning on the standard heavy-traffic benchmark, while improving on the strong LBFS heuristic. In fact, in \Cref{sec:deep-ht} below we show that this can be further refined.%

\begin{table}[htb]
\centering
\caption{Long-run average holding cost on the six-class reentrant network at $\rho=0.9$, all policies scored in one continuous-time simulator. $c\mu$/MaxWeight/MaxPressure are the classical priority rules; LBFS is last-buffer-first-served; PPO is work-conserving proximal policy optimization trained in the same simulator; the BAR-SOT dual uses the workload architecture~\eqref{eq:ssc-net} with $\varepsilon$-annealing. Values are mean $\pm$ standard error over independent seeds: five for the classical rules and LBFS, five independent training runs (three evaluation seeds each) for both PPO and the BAR-SOT dual.\label{tab:head-to-head}}
\makebox[\linewidth][c]{\begin{tabular}{lc}
\toprule
Policy & Avg.\ holding cost \\
\midrule
$c\mu$ & $17.78 \pm 0.01$ \\
MaxWeight & $17.58 \pm 0.01$ \\
MaxPressure & $23.40 \pm 0.01$ \\
LBFS & $14.32 \pm 0.01$ \\
PPO (work-conserving) & $14.10 \pm 0.02$ \\
\textbf{BAR-SOT dual (workload architecture)} & $14.04 \pm 0.01$ \\
\bottomrule
\end{tabular}
}

\end{table}

\begin{figure}[htb]
\centering
\begin{minipage}[t]{0.48\linewidth}\centering
\includegraphics[width=\linewidth]{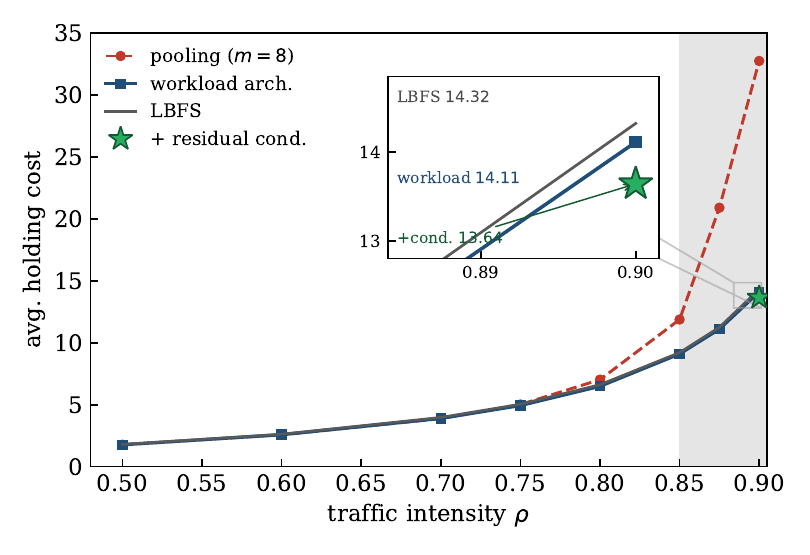}\endgraf\medskip
{\footnotesize\baselineskip=10pt\raggedright (a) Induced policy cost versus traffic intensity $\rho$ for an architecture-agnostic permutation-invariant pooling parametrization of $f$ (dim $m{=}8$) and the equivalent-workload architecture~\eqref{eq:ssc-net}, against the LBFS heuristic; all policies are scored in one continuous-time simulator. The workload architecture tracks LBFS through $\rho=0.9$, whereas the pooling parametrization diverges for $\rho\gtrsim0.85$ (shaded). The star marks the residual-conditioned dual~\eqref{eq:residual-cond} at $\rho=0.9$, which falls to $13.64$, below LBFS; the inset magnifies the $\rho\!\to\!0.9$ cluster.\endgraf}
\end{minipage}\hfill
\begin{minipage}[t]{0.48\linewidth}\centering
\includegraphics[width=\linewidth]{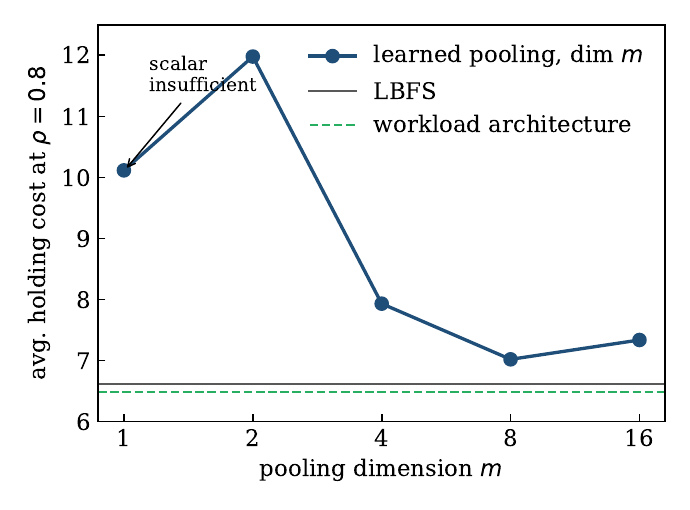}\endgraf\medskip
{\footnotesize\baselineskip=10pt\raggedright (b) Collapse-dimension scan: induced cost at $\rho=0.8$ against the pooling dimension $m$ of a permutation-invariant parametrization of $f$. A scalar statistic ($m=1$) is insufficient; $m\approx4$--$8$ approaches the workload architecture and LBFS, and larger $m$ over-parametrizes---thereby locating the empirical state-space-collapse dimension.\endgraf}
\end{minipage}
\caption{{Induced policy cost of the pooling and workload parametrizations of $f$ on the six-class network.}\label{fig:pooling-vs-workload}}

\end{figure}

\subsubsection{Improving the transverse correction}\label{sec:deep-ht}

The heavy-traffic theory of reentrant-line networks identifies where an improvement on LBFS must originate. Under state-space collapse (SSC)~\citep{bramson1998ssc,williams1998diffusion}, the diffusion-scaled queue length is a fixed image of the workload $W=Mx$, so the leading-order holding cost is a function of the workload alone. These results establish SSC for classes of work-conserving head-of-line policies that include the static priority rules LBFS and $c\mu$ and the maximum-pressure rule. The learned controllers are not covered by the theorems. Each of them is nonetheless head-of-line and work-conserving, and we observe empirically that the trained PPO is essentially LBFS, agreeing with LBFS on roughly $92\%$ of station-1 and $98$--$99\%$ of station-2 decisions (\Cref{fig:ppo-dual}); the workload dual carries the workload $W=Mx$ in its architecture by construction. Therefore, the class of head-of-line policies considered agrees at leading order and can differ only through the transverse allocation $z\in\ker M$, at lower order.  Indeed, this is the structure that the decomposition~\eqref{eq:ssc-net} encodes, with $V_\theta(W)$ the leading-order value shared by all of them, and $T_\theta(z)$ the lower-order correction through which a policy may improve upon a static priority such as LBFS. Any margin over LBFS lies entirely in this transverse term.

While the architecture can represent the improved allocation, the least-squares objective provides almost no signal with which to fit it. To see why this is the case, we observe that representing the correction and resolving it from the training objective are distinct issues, and the $(1-\rho)^2$ conditioning obstruction of \Cref{sec:ht-obstruction} obstructs the latter. The soft-HJB residual of~\eqref{eq:periodic-bvp} is dominated by the leading-order workload balance. At the same time, the transverse correction contributes to it at order $(1-\rho)^2$ (the estimate of \Cref{sec:ht-obstruction} applied to the transverse term); since the objective is the squared residual, the portion of it that the correction controls is then of order $(1-\rho)^4$. Therefore, the gradient that trains $T_\theta$ is negligible at heavy load.  Consequently, a dual trained in the standard way attains only parity with PPO at $\rho=0.9$ (see \Cref{tab:head-to-head}).

We therefore improve the conditioning of the fit (in the numerical sense) so that it resolves the transverse correction rather than being dominated by the leading value. Writing the soft-HJB residual $R_\theta(t,x)=\partial_t f_\theta(t,x)+\mathcal H^\varepsilon f_\theta(t,x)$, with $\mathcal H^\varepsilon$ the soft Hamiltonian of \Cref{sec:entropy-reg}, we fix the workload value $V_\theta$ and the gain $g$ at the converged solution of the load homotopy of \Cref{sec:workload-results} (well-conditioned, being two-dimensional) and train only the transverse correction $T_\theta$ against the per-state normalized objective
\begin{equation}\label{eq:residual-cond}
  \mathcal L(\theta)=\mathbb E_{(t,x)\sim\nu}\!\Big[\big(R_\theta(t,x)/s(x)\big)^2\Big],
  \qquad s(x)=c(x)+g+1,
\end{equation}
with $t\sim\mathrm{Unif}[0,T]$ and $x$ drawn from the mixture $\tfrac12\,\omega_\theta+\tfrac12\,q$. Here $\omega_\theta$ is the stationary law of the queue-length chain under the Gibbs control $q^\varepsilon_{f_\theta}$ of \Cref{sec:entropy-reg} induced by the current $f_\theta$; equivalently, it is the measure invariant under the controlled generator $\mathcal G^{q^\varepsilon_{f_\theta}}$, which we sample as the empirical occupation of a finite simulation. The proposal $q=\bigotimes_{k}\mathrm{Geom}(m_k)$ is fixed; its per-class mean $m_k$ is set to about twice the mean length of buffer $k$ under LBFS. Being over-dispersed relative to $\omega_\theta$, it keeps the collocation covering the state space. The normalizer $s(x)$ is chosen so that it rescales each state's residual to order one. Indeed, since the correction term $T_\theta$ enters $R_\theta$ at order $(1-\rho)^2$, the unnormalized loss weights it at order $(1-\rho)^4$, whereas~\eqref{eq:residual-cond} restores an order-one gradient. With $V_\theta$ and $g$ fixed, the $(1-\rho)^2$ obstruction resides entirely in the frozen leading value $V_\theta$, which the two-dimensional workload architecture already resolves accurately. The objective that remains is governed by the correction $T_\theta$, and it is well-conditioned.

At $\rho=0.9$ this refinement of the training reduces the induced cost to $13.64$, an improvement of $4.7\%$ over LBFS and $3.3\%$ over PPO, stable across five independent training runs (\Cref{tab:deep-ht}) and confirmed by a paired common-random-number simulation. The resulting policy has an interesting structure and is a route-aware refinement of LBFS (see \Cref{fig:policy}). Whereas LBFS serves the highest-indexed nonempty buffer at each station, the conditioned dual deprioritizes the short route (class~$3$, $3\!\to\!6\!\to\text{out}$) at station~$1$ in favor of the long reentrant route ($1\!\to\!4\!\to\!2\!\to\!5\!\to\text{out}$) and serves the reentrant class~$4$ at station~$2$. Thus, acting in concert, the two servers drain the buffer that carries the most nominal work, which then holds approximately one fewer job than under LBFS. The two policies nonetheless coincide on about $84\%$ of station-$1$ and $95\%$ of station-$2$ decisions (\Cref{fig:ppo-dual}), so the gain is a targeted reallocation over a small set of contested states. %

\begin{figure}[htb]
\centering
\includegraphics[width=0.55\linewidth]{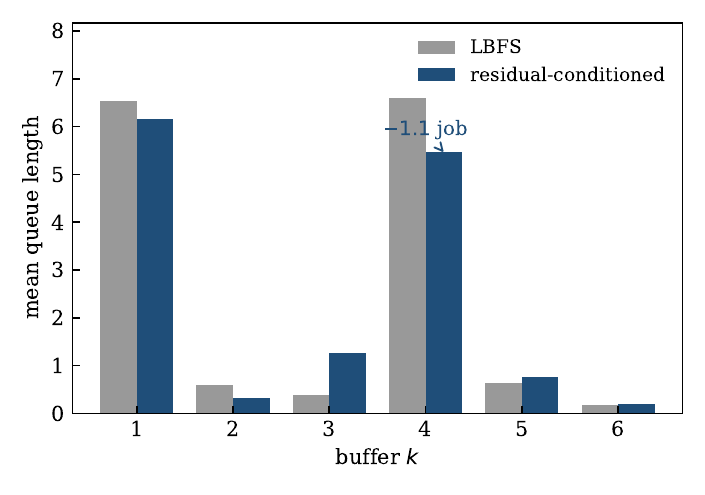}
\caption{Steady-state mean queue length per buffer under the residual-conditioned policy at $\rho=0.9$ against LBFS, both measured in the common simulator. The conditioned dual drains the reentrant bottleneck buffer~$4$ by approximately one job, and buffer~$1$ by less, while holding slightly more at the short-route buffer~$3$. The reallocation is route-aware. Where the two policies differ, the conditioned dual deprioritizes the short-route class~$3$ in favor of the long reentrant route (classes~$1,2$) at station~$1$ and serves the reentrant class~$4$ in place of class~$5$ at station~$2$. Per-station decision agreement across all three controllers is shown in \Cref{fig:ppo-dual}.\label{fig:policy}}

\end{figure}

\begin{table}[htb]
\centering
\caption{The workload dual at the benchmark load $\rho=0.9$, before and after conditioning the correction residual~\eqref{eq:residual-cond}; long-run average holding cost in the same simulator and against the same LBFS reference as \Cref{tab:head-to-head}. The residual-conditioned row is the mean $\pm$ standard error across five independent training runs, each scored against LBFS by paired common-random-number simulation. Conditioning the correction improves the dual from parity with PPO to a policy that outperforms both PPO and LBFS.\label{tab:deep-ht}}
\makebox[\linewidth][c]{\begin{tabular}{lcc}
\toprule
Policy ($\rho=0.9$) & Avg.\ holding cost & vs.\ LBFS \\
\midrule
LBFS & $14.32 \pm 0.01$ & --- \\
PPO (work-conserving) & $14.10 \pm 0.02$ & $-1.5\%$ \\
BAR-SOT dual, standard fit & $14.04 \pm 0.01$ & $-1.9\%$ \\
\textbf{BAR-SOT dual, residual-conditioned} & $\mathbf{13.64 \pm 0.03}$ & $\mathbf{-4.7\%}$ \\
\bottomrule
\end{tabular}
}

\end{table}

\begin{figure}[t]
\centering
\includegraphics[width=0.92\linewidth]{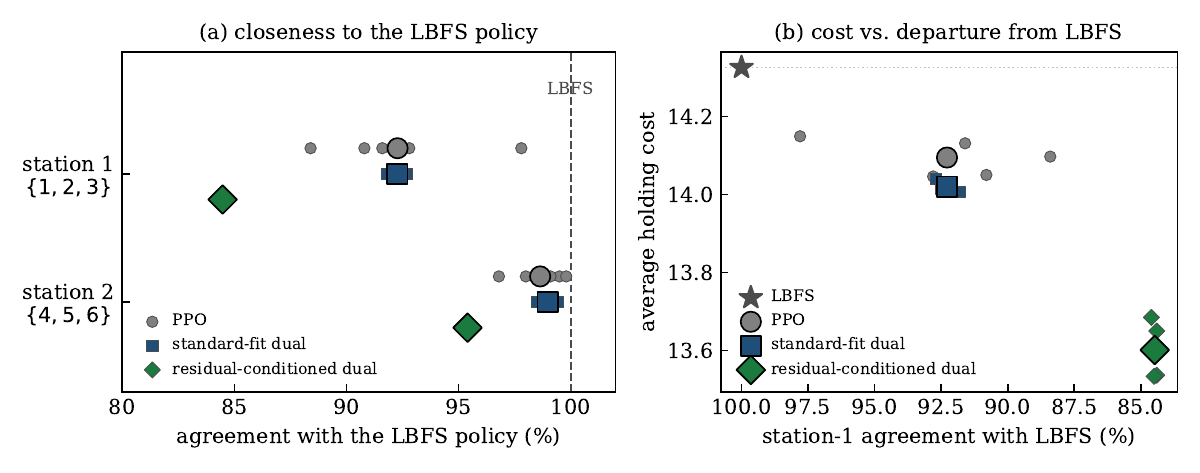}
\caption{Three learned controllers versus LBFS at $\rho=0.9$, PPO shown as five and each dual as four independent training runs (large markers: run means), evaluated on a common occupation. (a) Agreement with the LBFS policy by station. PPO and the standard-fit dual stay close to LBFS at both stations; only the residual-conditioned dual departs materially, and chiefly at station~1. (b) Average holding cost against station-1 agreement, with departure from LBFS increasing to the right. The conditioned-dual cluster is the only one that both leaves the LBFS policy and falls below its cost.\label{fig:ppo-dual}}

\end{figure}

Both ingredients of~\eqref{eq:residual-cond} are necessary, and neither suffices in isolation (\Cref{tab:ablation}). Normalizing the residual while leaving $V_\theta$ free to train recovers only the standard fit ($-1.9\%$): the leading value absorbs the correction, so the normalization has nothing sub-dominant to act on. Freezing $V_\theta$ without normalizing instead starves the correction of gradient, as the scaling in~\eqref{eq:residual-cond} anticipates, and the fit is then both worse than the standard dual and unstable across runs ($+1.4\%$ on average, ranging from $-1.3\%$ to $+4.7\%$). The improvement appears only when the leading value is frozen and the residual normalized together, so that the correction has both a sub-dominant residual to fit and an order-one gradient with which to fit it.

\begin{table}[htb]
\centering
\caption{Contribution of the two ingredients of the conditioned objective~\eqref{eq:residual-cond} at $\rho=0.9$, reported as improvement over LBFS (mean over three independent training runs with paired common-random-number simulation; five runs for the full method). Neither freezing the leading value nor normalizing the residual improves on the standard fit on its own, and freezing alone is unstable; the two are effective only in combination.\label{tab:ablation}}
\makebox[\linewidth][c]{\begin{tabular}{lc}
\toprule
Conditioning of the correction fit & vs.\ LBFS \\
\midrule
None (standard fit) & $-1.9\%$ \\
Normalize residual only & $-1.9\%$ \\
Freeze $V_\theta,g$ only & $+1.4\%$ (unstable) \\
\textbf{Freeze and normalize (full)} & $\mathbf{-4.7\%}$ \\
\bottomrule
\end{tabular}
}

\end{table}

\Cref{fig:ppo-dual} shows all three controllers against LBFS. PPO and the standard-fit dual reproduce the LBFS policy closely and their costs cluster just below the latter (\Cref{tab:deep-ht}). The residual-conditioned dual is the only controller that departs from the heuristic, primarily at station~1, where agreement falls to about $84\%$. The gain is thus a change in the station-1 scheduling decisions rather than merely tie-breaking.

Beyond the benchmark load of $\rho = 0.9$ the dual correction persists, but its resolution becomes variance-limited. Under the standard fit the greedy cost rises to $31.2$ at $\rho=0.95$ and $54.5$ at $\rho=0.97$, against LBFS values of $27.97$ and $39.93$ (excesses of $11\%$ and $37\%$). The degradation comes from the training objective. As shown above, the squared residual weights the transverse correction at order $(1-\rho)^4$, so the standard fit leaves the correction unresolved, and the cost of scheduling without it grows with the load. Residual conditioning holds the policy near LBFS, and below it on some training runs (by as much as $5\%$ at $\rho=0.95$), but the $(1-\rho)^2$ signal-to-noise ratio leaves the correction ill-conditioned to fit, and the training converges to a favorable or an unfavorable solution according to its initialization. We observed that larger batches, weight averaging, and $\varepsilon$-annealing reduce this spread without removing it. Exploring additional variance reduction is a significant exercise and worthy of its own investigation in this setting. %

Finally, we observe that this result is consistent with the available optimality theory. The sharp asymptotic-optimality results concern maximum pressure rather than LBFS and assume a single effective bottleneck, that is, complete resource pooling~\citep{dailin2008maxpressure}; our network loads both stations critically and therefore has a two-dimensional workload~\citep{atalin2008multiple}, for which no policy has been proven asymptotically optimal. %

\subsection{Scheduling the Input-Queued Switch}
\label{sec:switch}
The input-queued switch is a scheduling problem of independent interest that comes equipped with MaxWeight matching and its refinements. MaxWeight matching is a strong, widely used heuristic against which we measure the learned dual of this paper. The switch also exercises the scaling of any method, since the state is the full matrix of queue lengths, and the action is a combinatorial matching whose number of pure alternatives grows super-exponentially in the switch size. We now implement the periodic dual of \Cref{sec:periodic-bvp} and its entropy-regularized Gibbs control on this problem. 

\subsubsection{The switch and the comparison policies}
\label{sec:switch-model}

An $N\times N$ input-queued switch moves traffic in fixed-length \emph{cells}, the minimal unit of work, and keeps a separate queue $Q_{ij}$ for the cells waiting at input $i$ for output $j$; there are $N^2$ such virtual output queues, so the state is the matrix $Q=(Q_{ij})$. Time is slotted. Note that the slotted switch is a discrete-time self-transport problem in the sense of \Cref{sec:finite-mdp}, with $P^u-I$ in the role of the generator $\mathcal G^u$. The results of \Cref{sec:finite-mdp} are proved for a finite state space, while the switch backlog is unbounded. We therefore do not rely on the results of \Cref{sec:finite-mdp} here. The working equation~\eqref{eq:switch-acoe} is instead derived below from the slotted form of the periodic dual of \Cref{sec:periodic-bvp}. A \emph{matching} is defined as a permutation $\sigma$ that clears one cell from each nonempty queue $(i,\sigma(i))$, since the switch serves at most one cell per input and per output. In each slot, cells arrive, the switch serves a \emph{matching}, and a holding cost $c(Q)=\sum_{ij}c_{ij}Q_{ij}$ is incurred on the backlog that remains, with $c_{ij}>0$ the per-cell cost on edge $(i,j)$. The scheduler's only decision is which matching to serve each time slot, as a function of the current backlog, and the objective is the long-run average holding cost. %

A randomized (relaxed) matching is a doubly stochastic serve-rate matrix $P$, with $P_{ij}$ the probability of serving edge $(i,j)$; these form the Birkhoff polytope $\mathcal B_N$, whose vertices are precisely the permutation matrices. Writing $\widetilde W_{ij}=c_{ij}Q_{ij}$ for the cost-weighted backlog, the classical scheduler serves the maximum-weight matching $\sigma^\star=\arg\max_\sigma\langle\widetilde W,\sigma\rangle$, the vertex optimum of the linear program $\max_{P\in\mathcal B_N}\langle P,\widetilde W\rangle$~\citep{tassiulas1992,mckeown1999}.

We consider four comparison policies. MaxWeight serves the matching of largest total backlog, $\arg\max_\sigma\sum_i Q_{i\sigma(i)}$; it is throughput-optimal, keeping the switch stable whenever any policy can. However, it is myopic in two respects, ignoring the costs $c_{ij}$ and optimizing only the backlog cleared in the current slot with no account of where the state is driven next. Weighted MaxWeight (C-MWS) ranks matchings by cost-weighted backlog $\sum_i c_{i\sigma(i)}Q_{i\sigma(i)}$, and is cost-aware and stable but still a one-slot rule. Algorithm~2 is the $N\times N$ heuristic of \citet{lu2025weighteddelay}, which serves a matching of maximum \emph{cardinality}, clearing as many nonempty queues as possible, ensuring a backlog does not build up, and breaking ties among such matchings by cost-weighted backlog. This policy is shown to improve on MaxWeight in their steady-state experiments and is the strongest general-$N$ competitor of which we are aware; therefore, it is the policy the learned dual is measured against. However, the exact optimum is available only at $N=2$, where \citet{lu2025weighteddelay} formulate the $2\times2$ switch as a Markov decision process and solve it. 

\subsubsection{The dual and the induced schedule}
\label{sec:switch-control}

The dual for the switch is the periodic boundary-value problem~\eqref{eq:periodic-bvp} of \Cref{sec:soft-periodic-bvp}, in the state $Q$ with the matching $\sigma$ as the control. Because the dynamics are time-homogeneous, its solution takes the lift ansatz~\eqref{eq:ansatz}.
\begin{equation}\label{eq:switch-ansatz}
  f^{*}_t(Q)=h(Q)+(t-T)\,g,
\end{equation}
 Since the switch runs in slots, the boundary-value problem~\eqref{eq:periodic-bvp} takes its slotted form: over one slot the entropy-regularized dual advances by the soft dynamic-programming map
\begin{equation}\label{eq:switch-softdp}
\begin{aligned}
  f_{t+1}(Q)
  &=c(Q)+\operatorname*{softmin}^{\varepsilon,r}_{\sigma}\,
  \mathbb E_A\big[f_t(\mathrm{serve}_\sigma(Q)+A)\big], \\
  \operatorname*{softmin}^{\varepsilon,r}_{\sigma}V_\sigma
  &:=-\varepsilon\log\sum_\sigma r(\sigma\mid Q)\,e^{-V_\sigma/\varepsilon},
\end{aligned}
\end{equation}
the reference-weighted soft-minimum over the $N!$ matchings that specializes the soft Hamiltonian of \Cref{sec:entropy-reg}. Here, $\mathrm{serve}_\sigma(Q)$ is the backlog after serving the matching $\sigma$ and $A$ the arrivals in the slot, and the periodicity $f_T-f_0=\mathrm{const}$ of~\eqref{eq:periodic-dual} stands in for a terminal condition. The ansatz~\eqref{eq:switch-ansatz} makes the dual advance by exactly the gain $g$ in each slot, i.e., $f_{t+1}(Q)-f_t(Q)=g$, and its linear-in-time term is common to every matching and every arrival, so it factors through the soft-minimum,
\begin{equation*}
  \operatorname*{softmin}^{\varepsilon,r}_{\sigma}\mathbb E_A\big[h(\mathrm{serve}_\sigma(Q)+A)+(t-T)g\big]
  =(t-T)g+\operatorname*{softmin}^{\varepsilon,r}_{\sigma}\mathbb E_A\big[h(\mathrm{serve}_\sigma(Q)+A)\big].
\end{equation*}
Substituting $f_t=h+(t-T)g$ into~\eqref{eq:switch-softdp} and canceling the shared time index leaves a stationary equation for $(h,g)$, the entropy-regularized average-cost optimality equation
\begin{equation}\label{eq:switch-acoe}
  g + h(Q) \;=\; c(Q)\;-\;\varepsilon\log\sum_{\sigma} r(\sigma\mid Q)\,
  \exp\!\Big(-\tfrac1\varepsilon\,\mathbb{E}_{A}\big[h\big(\mathrm{serve}_\sigma(Q)+A\big)\big]\Big),
\end{equation}
with the reference control $r(\sigma\mid Q)$ tilting the soft-minimum toward the low-cost matchings. As $\varepsilon\to0$ the log-sum sharpens to the hard minimum and \eqref{eq:switch-acoe} becomes the classical average-cost optimality equation $g+h(Q)=c(Q)+\min_\sigma \mathbb{E}_A[h(\mathrm{serve}_\sigma(Q)+A)]$, whose minimizer is the schedule that serves the matching of least expected continuation,
\begin{equation}\label{eq:switch-greedy}
  \sigma^\star(Q)\;=\;\arg\min_\sigma\ \mathbb{E}_A\big[\,h\big(\mathrm{serve}_\sigma(Q)+A\big)\,\big].
\end{equation}
This is the value-function counterpart of MaxWeight: where MaxWeight ranks matchings by the backlog they clear in the current slot, \eqref{eq:switch-greedy} ranks them by the future cost of the resulting state. The difference between the two rankings is precisely the lookahead that a myopic rule discards.

The control paired with the soft equation~\eqref{eq:switch-acoe} is the Gibbs kernel of \Cref{sec:entropy-reg},
\begin{equation}\label{eq:switch-gibbs}
  q^\varepsilon(\sigma\mid Q)\;\propto\;r(\sigma\mid Q)\,\exp\!\Big(-\tfrac{1}{\varepsilon}\,\mathbb{E}_A\big[h(\mathrm{serve}_\sigma(Q)+A)\big]\Big),
\end{equation}
with the same reference $r$ and temperature $\varepsilon$ as in~\eqref{eq:switch-acoe}, concentrating on~\eqref{eq:switch-greedy} as $\varepsilon\to0$. When $N$ is small ($N=2,3$), the $N!$ matchings are few enough to enumerate, both the soft Hamiltonian and \eqref{eq:switch-gibbs} can be formed outright. In \Cref{sec:switch-scaling-dual}, we treat the enumeration-free form for large $N$.

\subsubsection{The construction of the learned dual}
\label{sec:switch-choices}
The reference control $r$ in \eqref{eq:switch-gibbs} anchors the stability of the entropic control. The Gibbs control inherits the \emph{support} of $r$, so every $q^\varepsilon$, at every $\varepsilon>0$, keeps positive probability on the work-conserving matchings of the reference. Support inheritance alone is not a stability proof of course. We do not verify that the softened reference is admissible in the sense of \Cref{assump:stability}; we make the empirical observation that the reference $r\propto\exp(\alpha\,\mathrm{card}_\sigma(Q)+\beta\,\mathrm{weight}_\sigma(Q))$ is stable at $\varepsilon>0$, consistently across loads. This is a softened Algorithm~2 that concentrates on the high-cardinality, high-weight matchings while retaining full support, with $\alpha$ large so that cardinality dominates. This reference also initializes the iteration at a known-good policy: at large $\varepsilon$, the control is essentially Algorithm~2, and as $\varepsilon$ falls, it departs from it only as far as the dual allows. The dual itself is represented by a neural network. The state has $N^2$ unbounded coordinates, so a lookup table for $h(Q)$ is feasible only at $N=2$; for $N\ge3$ we represent the dual by a network $h_\theta(Q)$.

The dual is fit on the states the control visits, since a value function need only be accurate where the policy operates. We alternate a \emph{forward} step---simulating the current Gibbs control and collecting the backlogs it visits in steady state---with a \emph{backward} step---updating $h_\theta$ to satisfy \eqref{eq:switch-acoe} on exactly those backlogs---refreshing the states as $h_\theta$ changes. The forward simulation is run long enough to reach steady state before collection of the backlog, so the states fed to the backward step are those the policy occupies rather than transient states. Fitting the dual accurately on a transient set places it in the wrong place and can produce an unstable schedule.

The backward step is a fitted-value update rather than a residual minimization. The natural way to enforce \eqref{eq:switch-acoe} is to minimize the squared difference between its two sides in expectation over sampled arrivals. However, this is biased, since the right-hand side holds an expectation over arrivals inside a nonlinear function of the next state, and squaring a single sampled evaluation estimates the square plus the arrival variance. This drives $h_\theta$ toward a function shaped by noise rather than by the dynamics. In order to avoid this, we form the target with the exact arrival distribution (which is known), and treat it as fixed when we regress $h_\theta$ onto the target. The gain $g$ is then the average of the target minus the value over the fitted states.

\subsubsection{Performance at moderate load}
\label{sec:switch-moderate}

We report three checks of increasing difficulty, all on average holding cost in steady state. On the $2\times2$ switch, solving \eqref{eq:switch-acoe} in tabular form by dynamic programming on the truncated chain gives the true optimal average cost, and the BAR-SOT dual reproduces it to within simulation error on a symmetric unit-cost instance and on a weighted one. On the weighted instance, it matches the optimum at $+0.0\%$, while weighted MaxWeight sits about $20\%$ above it and MaxWeight is higher still. The comparison is average-cost throughout, against our own exact solution; \citet{lu2025weighteddelay} study this switch in a partly discounted formulation, so we compare against the exact average-cost optimum directly. Recovering the known optimum confirms that the average-cost dual, solved without approximation, is the correct object.

On the symmetric $3\times3$ switch, with the network in place of the table and the matchings still enumerated, the learned dual matches Algorithm~2 on a symmetric unit-cost instance; both are essentially optimal there, so agreement is the correct outcome and confirms that the network solve reproduces the tabular one. On a weighted $3\times3$ instance, where Algorithm~2 and the optimum need not agree, the learned dual improves on Algorithm~2: averaged over five training seeds, its schedule costs $5.3\%$ less, with every seed beating it, and it improves on the MaxWeight family by wider margins (\Cref{tab:switch-moderate}). The improvement is the value of the lookahead in \eqref{eq:switch-greedy}: the dual serves matchings that are not best in the current slot but leave the switch in a cheaper state to schedule next.

\begin{table}[htb]
\centering
\caption{Average weighted backlog on the weighted $3\times3$ switch at $\rho=0.85$; lower is better.
Mean $\pm$ standard error over five training seeds for the dual; the heuristics are deterministic. The
$5.3\%$ improvement is the per-seed paired mean (each seed's dual against Algorithm~2 under common
random numbers) and need not equal the ratio of the two aggregate means shown. The heuristic entries are themselves simulation estimates; across the independent evaluation streams their standard errors are at most $0.1\%$ (Algorithm~2 $\pm0.05\%$), negligible relative to the reported margins.\label{tab:switch-moderate}}
\makebox[\linewidth][c]{\begin{tabular}{lc}
\toprule
Policy & Avg.\ weighted backlog \\
\midrule
MaxWeight            & $44.27$ \\
Weighted MaxWeight   & $31.28$ \\
Algorithm~2          & $29.30$ \\
BAR-SOT dual         & $\mathbf{27.79\pm0.04}$\ \ ($-5.3\%$) \\
\bottomrule
\end{tabular}
}

\end{table}

\subsubsection{Deep heavy traffic}
\label{sec:switch-dual-deepht}

Heavy traffic sharpens the evaluation in the same way as in \Cref{sec:heavy-traffic}. Observe, as $\rho\to1$ the queues and the value $h(Q)$ grow, but the schedule reads only the difference in value between one matching and another. This is a small fraction of $h(Q)$ that shrinks further as the load rises. The dual must therefore resolve an ever-finer difference in value against an ever-larger background, and a network fit to the value resolves the fine-differences poorly.

This is a symptom of a loss of stability, and it appears precisely when the entropy is removed. Taking $\varepsilon\to0$ turns the control into the hard minimizer over an imperfectly resolved value, and at $\rho=0.9$ and $0.95$ the resulting schedule is not merely suboptimal but unstable, its backlog running to hundreds of times Algorithm~2's while MaxWeight, weighted MaxWeight, and Algorithm~2 remain bounded. {The observed stability of the regularized control tracks the reference measure. The reference measure keeps positive probability on the stabilizing matchings at every $\varepsilon>0$, and the positive probability vanishes in the limit $\varepsilon\to0$}, consequently sharpening an exact dual to a deterministic rule is safe while sharpening an approximate one is apparently not.

The remedy is to stop, keep $\varepsilon$ bounded away from zero, so that the reference measure holds the control stable while the dual's tilt still improves on Algorithm~2. At $\varepsilon=0.5$ the control beats Algorithm~2 by $5.0\%$ at $\rho=0.9$ and $3.3\%$ at $\rho=0.95$, averaged over five seeds, with every seed besting Algorithm~2 (\Cref{tab:switch-deepht}). 

\begin{table}[htb]
\centering
\caption{Average weighted backlog on the weighted $3\times3$ switch as load increases, learned dual
against Algorithm~2. The middle column is the fully sharpened control ($\varepsilon\to0$); the right
keeps $\varepsilon=0.5$. Beats are mean $\pm$ standard error over five seeds. {The heuristic entries are themselves simulation estimates. Across the evaluation streams of the five runs, the standard errors of the heuristic entries are $\pm0.2\%$ (Algorithm~2, $\rho=0.9$) and $\pm0.4\%$ ($\rho=0.95$), negligible relative to the reported margins.}\label{tab:switch-deepht}}
\makebox[\linewidth][c]{\begin{tabular}{lccc}
\toprule
Load & Algorithm~2 & BAR-SOT, $\varepsilon\to0$ & BAR-SOT, $\varepsilon=0.5$ \\
\midrule
$\rho=0.85$ & $29.30$ & $27.79$ \ (stable, $-5.3\%$) & --- \\
$\rho=0.90$ & $41.5$  & ${\sim}2.9\times10^{4}$ (unstable) & $\mathbf{-5.0\%\pm0.17\%}$ \\
$\rho=0.95$ & $76.5$  & ${\sim}1.3\times10^{5}$ (unstable) & $\mathbf{-3.3\%\pm0.29\%}$ \\
\bottomrule
\end{tabular}
}

\end{table}

\subsubsection{Scaling to large switches}
\label{sec:switch-scaling-dual}

The results so far enumerate the matchings, which is impossible beyond a handful of ports. To handle large $N$, the value is represented by a graph-attention network on the bipartite structure of the switch. The design of this network is what makes the dual resolvable in heavy traffic and applicable at every switch size with the same parameter count. Input ports and output ports are the nodes, and the queues live on the edges: each input port $i$ carries an embedding $u_i\in\mathbb R^{d}$, each output port $j$ an embedding $v_j\in\mathbb R^{d}$, and each edge $(i,j)$ an embedding $e_{ij}\in\mathbb R^{d}$. The node embeddings are seeded with the port workloads, and the edges with the local backlog, holding cost, and per-cell arrival rate $\lambda_{ij}$,
\begin{equation}\label{eq:gnn-init}
  u_i^{0}=\phi_{\mathrm{in}}(R_i),\quad v_j^{0}=\phi_{\mathrm{out}}(C_j),\quad e_{ij}^{0}=\phi_{\mathrm e}(Q_{ij},c_{ij},\lambda_{ij}),
  \qquad R_i=\textstyle\sum_j Q_{ij},\quad C_j=\sum_i Q_{ij},
\end{equation}
with $\phi_{\mathrm{in}},\phi_{\mathrm{out}},\phi_{\mathrm e}$ small multilayer perceptrons. Over a few layers each port refreshes its embedding by an attention-weighted aggregation over its own edges, and each edge from its two endpoints,
\begin{equation}\label{eq:gnn-round}
  u_i \;\leftarrow\; u_i+\psi_{\mathrm{in}}\!\Big(\sum_j \alpha_{ij}\,e_{ij}\Big),
  \quad v_j \;\leftarrow\; v_j+\psi_{\mathrm{out}}\!\Big(\sum_i \beta_{ij}\,e_{ij}\Big),
  \quad e_{ij}\;\leftarrow\; e_{ij}+\psi_{\mathrm e}(u_i,v_j,e_{ij}),
\end{equation}
where $\alpha_{i\cdot}=\operatorname*{softmax}_{j} s(u_i,v_j,e_{ij})$ and $\beta_{\cdot j}=\operatorname*{softmax}_{i} s(u_i,v_j,e_{ij})$ are the row- and column-wise attention weights of a shared score $s$, and $\psi_{\mathrm{in}},\psi_{\mathrm{out}},\psi_{\mathrm e}$ are further multilayer perceptrons. The value is a permutation-invariant readout of the port embeddings,
\begin{equation}\label{eq:gnn-readout}
  h_\theta(Q)\;=\;\psi_{\mathrm r}\Big(\tfrac1N\textstyle\sum_i u_i,\ \tfrac1N\sum_j v_j\Big).
\end{equation}
\Cref{fig:switch-gnn} depicts the construction.

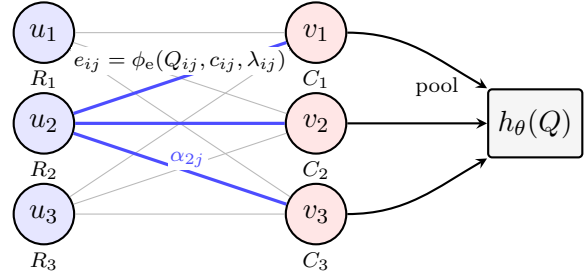
\begin{figure}[htb]
\centering
\begin{tikzpicture}[>=stealth,
  inp/.style={draw, thick, circle, minimum size=8mm, fill=blue!10},
  outp/.style={draw, thick, circle, minimum size=8mm, fill=red!10},
  ed/.style={gray!55}, att/.style={blue!70, very thick}]
  \node[inp] (i1) at (0,2.4) {$u_1$}; \node[font=\scriptsize, fill=white, inner sep=1pt] at (0,1.78) {$R_1$};
  \node[inp] (i2) at (0,1.2) {$u_2$}; \node[font=\scriptsize, fill=white, inner sep=1pt] at (0,0.58) {$R_2$};
  \node[inp] (i3) at (0,0)   {$u_3$}; \node[font=\scriptsize, fill=white, inner sep=1pt] at (0,-0.62) {$R_3$};
  \node[outp] (o1) at (3.6,2.4) {$v_1$}; \node[font=\scriptsize, fill=white, inner sep=1pt] at (3.6,1.78) {$C_1$};
  \node[outp] (o2) at (3.6,1.2) {$v_2$}; \node[font=\scriptsize, fill=white, inner sep=1pt] at (3.6,0.58) {$C_2$};
  \node[outp] (o3) at (3.6,0)   {$v_3$}; \node[font=\scriptsize, fill=white, inner sep=1pt] at (3.6,-0.62) {$C_3$};
  \foreach \i in {i1,i2,i3} \foreach \o in {o1,o2,o3} \draw[ed] (\i)--(\o);
  \draw[att] (i2)--(o1); \draw[att] (i2)--(o2); \draw[att] (i2)--(o3);
  \node[font=\scriptsize, fill=white, inner sep=1pt] at (1.8,2.02) {$e_{ij}=\phi_{\mathrm e}(Q_{ij},c_{ij},\lambda_{ij})$};
  \node[att, font=\scriptsize, fill=white, inner sep=1pt] at (1.9,0.72) {$\alpha_{2j}$};
  \node[draw, thick, rounded corners=1.5pt, minimum height=9mm, fill=gray!8] (h) at (6.5,1.2) {$h_\theta(Q)$};
  \draw[->, thick] (o1) to[out=0,in=150] (h.north west);
  \draw[->, thick] (o2) -- (h.west);
  \draw[->, thick] (o3) to[out=0,in=210] (h.south west);
  \node[font=\scriptsize] at (5.2,1.68) {pool};
\end{tikzpicture}
\caption{The bipartite graph-attention dual $h_\theta(Q)$ of \eqref{eq:gnn-init}--\eqref{eq:gnn-readout}. The $N$ input ports (blue) and $N$ output ports (red) are the nodes and the $N^2$ queues live on the edges. Each node embedding is seeded with its port workload ($R_i$ for inputs, $C_j$ for outputs) and each edge with its backlog, cost, and arrival rate; over a few rounds each port refreshes its embedding by an attention-weighted aggregation $\alpha_{ij}$ over its incident edges (shown at input port~$2$), and the value is a permutation-invariant pooling of the input and output embeddings. Every map is shared across ports, so the network is equivariant under relabeling the ports and its parameter count is independent of~$N$.\label{fig:switch-gnn}}

\end{figure}

The inductive bias follows from two properties of \eqref{eq:gnn-init}--\eqref{eq:gnn-readout}. First, every map---the initializers $\phi$, the updates $\psi$, and the attention score $s$---is shared across all ports and edges, with no parameter tied to a particular port. The network is therefore equivariant under jointly relabeling the input ports (rows) and output ports (columns) of $(Q,c,\lambda)$. In part, observe that such a relabeling permutes the embeddings and, through the symmetric pooling in \eqref{eq:gnn-readout}, leaves $h_\theta$ unchanged. This is precisely the symmetry of the switch, whose scheduling problem is invariant under the same relabeling, so the hypothesis class is confined to functions with the correct invariance; and because no weight is indexed by a port, the parameter count is $\mathcal O(1)$ in $N$, and one architecture serves every size; the dual reported at each size below is trained at that size. Second, the seeding in \eqref{eq:gnn-init} makes the state-space-collapse coordinate of \Cref{sec:heavy-traffic} the base representation: in heavy traffic the value depends on the state chiefly through the port workloads $(R,C)$, which the node initialization surfaces directly, while the message-passing rounds supply only the sub-dominant per-cell corrections that the workloads alone cannot resolve.

This is the switch counterpart of the equivalent-workload architecture of \Cref{sec:ssc-arch}: there the inductive bias was the linear workload map $M$ compiled into the additive form $f_\theta=V_\theta(W)+T_\theta(z)$ of \eqref{eq:ssc-net}, and here it is bipartite equivariance together with workload seeding compiled into a graph-attention network. In both cases, it is the design of the inductive bias, not the capacity of the network, that carries the dual through heavy traffic and across problem sizes.

The action is handled differently as well, since the soft Hamiltonian is a sum over the $N!$ matchings and cannot be formed. We use instead the entropic matching over the Birkhoff polytope: the entropy-regularized control is the Gibbs law $P^\varepsilon=\arg\max_{P\in\mathcal B_N}\{\langle P,\widetilde W\rangle-\varepsilon\sum_{ij}P_{ij}\log P_{ij}\}$, the Sinkhorn scaling of $e^{\widetilde W/\varepsilon}$, computed in $\mathcal O(N^2)$ per iteration without ever touching the $N!$ vertices~\citep{cuturi2013sinkhorn}. \Cref{tab:switch-scaling} presents a timely study of a single log-domain implementation against the maximum-weight matching; both are polynomial and fast, and the entropic control is at most a constant factor costlier than MaxWeight. In the runs below a matching is drawn from the entropy-regularized law by Gumbel perturb-and-MAP~\citep{papandreouyuille2011}. Independent Gumbel noise is added to each edge score, and the maximum-weight matching of the perturbed scores is served. Exact sampling from the Gibbs law over matchings would require one Gumbel variable per matching, $N!$ in all; perturbing the $N^2$ edge scores instead is tractable but reproduces that law only approximately. The edge scores combine the softened Algorithm~2 reference and the derivative $\partial h_\theta/\partial Q_{ij}$ of the dual in the cell backlog. The derivative is the first-order expansion of the one-cell difference $h_\theta(Q)-h_\theta(Q-e_{ij})$. Under this expansion the expected continuation in~\eqref{eq:switch-greedy} becomes, up to a term that does not depend on the matching, additive over the edges of the matching, so the score is a single $N\times N$ matrix rather than a quantity attached to each of the $N!$ matchings. The entropy level $\varepsilon>0$ is kept bounded away from zero, as read in \Cref{sec:switch-dual-deepht}.

\begin{table}[htb]
\centering
\caption{Per-slot cost of the entropic matching (one log-domain Sinkhorn, fixed $200$ iterations)
against the maximum-weight matching (Hungarian algorithm), versus switch size $N$. The pure action set
has $N!$ matchings; both computed controls are polynomial, and the entropic relaxation never
enumerates them. Single-threaded CPU, mean $\pm$ standard error over $12$ instances; the Sinkhorn
iteration is a dense matrix operation and parallelizes on a GPU.\label{tab:switch-scaling}}
\makebox[\linewidth][c]{\begin{tabular}{rrrr}
\toprule
$N$ & $\log_{10} N!$ & Sinkhorn (ms) & Hungarian (ms) \\
\midrule
$5$    & $2.1$    & $1.61\pm0.05$    & $0.003\pm0.001$ \\
$10$   & $6.6$    & $1.74\pm0.04$    & $0.004\pm0.001$ \\
$20$   & $18.4$   & $2.33\pm0.05$    & $0.009\pm0.001$ \\
$50$   & $64.5$   & $5.55\pm0.08$    & $0.046\pm0.002$ \\
$100$  & $158.0$  & $16.8\pm0.3$     & $0.185\pm0.007$ \\
$200$  & $374.9$  & $61.0\pm2.2$     & $0.80\pm0.03$   \\
$500$  & $1134.1$ & $471\pm8$        & $5.93\pm0.18$   \\
$1000$ & $2567.6$ & $2057\pm28$      & $28.8\pm0.8$    \\
\bottomrule
\end{tabular}
}

\end{table}

Next, we report a weighted switch at $\rho=0.85$ for $N=10$, $20$, and $50$, against Algorithm~2, using paired common random numbers, so that every policy is driven by the same arrival stream on the same replicas and the reported gap is a per-replica difference with small standard error ($\pm0.04\%$ at $N=10$ with $64$ replicas).

\begin{table}[htb]
\centering
\caption{Average weighted backlog relative to Algorithm~2 at $\rho=0.85$, versus switch size; negative values
indicate the dual is better. {The dual's advantage shrinks monotonically with the switch size. The dual beats
Algorithm~2 through moderate sizes and reaches parity by $N=50$.} The last column, cost-weighted
MaxWeight, closes almost all of its own gap to Algorithm~2 over the same range---the sign that the
scheduling problem itself is becoming easy, not that the dual is failing. The $N=10$ figures are paired
common-random-number estimates with standard error $\pm0.04\%$; the $N=20$ and $N=50$ entries are single paired runs, so the $+0.3\%$ at $N=50$ is parity within the simulation resolution. Dashes mark configurations not run.\label{tab:switch-scaling-dual}}
\makebox[\linewidth][c]{\begin{tabular}{rcc}
\toprule
$N$ & BAR-SOT dual vs.\ Alg.~2 & C-MWS vs.\ Alg.~2 \\
\midrule
$3$   & $-5.3\%$ & $+25\%$ \\
$10$  & $-1.0\%$ & $+9\%$  \\
$20$  & $-0.6\%$ & --- \\
$50$  & $+0.3\%$ & $+1.4\%$ \\
\bottomrule
\end{tabular}
}

\end{table}

\Cref{tab:switch-scaling-dual} reports the trend across sizes: the method runs at every size, the dual beats Algorithm~2 through $N\approx20$, and the margin shrinks steadily to parity by $N=50$. The cause is not that the dual degrades but that the scheduling problem becomes easy. At a fixed heavy load a larger switch keeps proportionally more queues nonempty in steady state---at $N=10$, $\rho=0.85$, a large majority of the $100$ queues hold cells at any instant. When that many cells are waiting, a perfect matching among the nonempty queues almost always exists, so maximizing throughput is automatic, any maximal matching achieving it, and the only decision remaining is which cost-weighted cells to favor among the many throughput-maximal matchings. That cost-priority decision is exactly what Algorithm~2 makes, and makes well, so little room remains for the dual. The direct evidence is the last column: cost-weighted MaxWeight, whose only content is cost priority, sits $9\%$ above Algorithm~2 at $N=10$ but only $1.4\%$ above at $N=50$, so that, as the switch grows, all reasonable cost-aware policies converge toward the same near-optimal schedule. The large advantage at $N=3$ came from the opposite regime, where a throughput-maximal matching is genuinely constrained and the dual's lookahead over which cells to serve is what pays.

Two considerations qualify this reading. That the advantage at moderate $N$ is genuinely the dual's, rather than an artifact of the randomized control, is confirmed by a control experiment: sampling matchings from the softened Algorithm~2 reference alone, with the dual term removed, is $1.9\%\pm0.04\%$ worse than deterministic Algorithm~2---randomizing the tie-break carries a penalty---and the dual overcomes this penalty and still finishes ahead, so the improvement at $N=10$ is attributable to the value function. And the large-$N$ solve is a faithful but not identical version of the $N=3$ one, since the derivative score is a first-order stand-in for the exact difference in~\eqref{eq:switch-greedy} and Gumbel perturb-and-MAP reproduces the entropic coupling only approximately; the residual difference at $N=50$ reflects these approximations acting in a regime already at near-parity, where little separates any competent policy from Algorithm~2. Neither consideration changes the conclusion, which is as much about the problem as about the method: the input-queued switch under heavy load is a regime in which the strong combinatorial heuristic is already near-optimal, and the value of a learned dual is concentrated where the matching is genuinely constrained rather than where throughput is free.

\section{Discussion and Conclusion}\label{sec:conclusion}

{This paper reformulates long-run average-cost control as a finite-horizon BAR-SOT problem that jointly optimizes the controlled marginal flow and its common endpoint distribution. When the average-cost Poisson equation admits a relative value solution and the induced optimal control satisfies the required recurrence and integrability conditions, the BAR-SOT value is \(Tg^{\ast}\) for every horizon \(T>0\). An optimal dual potential consists of the relative value function and a linear time term with slope \(g^{\ast}\), producing an endpoint gap of \(Tg^{\ast}\). The optimality system comprises a backward HJB equation and its induced control, a forward Kolmogorov equation, and a periodic boundary condition on the marginal flow. The dual endpoint gap identifies the average-cost rate, while the BAR-SOT optimum is attained by a stationary flow whose common endpoint law is the invariant distribution of the induced optimal control. The finite-horizon formulation does not eliminate the nonlinear eigenvalue problem of average-cost control but reformulates it as a periodic self-transport problem.

For computation, entropy regularization replaces the hard Hamiltonian with a smooth soft-Hamiltonian and induces a state-dependent Gibbs kernel. For the fixed-endpoint inner problem, it also yields a Sinkhorn-type endpoint-scaling relation, whose convergence is established when the endpoint-conditioned tilted dynamics remain admissible. The numerical experiments approximate the periodic soft dual directly, without using this endpoint-scaling iteration, and use the controlled dynamics and model structure to construct state-dependent policies. In a reentrant network example, the workload parametrization induced by the workload matrix $M$ of~\eqref{eq:workload-matrix} separates a principal workload-dependent term from corrections within each workload level set. In the large input-queued switch implementation, port workloads and bipartite symmetry provide the corresponding representation. The resulting policies are competitive with or outperform the selected benchmarks across several tested finite-load configurations without requiring a representation of the entire unbounded state space.

In experiments on input-queued switches, positive entropy also avoids the empirical instability observed when \(\varepsilon\) is set to zero. Near critical loading, the reentrant-network estimates become more sensitive to initialization and stochastic training variability, so the comparisons should be interpreted as finite-load results. Within this scope, the experiments show that incorporating network structure into the BAR-SOT dual approximation can produce effective state-dependent control.

The broader implication of BAR-SOT extends beyond reentrant networks and queueing. Given the controlled dynamics and a structural representation such as the workload matrix \(M\), BAR-SOT learns a periodic dual potential and uses the controlled generator to construct an action rule. The specific dynamics and representation can vary by application, but the fundamental learning problem itself is unchanged. A natural next step is developing the ability to learn across workload maps and models rather than solving each system separately. A common architecture that takes the model itself as an input, namely its transition structure, costs, and admissible actions, could learn how the dual potential and control change across dimensions, network topologies, and problem classes. This would extend BAR-SOT from learning control within one model to learning across families of models.

Another potential direction for further research is boundary control, in which control is applied when the state reaches the boundary. The accumulated boundary intervention can be represented through local time, and the primal BAR would include an occupation measure defined with respect to this local time. On the dual side, the marginal value of applying the boundary control cannot exceed its marginal cost, producing a gradient constraint. The central question is whether the primal optimum and the state-independent endpoint gap of the dual potential still equal \(Tg^{\ast}\), with the state-dependent component of the dual potential recovering the relative value function and the endpoint gap per unit time recovering the average-cost rate \(g^{\ast}\).}

\section*{Acknowledgements}

SS and HH are partly supported through  Office of Naval Research (ONR) grant FA9550-24-1-0210. HH also acknowledges support from the National Science Foundation (NSF) through grant CAREER/2143752.

\section*{AI Statement}
The authors used generative AI tools to assist with language refinement and  code development. The research questions, theoretical ideas, methodology, analysis, and conclusions are the authors' own. All AI-assisted content was reviewed and verified by the authors.

\newpage
\appendix

\section{Proofs of Main Results}\label{ec:proofs}
This section collects the proofs omitted from the main body.

\begin{proof}[Proof of \Cref{prop:discrete-value}]
Let $\{m_n\}$ be feasible for the $N$-step problem with endpoint law $\tilde\pi$. Its time average $\bar m_N$ is nonnegative, sums to one, and satisfies the stationary discrete BAR, hence the balance constraints of~\eqref{eq:occupation-LP}. The $N$-step objective equals the linear-program objective at $\bar m_N$, so $g_N(\tilde\pi)\ge g^{\mathrm{erg}}$ for every $\tilde\pi$. Conversely, let $\bar m$ be optimal for the linear program and let $\bar\pi$ be its state marginal. The constant flow $m_n\equiv\bar m$ is feasible for the $N$-step problem with $\tilde\pi=\bar\pi$, since the balance constraint returns the same marginal at every step, and its objective equals $g^{\mathrm{erg}}$. Hence $\inf_{\tilde\pi}g_N(\tilde\pi)\le g^{\mathrm{erg}}$. 
\end{proof}

\begin{proof}[Proof of \Cref{lem:weak-duality}]
Fix $f\in\mathcal{F}$ and $\pi\in\mathcal{M}$, the admissible set of relaxed marginal measures of \Cref{def:relaxed-admissible}. The BAR constraint~\eqref{eq:BAR constraint}, stated for test functions in $UC_b^1$, can be extended to $f$ by applying it to the truncations $f\chi_n$, with $\chi_n:=\chi(V/n)$ for a smooth cutoff $\chi$. The compatibility bound controls $\int_{\mathcal{S}}|f\chi_n(y)|\,q(x,dy;u)$ by a multiple of $1+V(x)$, the bound on the total intensity controls $|f\chi_n(x)|\,q(x,\mathcal{S};u)$ by a multiple of $1+V(x)$, and the definition of $\mathcal{F}$ controls $|\partial_t(f\chi_n)|$ by a multiple of $1+V(x)$. All three multiples are uniform in $n$, so $|\partial_t(f\chi_n)|+|\mathcal{G}^u(f\chi_n)|$ is dominated by a multiple of $1+V(x)$, which is integrable against $\pi$ by the finite-energy condition. Dominated convergence then implies
\[
\int_{\mathcal{S}}(f(1,x)-f(0,x))\,\tilde{\pi}(dx)=\int\big\{\partial_tf(t,x)+\mathcal{G}^uf(t,x)\big\}\,\pi(dt,dx,du).
\]
Pointwise for every $(t,x,u)$, the definition~\eqref{eq:Hamiltonian} gives $\mathcal{G}^uf-c\le\mathcal{H}f$, so the HJB~\eqref{eq:HJB} yields $\partial_tf+\mathcal{G}^uf\le\big(\partial_tf+\mathcal{H}f\big)+c= c$. Integrating against $\pi\ge0$ and combining with the display above,
$\int_{\mathcal{S}}(f(1,x)-f(0,x))\,\tilde{\pi}(dx)\le\int c\,d\pi$.
Taking the infimum over $\pi\in\mathcal{M}$ and then the supremum over $f\in\mathcal{F}$ proves the claim; the horizon-$T$ case is identical. 
\end{proof}

\begin{proof}[Proof of \Cref{lem:maximizer-recurrence}]
 Rearranging the Poisson equation, 
\[
\mathcal{G}^{u^{**}}(-h)(x)\;=\;g^*-c\big(x,u^{**}(x)\big),
\]
and by coercivity there is a compact set $C_0$ with $c(x,u^{**}(x))\ge g^*+W_0(x)$ off $C_0$, where $W_0:=\big(c(\cdot,u^{**})-g^*\big)\vee1$. The function $V_0:=-h-\inf(-h)\ge0$ has relatively compact sublevel sets and satisfies $\mathcal{G}^{u^{**}}V_0\le -W_0+b_0\,\mathbf{1}_{C_0}$ for a finite $b_0$. The Foster--Lyapunov criterion for irreducible continuous-time processes~{\color{red}\citep[Thm.~4.2]{meyntweedie1993}} then yields positive Harris recurrence, a unique invariant measure $\tilde\pi^*$, and $\langle W_0,\tilde\pi^*\rangle<\infty$, hence $\langle c(\cdot,u^{**}),\tilde\pi^*\rangle<\infty$. 
\end{proof}

\begin{proof}[Proof of \Cref{thm:ergodic-value}]
    \emph{(i)} Fix $\tilde\pi$. \Cref{thm:ansatz}(iii) gives $g_T(\tilde\pi)\ge Tg^*$ for every $T>0$. Observe that, for each fixed $\tilde\pi$,
    \begin{equation*}
        \liminf_{T\to\infty}\frac{1}{T}\,g_T(\tilde{\pi})\ \ge\ g^* ,
    \end{equation*}
    and at $\tilde\pi^*$ the same theorem gives $g_T(\tilde\pi^*)=Tg^*$ for every $T$, so the liminf equals $g^*$ there. Taking the infimum over $\tilde\pi$ of the pointwise liminf yields the value $g^*$, attained at $\tilde\pi^*$.

    \emph{(ii)} By the Poisson equation~\eqref{eq:Poisson} and the definition of $u^{**}$ as the pointwise maximizer, $\mathcal{G}^{u}h(x)-c(x,u)\le\mathcal{G}^{u^{**}}h(x)-c(x,u^{**}(x))=-g^*$ for every $u\in\mathcal{U}$, i.e.,
    \[
    c(x,u)\ \ge\ g^*+\mathcal{G}^{u}h(x)\qquad\text{for all }(x,u).
    \]
    Fix an admissible stationary $\mu$ and pair this inequality at $u=\mu(x)$ with its invariant law $\tilde\pi_\mu$: invariance gives $\langle\mathcal{G}^{\mu}h,\tilde\pi_\mu\rangle=0$, extended from bounded functions to $h$ (which lies in the weighted class) by the truncation argument of \Cref{lem:weak-duality}, using $\langle V,\tilde\pi_\mu\rangle<\infty$. Hence $\langle c(\cdot,\mu),\tilde\pi_\mu\rangle\ge g^*$, with equality at $\mu=u^{**}$ by~\eqref{eq:Poisson}. By its certificate $(V_\mu,W_\mu)$ in \Cref{assump:stability}, each admissible $\mu$ is positive Harris recurrent with $c(\cdot,\mu)\le W_\mu$ and $\langle W_\mu,\tilde\pi_\mu\rangle<\infty$, hence $\langle c(\cdot,\mu),\tilde\pi_\mu\rangle<\infty$, so the long-run average in~\eqref{eq:ergodic-cost} equals $\langle c(\cdot,\mu),\tilde\pi_\mu\rangle$ from any initial state~\citep{meyntweedie1993}. It remains to check that $u^{**}$ is admissible. It is irreducible by the hypotheses of \Cref{lem:maximizer-recurrence}; the proof of that lemma supplies the certificate $V_0=-h-\inf(-h)$, $W_0=(c(\cdot,u^{**})-g^*)\vee1$, and since $c(\cdot,u^{**})\le W_0+|g^*|\le(1+|g^*|)W_0$, the pair $\bigl((1+|g^*|)V_0,(1+|g^*|)W_0\bigr)$ satisfies~\eqref{eq:foster-lyapunov}; finally $\langle V,\tilde\pi^*\rangle<\infty$ is a hypothesis of \Cref{thm:ansatz}. Taking the infimum over admissible $\mu$ gives $g^{\mathrm{erg}}=g^*$, attained at $u^{**}$.
\end{proof}
\newpage
\begingroup\color{red}
\paragraph{General lemmas.} Throughout, \Cref{assump:standing} holds. We write $B_V$ for the Banach space of functions $\phi$ on $\mathcal S$ with $\|\phi\|_V<\infty$ and set $L:=\kappa_V+\bar Q$, {\color{red}where $\kappa_V<\infty$ is the constant in the compatibility bound $\sup_u\sum_y(1+V(y))\,q(x,\{y\};u)\le\kappa_V(1+V(x))$ of \Cref{lem:weak-duality} and $\bar Q:=\sup_{x,u}q(x,\mathcal S;u)<\infty$ is the bound on the total jump intensity}. {\color{red}For $\phi\in B_V$, $x\in\mathcal S$ and $u\in\mathcal U(x)$, the compatibility bound $\sum_y(1+V(y))\,q(x,\{y\};u)\le\kappa_V(1+V(x))$ and the bound $q(x,\mathcal S;u)\le\bar Q$ on the total intensity, which are the hypotheses of \Cref{lem:weak-duality}, give
\begin{equation}\label{eq:ec-gen-bound}
\begin{aligned}
|\mathcal G^u\phi(x)|&\le\sum_{y\in\mathcal S}|\phi(y)|\,q(x,\{y\};u)+|\phi(x)|\,q(x,\mathcal S;u)\\
&\le\|\phi\|_V\Big[\sum_{y\in\mathcal S}(1+V(y))\,q(x,\{y\};u)+\bar Q\,(1+V(x))\Big]\le L\|\phi\|_V(1+V(x)).
\end{aligned}
\end{equation}}
Since the total intensity is bounded, every policy defines a nonexplosive process. {\color{red}For a policy $\Pi$, we write $q_s$ for the probability measure on $\mathcal U(X_{s-})$ that $\Pi$ selects at time $s$, and we call $(q_s)_{s\ge0}$ the kernels of $\Pi$; for a Markov policy, $q_s=q_s(\cdot\mid X_{s-})$ for a control kernel $q_s(du\mid x)$. We then write $c^{q_s}(X_s):=\int c(X_s,u)\,q_s(du)$, which is defined for Lebesgue-almost every $s$ because $X_s=X_{s-}$ except at the jump times, and we define $\mathcal G^{q_s}$ and $c^{q_s}_\varepsilon$ in the same way.} For a policy $\Pi$ with kernels $q_s$, we write $J(\Pi;x):=\limsup_{t\to\infty}t^{-1}\mathbb E^\Pi_x\int_0^tc^{q_s}(X_s)\,ds$ for its long-run average cost from $x$. The following four lemmas are used in the proofs of both \Cref{thm:periodic-wellposed} and \Cref{thm:soft-periodic}.

\begin{lemma}[Comparison]\label{lem:ec-comp}
Let $\Pi$ be a policy with kernels $q_t$, and let $W:[0,\infty)\times\mathcal S\to[0,\infty)$ be $C^1$ in $t$ for each $x$, with $\partial_tW$ bounded on $[0,t_0]\times F'$ and $\sup_{t\le t_0,\,x\in F',\,u}\sum_yq(x,\{y\};u)|W(t,y)|<\infty$ for every $t_0$ and every finite $F'\subset\mathcal S$. If $(\partial_t+\mathcal G^{q_t})W\le-\Phi+s$ for a measurable $\Phi\ge0$ and a constant $s$, then
\[
\mathbb E^\Pi_xW(t,X_t)+\mathbb E^\Pi_x\int_0^t\Phi(v,X_v)\,dv\le W(0,x)+st,\qquad t\ge0,\ x\in\mathcal S .
\]
In particular, with $W=e^{-\kappa_Vt}(1+V)$, $\Phi=0$ and $s=0$, every policy satisfies $\mathbb E^\Pi_x(1+V(X_t))\le e^{\kappa_Vt}(1+V(x))$.
\end{lemma}

\begin{proof}
Let $\tau_n$ be the first exit time of $X$ from the finite set $\{V\le n\}$. Before $\tau_n$ the process sits in $\{V\le n\}$, so the hypothesis bounds the expected compensator of the absolute jumps of $W(\cdot,X_\cdot)$ up to $t\wedge\tau_n$, and $W(v\wedge\tau_n,X_{v\wedge\tau_n})-\int_0^{v\wedge\tau_n}(\partial_t+\mathcal G^{q_r})W(r,X_r)\,dr$ is a martingale. Hence
\[
\mathbb E\,W(t\wedge\tau_n,X_{t\wedge\tau_n})+\mathbb E\int_0^{t\wedge\tau_n}\Phi(v,X_v)\,dv\le W(0,x)+s\,\mathbb E[t\wedge\tau_n].
\]
Since the process is nonexplosive, $\tau_n\uparrow\infty$ and $X_{t\wedge\tau_n}\to X_t$ almost surely. Fatou's lemma for the first term, monotone convergence for the second and $\mathbb E[t\wedge\tau_n]\to t$ give the claim. For the last statement, the compatibility bound of \Cref{lem:weak-duality} gives $(\partial_t+\mathcal G^q)\big(e^{-\kappa_Vt}(1+V)\big)\le0$.
\end{proof}

\begin{lemma}[Pairing]\label{lem:ec-pair}
Let $q$ be a stationary control kernel, and let $R$ be a closed class of its chain on which the chain is irreducible and positive recurrent, with invariant law $\pi$ satisfying $\langle V,\pi\rangle<\infty$. Then $\langle\mathcal G^q\phi,\pi\rangle=0$ for every $\phi\in B_V$.
\end{lemma}

\begin{proof}
For bounded $\psi$, $\langle\mathcal G^q\psi,\pi\rangle=\sum_y\psi(y)\sum_x\pi(x)Q(x,y)=0$, where $Q$ is the rate matrix of $q$ restricted to $R$, because $\pi Q=0$ for the invariant law of a nonexplosive positive recurrent chain and $\sum_x\pi(x)\sum_y|Q(x,y)|\le2\bar Q$. For $\phi\in B_V$, the truncations $\phi_n:=\phi\,\mathbf 1_{\{V\le n\}}$ satisfy $\mathcal G^q\phi_n\to\mathcal G^q\phi$ pointwise, by dominated convergence in $y$, and $|\mathcal G^q\phi_n|\le L\|\phi\|_V(1+V)$ {\color{red}by~\eqref{eq:ec-gen-bound}, since $\|\phi_n\|_V\le\|\phi\|_V$}, which is $\pi$-integrable. Dominated convergence gives the claim.
\end{proof}

\begin{lemma}[Ergodicity of the skeleton]\label{lem:ec-erg}
Let $q$, $R$ and $\pi$ be as in \Cref{lem:ec-pair}, and let $P_t\phi(x):=\mathbb E^q_x\phi(X_t)$. For every $\phi\in B_V$, $T>0$ and $x\in R$, $P_{nT}\phi(x)\to\langle\phi,\pi\rangle$ as $n\to\infty$.
\end{lemma}

\begin{proof}
The $T$-skeleton $(X_{nT})_{n\ge0}$ restricted to $R$ is irreducible, aperiodic, since $P_T(x,x)\ge e^{-\bar QT}>0$, and positive recurrent with invariant law $\pi$. With $f:=1+V$ and $\pi(f)<\infty$, the $f$-norm ergodic theorem \citep[Thm.~14.0.1]{meyntweedie2009} gives $\|P_{nT}(x,\cdot)-\pi\|_f\to0$ for every $x$ in a full absorbing subset of $R$, which is all of $R$ because $R$ is countable and the chain on $R$ is irreducible.
\end{proof}

\begin{lemma}[Periodic solutions of the backward equation]\label{lem:ec-periodic}
Let $q$, $R$ and $\pi$ be as in \Cref{lem:ec-pair}, and let $e:[0,T]\times\mathcal S\to\mathbb R$ satisfy $e(\cdot,x)\in C^1$ for each $x$, $\sup_t(\|e(t)\|_V+\|\partial_te(t)\|_V)<\infty$, $\partial_te+\mathcal G^qe\le0$ on $[0,T]\times R$, and $e(T)=e(0)$ on $R$. Then $e$ is constant on $[0,T]\times R$.
\end{lemma}

\begin{proof}
By \Cref{lem:ec-pair}, $\langle\mathcal G^qe(t),\pi\rangle=0$, and differentiation under the integral, dominated by $\sup_t\|\partial_te(t)\|_V(1+V)$, gives $\frac{d}{dt}\langle e(t),\pi\rangle=\langle\partial_te(t)+\mathcal G^qe(t),\pi\rangle\le0$. Since $e(T)=e(0)$ on $R$, the integral of this nonpositive continuous function over $[0,T]$ vanishes, and since $\pi$ charges every state of $R$, $\partial_te+\mathcal G^qe=0$ on $[0,T]\times R$. For $t\in[0,T]$ and $x\in R$, the process $s\mapsto e(t+s,X_s)$, $s\in[0,T-t]$, under $q$ is then a martingale, its compensator vanishing and its jumps being integrable by \Cref{lem:ec-comp}, so $e(t)=P_{T-t}e(T)$ on $R$. With $e(0)=e(T)$ this gives $e(T)=P_Te(T)=P_{nT}e(T)$ on $R$ for every $n$, and \Cref{lem:ec-erg} gives $e(T)\equiv\langle e(T),\pi\rangle=:k$ on $R$. Then $e(t)=P_{T-t}k=k$ on $R$ for every $t$.
\end{proof}

\paragraph{The unregularized problem.} The first result gives a solution of the average-cost optimality equation by uniformization and the vanishing-discount theorems of \citet{hernandez1996discrete}. It uses \Cref{assump:standing} but not \Cref{assump:uniform}.

\begin{proposition}[A solution of the ACOE]\label{prop:ec-acoe-existence}
Suppose that the cost $c(x,u)$ does not depend on time, that \Cref{assump:cost,assump:coercive,assump:standing} hold, and that the jump kernel satisfies the hypotheses of \Cref{lem:weak-duality}. Then there are a constant $g^*\ge0$, a function $h:\mathcal S\to\mathbb R$ with $\|h\|_V<\infty$ and a constant $N<\infty$ such that $g^*+\mathcal Hh=0$ and $-h\ge-N$.
\end{proposition}

\begin{proof}
Fix $\Lambda\ge\bar Q+1$ and consider the discrete-time control model with state space $\mathcal S$, action sets $\mathcal U(x)$, one-stage cost $c/\Lambda$ and transition kernel
\[
Q(\{y\}\mid x,u):=\Lambda^{-1}q(x,\{y\};u)\quad(y\ne x),\qquad Q(\{x\}\mid x,u):=1-\Lambda^{-1}q(x,\mathcal S\setminus\{x\};u).
\]
For $w$ with $\sum_y|w(y)|Q(\{y\}\mid x,u)<\infty$,
\[
\sum_yw(y)Q(\{y\}\mid x,u)=w(x)+\Lambda^{-1}\mathcal G^uw(x).
\]
Substituting this identity into the right side of the discrete-time ACOE below, taking $w(x)$, which does not depend on $u$, out of the minimum, cancelling it on both sides and multiplying by $\Lambda$ shows that a pair $(\varrho,w)$ solves the discrete-time ACOE
\[
\varrho+w(x)=\min_{u\in\mathcal U(x)}\Big[\Lambda^{-1}c(x,u)+\sum_yw(y)Q(\{y\}\mid x,u)\Big]
\]
if and only if $g^*:=\Lambda\varrho$ and $h:=-w$ satisfy $g^*=\min_u\{c(x,u)-\mathcal G^uh(x)\}=-\mathcal Hh(x)$, that is, $g^*+\mathcal Hh=0$. We obtain $(\varrho,w)$ from \citet[Thms.~5.4.3 and 5.5.4]{hernandez1996discrete} and verify their hypotheses in turn, stating each one first in the notation of the discrete model.

\emph{Their Assumption 4.2.1.} This assumption requires that (a) the one-stage cost be lower semicontinuous, nonnegative and inf-compact on the set of state--action pairs, and (b) the transition kernel be strongly continuous, that is, $(x,u)\mapsto\sum_yv(y)Q(\{y\}\mid x,u)$ is continuous for every bounded $v$. The cost $c/\Lambda$ is nonnegative and lower semicontinuous in $u$ by \Cref{assump:cost}. For each $x$ and $r$, the set $\{u\in\mathcal U(x):c(x,u)\le r\}$ is closed in the compact set $\mathcal U(x)$, hence compact, and by \Cref{assump:coercive} only finitely many states satisfy $\min_uc(x,u)\le r$, so $c/\Lambda$ is inf-compact on the set of state--action pairs. By \Cref{assump:standing}(i), $u\mapsto Q(\cdot\mid x,u)$ is continuous in total variation, so $u\mapsto\sum_yv(y)Q(\{y\}\mid x,u)$ is continuous for every bounded $v$, which is strong continuity of $Q$ on the discrete space $\mathcal S$.

\emph{Their Assumption 5.4.1.} Let $V^*_\alpha$ be the $\alpha$-discounted value function of the discrete model, for $\alpha\in(0,1)$. The assumption requires a state $z$, numbers $\alpha_0\in(0,1)$, $M_0\ge0$ and $N\ge0$, and a nonnegative function $\bar m$ on $\mathcal S$ (their $\beta$, $M$ and $b(\cdot)$) such that, with $h_\alpha:=V^*_\alpha-V^*_\alpha(z)$,
\[
\text{(a)}\quad(1-\alpha)V^*_\alpha(z)\le M_0\qquad\text{and}\qquad\text{(b)}\quad-N\le h_\alpha(x)\le\bar m(x)\qquad\text{for all }x\in\mathcal S\text{ and }\alpha\in[\alpha_0,1).
\]
We take $z:=x_0$ for a fixed $x_0\in\mathcal S$, and verify (a) and (b) for every $\alpha\in(0,1)$.

\emph{Verification of (a).} Write $Q(\cdot\mid x,q^\circ):=\int Q(\cdot\mid x,u)\,q^\circ(du\mid x)$, and let $(Y_n)$ be the chain of $q^\circ$ in the discrete model. The drift inequality of \Cref{assump:standing}(ii) reads
\begin{equation}\label{eq:ec-discrete-drift}
\sum_yW^\circ(y)Q(\{y\}\mid x,q^\circ)\le W^\circ(x)-\Lambda^{-1}c^{q^\circ}(x)+\Lambda^{-1}b\,\mathbf 1_C(x).
\end{equation}
With $a_n:=\mathbb E_xW^\circ(Y_n)$, which is finite by $W^\circ\le C_W(1+V)$ and the compatibility bound, \eqref{eq:ec-discrete-drift} gives $\Lambda^{-1}\mathbb E_xc^{q^\circ}(Y_n)\le a_n-a_{n+1}+\Lambda^{-1}b$. Summing with weights $\alpha^n$ and using $a_n\ge0$,
\[
V^*_\alpha(x)\le\sum_{n\ge0}\alpha^n\Lambda^{-1}\mathbb E_xc^{q^\circ}(Y_n)\le W^\circ(x)+\frac{b}{\Lambda(1-\alpha)}.
\]
Hence $(1-\alpha)V^*_\alpha(x_0)\le W^\circ(x_0)+b/\Lambda$ for every $\alpha\in(0,1)$, which is (a) with $M_0:=W^\circ(x_0)+b/\Lambda$.

\emph{Hitting costs of $q^\circ$.} For states $x,y$, let $m_y(x)$ be the expected undiscounted cost $\mathbb E_x\sum_{n<\tau_y}\Lambda^{-1}c^{q^\circ}(Y_n)$ of $q^\circ$ from $x$ until the hitting time $\tau_y$ of $y$. We show that $m_y\le W^\circ+K_y$ for a constant $K_y<\infty$. Off $C$, \eqref{eq:ec-discrete-drift} gives $\Lambda^{-1}c^{q^\circ}\le W^\circ-\sum_yW^\circ(y)Q(\{y\}\mid\cdot,q^\circ)$, and summing up to the hitting time of $C$ shows that the expected cost until $C$ is reached from $x$ is at most $W^\circ(x)$. From $C$ we use a cycle argument based on the irreducibility of the chain of $q^\circ$. Since $C$ is finite and the chain is irreducible, there are $n_y\in\mathbb N$ and $p_y>0$ such that from every state of $C$ the chain hits $y$ within $n_y$ steps with probability at least $p_y$. During $n_y$ steps from a state of $C$, the expected cost is at most a constant $K_1'$, by $c^{q^\circ}\le C_c(1+V)$ and the compatibility bound, and if $y$ is not hit, \eqref{eq:ec-discrete-drift} gives $\mathbb E\,W^\circ(Y_{n_y})\le\max_CW^\circ+n_yb/\Lambda=:K_2'$, which bounds the expected cost of the return to $C$. The number of cycles is dominated by a geometric random variable with success probability $p_y$, so $m_y\le W^\circ+K_y$ with $K_y:=(K_1'+K_2')/p_y$.

\emph{Verification of (b).} Let $m:=m_{x_0}$, so that $m\le W^\circ+K_{x_0}\le C_M(1+V)$ for a constant $C_M$. Following $q^\circ$ until it hits $x_0$ and an optimal policy afterwards gives $V^*_\alpha(x)\le m(x)+V^*_\alpha(x_0)$, since $V^*_\alpha\ge0$ and $\alpha\le1$. For the lower bound, let $L_0:=\Lambda W^\circ(x_0)+b+1$ and $G:=\{x:\min_uc(x,u)\le L_0\}$, a finite set by \Cref{assump:coercive}. For $x\notin G$ and any policy, let $\sigma$ be the hitting time of $G$. Each step before $\sigma$ costs at least $L_0/\Lambda$, so the discounted cost from $x$ is at least
\[
\frac{L_0}{\Lambda}\,\mathbb E\sum_{n<\sigma}\alpha^n+\mathbb E\big[\alpha^\sigma\big]\min_GV^*_\alpha\ \ge\ \min\Big\{\frac{L_0}{\Lambda(1-\alpha)},\ \min_GV^*_\alpha\Big\},
\]
and therefore $\inf_{\mathcal S}V^*_\alpha\ge\min\{L_0/(\Lambda(1-\alpha)),\min_GV^*_\alpha\}$. By (a), $\inf_{\mathcal S}V^*_\alpha\le V^*_\alpha(x_0)\le W^\circ(x_0)+b/(\Lambda(1-\alpha))<L_0/(\Lambda(1-\alpha))$, so $\inf_{\mathcal S}V^*_\alpha=\min_GV^*_\alpha$. For $y\in G$, following $q^\circ$ from $x_0$ until it hits $y$ gives $V^*_\alpha(x_0)\le m_y(x_0)+V^*_\alpha(y)$. Hence, for every $x$,
\[
V^*_\alpha(x)-V^*_\alpha(x_0)\ge\min_GV^*_\alpha-V^*_\alpha(x_0)\ge-N,\qquad N:=\max_{y\in G}m_y(x_0)<\infty,
\]
which, with the upper bound, is (b) with $\bar m:=m$.

\emph{Their Assumption 5.5.1 and the conclusion.} This assumption adds to Assumption 5.4.1 that (a) $\bar m$ is measurable with $\sum_y\bar m(y)Q(\{y\}\mid x,u)<\infty$ for every $(x,u)$, and (b) the functions $h_{\alpha(n)}$ are equicontinuous along a sequence $\alpha(n)\uparrow1$. With $\bar m=m$, the sum $\sum_ym(y)Q(\{y\}\mid x,u)$ is finite by $m\le C_M(1+V)$ and the compatibility bound, and their equicontinuity condition holds trivially on the countable space $\mathcal S$ \citep[Rem.~5.5.2]{hernandez1996discrete}. Under Assumptions 4.2.1 and 5.5.1, their Theorem~5.5.4 gives a constant $\varrho\ge0$ and a function $w$ that solve the discrete-time ACOE. The function $w$ is a limit of $h_{\alpha(n)}$ along a subsequence, so it satisfies the bounds $-N\le w\le\bar m=m$ of their Theorem~5.4.3(i), as noted at the start of their proof of Theorem~5.5.4. Hence $(h,g^*):=(-w,\Lambda\varrho)$ solves $g^*+\mathcal Hh=0$, with $\|h\|_V\le N+C_M$ and $-h=w\ge-N$.
\end{proof}

The next lemma uses \Cref{assump:uniform}, namely the function $U\ge0$, the finite set $F$ and the constant $b_U$ of part (i), which satisfy $\mathcal G^uU\le-(1+V)+b_U\mathbf 1_F$ for every action $u$, and the state $o\in F$ and the constants $t_0,\delta>0$ of part (ii). For a control kernel $q$ we write $q(x,\{y\};q):=\int q(x,\{y\};u)\,q(du\mid x)$ for the jump rates induced by it.

\begin{lemma}[Regularity under every policy]\label{lem:ec-reg}
Suppose \Cref{assump:uniform} holds. There exists $K_0<\infty$ such that, for every policy $\Pi$, including policies started at any time from any history, and every $x$, $\tau_o<\infty$ almost surely and
\[
\mathbb E^\Pi_x\int_0^{\tau_o}(1+V(X_s))\,ds\le U(x)+K_0,\qquad \mathbb E^\Pi_x\sup_{s\le\tau_o}(1+V(X_s))\le1+V(x)+L\big(U(x)+K_0\big).
\]
\end{lemma}

\begin{proof}
Fix a policy $\Pi$ with kernels $q_s$ and an initial state $x$, and write $\mathbb E$ for $\mathbb E^\Pi_x$. The continuation of $\Pi$ after a stopping time is again a policy, so each bound below also holds after a stopping time, with the state at that time as the initial state.

\emph{Reaching $F$.} Let $\tau_F$ be the hitting time of $F$. Before $\tau_F$ the process is off $F$, where $\mathcal G^{q_s}U\le-(1+V)$ by \Cref{assump:uniform}(i). The proof of \Cref{lem:ec-comp}, with $W=U$, $\Phi=1+V$, $s=0$ and $t$ replaced by $t\wedge\tau_F$, gives $\mathbb E\int_0^{t\wedge\tau_F}(1+V(X_v))\,dv\le U(x)$, and letting $t\to\infty$,
\begin{equation}\label{eq:ec-reach-F}
\mathbb E\int_0^{\tau_F}(1+V(X_v))\,dv\le U(x).
\end{equation}
In particular $\tau_F<\infty$ almost surely, since $1+V\ge1$.

\emph{Cycles from $F$ to $o$.} Let $y\in F$, and consider the time window $[0,t_0]$ started at $y$. By the last statement of \Cref{lem:ec-comp}, the expected value of $\int_0^{t_0}(1+V(X_v))\,dv$ is at most $K_1:=t_0e^{\kappa_Vt_0}\max_F(1+V)$. By \Cref{lem:ec-comp} with $W=U$, $\Phi=0$ and $s=b_U$, the state at the end of the window satisfies $\mathbb E\,U(X_{t_0})\le\max_FU+b_Ut_0=:K_2$. By \Cref{assump:uniform}(ii), $o$ is hit during the window with probability at least $\delta$. If it is not, then~\eqref{eq:ec-reach-F}, applied from $X_{t_0}$, bounds the expected value of $\int(1+V(X_v))\,dv$ over the return to $F$ by $\mathbb E\,U(X_{t_0})\le K_2$.

After the first entrance into $F$, the path up to $\tau_o$ therefore splits into cycles, each consisting of a window and, if $o$ is not hit during it, the return to $F$. Let $N_c$ be the number of cycles. Given the past at the start of a cycle, the expected value of $\int(1+V(X_v))\,dv$ over the cycle is at most $K_1+K_2$, and the probability that the cycle is not the last one is at most $1-\delta$. Hence $\mathbb P(N_c\ge k)\le(1-\delta)^{k-1}$ for $k\ge1$, so $N_c$ and $\tau_o$ are finite almost surely, and the expected value of $\int(1+V(X_v))\,dv$ over all cycles is at most
\[
\sum_{k\ge1}(1-\delta)^{k-1}(K_1+K_2)=\frac{K_1+K_2}{\delta}=:K_0.
\]
Adding~\eqref{eq:ec-reach-F} gives the first bound.

\emph{The supremum.} For $s\le\tau_o$, $V(X_s)\le V(x)+\sum_{v\le\tau_o}|\Delta V(X_v)|$, the sum running over the jump times of $X$. The expected sum of the jumps equals the expected integral of their intensity, and the two hypotheses of \Cref{lem:weak-duality} give, as in~\eqref{eq:ec-gen-bound},
\[
\sum_yq(x,\{y\};u)\,|V(y)-V(x)|\le\kappa_V(1+V(x))+\bar Q\,V(x)\le L(1+V(x)).
\]
Hence
\[
\mathbb E\sup_{s\le\tau_o}(1+V(X_s))\le1+V(x)+L\,\mathbb E\int_0^{\tau_o}(1+V(X_v))\,dv\le1+V(x)+L\big(U(x)+K_0\big).
\]
\end{proof}

\begin{proof}[Proof of \Cref{thm:periodic-wellposed}]
Recall that $J(\Pi;x)$ is the long-run average cost of a policy $\Pi$ from $x$, defined at the start of this section, that $A(x)$ is the set of actions attaining the supremum that defines $\mathcal Hh(x)$, and that $\mathcal F$ is the class~\eqref{eq:dual-class}, taken on the horizon $[0,T]$. We call a stationary control kernel $q$ \emph{optimal} if $J(q;x)=g^*$ for every $x\in\mathcal S$.

\emph{(i) Existence.} By \Cref{prop:ec-acoe-existence}, there are a constant $g^*\ge0$, a function $h$ with $\|h\|_V<\infty$ and a constant $N<\infty$ such that $g^*+\mathcal Hh=0$ and $-h\ge-N$.

\emph{(i) Optimality over all policies.} Set $\hat w:=-h+N$, so that $\hat w\ge0$ and $\|\hat w\|_V<\infty$. Since $g^*+\mathcal Hh=0$ and $\mathcal G^u$ annihilates constants,
\[
\min_{u\in\mathcal U(x)}\big\{c(x,u)+\mathcal G^u\hat w(x)\big\}=-\mathcal Hh(x)=g^*,
\]
so $c^q+\mathcal G^q\hat w\ge g^*$ for every control kernel $q$. Let $\Pi$ be a policy with kernels $q_s$. We first check that Dynkin's formula holds for $\hat w$ under $\Pi$. Since $\hat w\ge0$ and, because $1+V\ge1$, $\|\hat w\|_V\le\|h\|_V+N<\infty$, the definition of $\|\cdot\|_V$ gives $0\le\hat w\le\|\hat w\|_V(1+V)$. The generator term can take either sign, and~\eqref{eq:ec-gen-bound}, averaged over $u$ against $q_s$, bounds its absolute value, $|\mathcal G^{q_s}\hat w(x)|\le L\|\hat w\|_V(1+V(x))$ for every $s$. Moreover, the jump argument in the proof of \Cref{lem:ec-reg} and the bound $\mathbb E^\Pi_x(1+V(X_v))\le e^{\kappa_Vv}(1+V(x))$ of \Cref{lem:ec-comp} give
\[
\mathbb E^\Pi_x\sup_{s\le t}(1+V(X_s))\le1+V(x)+L\int_0^t\mathbb E^\Pi_x(1+V(X_v))\,dv\le\big(1+Lte^{\kappa_Vt}\big)(1+V(x)).
\]
Let $\tau_n$ be the exit time of $X$ from the finite set $\{V\le n\}$. By Dynkin's formula at the stopping time $t\wedge\tau_n$, which applies because before $\tau_n$ the process stays in $\{V\le n\}$, where the expected absolute jump rate of $\hat w$ is at most $L\|\hat w\|_V(1+n)$,
\[
\mathbb E^\Pi_x\hat w(X_{t\wedge\tau_n})=\hat w(x)+\mathbb E^\Pi_x\int_0^{t\wedge\tau_n}\mathcal G^{q_s}\hat w(X_s)\,ds,
\]
and the two integrands are dominated by $\|\hat w\|_V\sup_{s\le t}(1+V(X_s))$ and $L\|\hat w\|_V\int_0^t(1+V(X_s))\,ds$ (respectively), which are integrable. Letting $n\to\infty$, with $\tau_n\uparrow\infty$ because the process is nonexplosive, yields 
\[
\mathbb E^\Pi_x\hat w(X_t)=\hat w(x)+\mathbb E^\Pi_x\int_0^t\mathcal G^{q_s}\hat w(X_s)\,ds\ge \hat w(x)+g^*t-\mathbb E^\Pi_x\int_0^tc^{q_s}(X_s)\,ds.
\]
Since $0 \le \hat w \le C(1+V)$ with $C:=\|\hat w\|_V$,
\[
\mathbb E^\Pi_x\int_0^tc^{q_s}(X_s)\,ds\ge g^*t-C\,\mathbb E^\Pi_x\big(1+V(X_t)\big),\qquad t\ge0.
\]
By \Cref{assump:uniform}(i) and \Cref{lem:ec-comp} with $W=U$, $\Phi=1+V$ and $s=b_U$, $\int_0^t\mathbb E^\Pi_x(1+V(X_s))\,ds\le U(x)+b_Ut$. Hence $\mathbb E^\Pi_x(1+V(X_t))/t$ cannot stay above a positive constant for all large $t$, since its integral would then grow quadratically, which is a contradiction. Therefore, there exists a sequence $t_k\to\infty$ with $\mathbb E^\Pi_x(1+V(X_{t_k}))/t_k\to0$. Dividing the previous display by $t_k$ and letting $k\to\infty$ gives $J(\Pi;x)\ge g^*$.

\emph{(i) Optimal stationary controls.} Let $q$ be a stationary control kernel with $q(A(x)\mid x)=1$ for every $x$. Every $u\in A(x)$ attains the supremum defining $\mathcal Hh(x)$, so $\mathcal G^qh-c^q=\mathcal Hh=-g^*$, and $\hat w=-h+N$ satisfies $\mathcal G^q\hat w=g^*-c^q$. \Cref{lem:ec-comp} with $W=\hat w$, $\Phi=c^q$ and $s=g^*$ gives $\mathbb E^q_x\int_0^tc^q(X_s)\,ds\le \hat w(x)+g^*t$, so $J(q;x)\le g^*$, and with the previous step $J(q;x)=g^*$ for every $x$.

By \Cref{lem:ec-reg}, from every state $x$ the chain of $q$ hits $o$ almost surely, with $\mathbb E^q_x\tau_o\le U(x)+K_0$. The return time to $o$ is also integrable: the chain leaves $o$ after an exponential holding time and then hits $o$ from the new state $y$ within expected time $U(y)+K_0$, and $\sum_yq(o,\{y\};q)U(y)<\infty$ by \Cref{assump:uniform}(i). Hence $o$ is positive recurrent, every recurrent class of the chain contains $o$, so there is a single recurrent class $R_q\ni o$, and the chain has a unique invariant law $\pi_q$, supported on $R_q$.

For the moment bound, fix $\ell<\infty$, let $v_\ell:=(1+V)\wedge\ell$ and $P_s\phi(x):=\mathbb E^q_x\phi(X_s)$, and let $a_t(x):=\frac1t\int_0^tP_sv_\ell(x)\,ds$. Since $\pi_q$ is invariant, Tonelli's theorem gives $\langle v_\ell,\pi_q\rangle=\sum_x\pi_q(x)\,a_t(x)$ for every $t>0$. The sum is the integral of $a_t$ against the probability measure $\pi_q$, and $0\le a_t\le\ell$. Applying reverse Fatou's lemma to these integrals along $t=m$, $m=1,2,\dots$, and using $a_t(x)\le(U(x)+b_Ut)/t$, which follows from \Cref{lem:ec-comp} with $W=U$, gives $\langle v_\ell,\pi_q\rangle\le b_U$. Letting $\ell\uparrow\infty$, monotone convergence gives $\langle1+V,\pi_q\rangle\le b_U$. Finally, $h\in B_V$, so \Cref{lem:ec-pair} on $R_q$ gives $\langle\mathcal G^qh,\pi_q\rangle=0$, and $c^q=g^*+\mathcal G^qh$ gives $\langle c^q,\pi_q\rangle=g^*$.

\emph{(ii)} Let $f\in\mathcal F$ solve~\eqref{eq:periodic-bvp-hard}, write $f_t:=f(t,\cdot)$, and write the endpoint gap, which is constant in $x$ by the boundary condition, as $f_T-f_0=T\gamma'$ with $\gamma'\in\mathbb R$. Let $f^*:=h+(t-T)g^*$ be the solution given by part (i), and set $e:=f-f^*$ and $\Delta:=\gamma'-g^*$. We show that $e$ is constant.

\emph{Step 0: extension and selectors.} Extend $f$ to $t\ge0$ by $f_{t+T}:=f_t+T\gamma'$. Since $\mathcal H(\phi+k)=\mathcal H\phi$ for every constant $k$, the extension solves $\partial_tf_t+\mathcal Hf_t=0$ on each interval $(mT,(m+1)T)$, $m\ge0$, and its time derivative is continuous at the points $mT$, because the left derivative at $T$ is $-\mathcal Hf_T=-\mathcal Hf_0$, which is the right derivative at $0$. Then $e_t=\bar e_t+t\Delta$, where $\bar e_t:=e_t-t\Delta$ is $T$-periodic, and $\sup_t\|\bar e_t\|_V+\sup_t\|\partial_te_t\|_V<\infty$ because $f\in\mathcal F$ and $\|h\|_V<\infty$.

By \Cref{assump:standing}(i) and the bound $\|f_t-f_s\|_V\le|t-s|\sup_v\|\partial_vf_v\|_V$, the map $(t,u)\mapsto\mathcal G^uf_t(x)$ is jointly continuous, so $(t,u)\mapsto\mathcal G^uf_t(x)-c(x,u)$ is jointly upper semicontinuous, and the measurable selection theorem of \citet{aumann1969} gives, for each $x$, a Lebesgue-measurable map $t\mapsto u_t(x)\in\mathcal U(x)$ that attains the supremum defining $\mathcal Hf_t(x)$ for almost every $t$. Let $\Pi_f$ be the $T$-periodic Markov policy that uses the action $u_t(x)$ at time $t$ in state $x$. Let also $u^{**}$ be a stationary selector of $A$, so that $u^{**}(x)$ attains the supremum defining $\mathcal Hh(x)=\mathcal Hf^*_t(x)$. Each of the two selectors attains the supremum for one of the functions $f$ and $f^*$ and is feasible for the other, and $\partial_tf^*_t=g^*=-\mathcal Hf^*_t$, so
\begin{equation}\label{eq:ec-e-drifts}
\partial_te_t+\mathcal G^{u_t}e_t\ge0\qquad\text{and}\qquad\partial_te_t+\mathcal G^{u^{**}}e_t\le0\qquad\text{on }[0,\infty)\times\mathcal S,
\end{equation}
the first for almost every $t$.

\emph{Step 1: $\Delta\ge0$.} Under $\Pi_f$ from $x$, the process $e(s,X_s)$ is a submartingale on every bounded time interval, by the first inequality in~\eqref{eq:ec-e-drifts}, its compensated jumps being integrable by \Cref{lem:ec-comp}. Hence, with $C:=\sup_t\|\bar e_t\|_V$,
\[
s\Delta+\mathbb E\,\bar e(s,X_s)\ge e(0,x),\qquad\text{so}\qquad s\Delta\ge e(0,x)-C\,\mathbb E\big(1+V(X_s)\big),\qquad s\ge0.
\]
Integrating over $s\in[0,S]$ and using $\int_0^S\mathbb E(1+V(X_s))\,ds\le U(x)+b_US$, from \Cref{assump:uniform}(i) and \Cref{lem:ec-comp}, gives $S^2\Delta/2\ge S\,e(0,x)-C\,(U(x)+b_US)$. Dividing by $S^2$ and letting $S\to\infty$ gives $\Delta\ge0$.

\emph{Step 2: $\Delta=0$.} Let $R\ni o$ be the recurrent class of the chain of $u^{**}$ and $\pi$ its invariant law, with $\langle V,\pi\rangle<\infty$ by part (i). By \Cref{lem:ec-pair}, $\langle\mathcal G^{u^{**}}e_t,\pi\rangle=0$, and differentiation under the sum, dominated by $\sup_t\|\partial_te_t\|_V(1+V)$, gives
\[
\frac{d}{dt}\langle e_t,\pi\rangle=\langle\partial_te_t+\mathcal G^{u^{**}}e_t,\pi\rangle\le0.
\]
Since $e_T-e_0=T\Delta$ is constant in $x$, integrating over $[0,T]$ gives $T\Delta=\langle e_T-e_0,\pi\rangle\le0$. With Step 1, $\Delta=0$, so $e_T=e_0$ and $e$ is $T$-periodic.

\emph{Step 3: $e$ is constant on $R$.} By Step 2 and the second inequality in~\eqref{eq:ec-e-drifts}, $e$ satisfies the hypotheses of \Cref{lem:ec-periodic} with $q=u^{**}$, the class $R$ and the law $\pi$, which gives a constant $k$ with $e\equiv k$ on $[0,T]\times R$, hence on $[0,\infty)\times R$ by periodicity. In particular $e(t,o)=k$ for every $t$.

\emph{Step 4: $e\ge k$.} Fix $(t,x)$, run the chain of $u^{**}$ from $x$, and let $Y_s:=e(t+s,X_s)$. By the second inequality in~\eqref{eq:ec-e-drifts}, $Y$ has nonpositive drift, and its compensated jumps up to $\tau_o$ form a martingale because $|e|\le C'(1+V)$, with $C':=\sup_t\|e_t\|_V$, and \Cref{lem:ec-reg} bounds $\mathbb E\int_0^{\tau_o}(1+V(X_v))\,dv$. Hence $\mathbb E\,Y_{s\wedge\tau_o}\le e(t,x)$ for every $s$. As $s\to\infty$, $Y_{s\wedge\tau_o}\to e(t+\tau_o,o)=k$ almost surely, since $\tau_o<\infty$, and $|Y_{s\wedge\tau_o}|\le C'\sup_{v\le\tau_o}(1+V(X_v))$, which is integrable by \Cref{lem:ec-reg}. Dominated convergence gives $k\le e(t,x)$.

\emph{Step 5: $e\le k$.} Run instead the policy $\Pi_f$, shifted to start at time $t$ from $x$. By the first inequality in~\eqref{eq:ec-e-drifts}, $Y$ is a submartingale up to $\tau_o$, and \Cref{lem:ec-reg}, which holds for every policy, gives $e(t,x)\le k$ in the same way. Hence $e\equiv k$, that is, $f=h+(t-T)g^*+k$, and the endpoint gap of $f$ is $Tg^*$.

For the last claim, let $(h',g)$ satisfy $g+\mathcal Hh'=0$ with $\|h'\|_V<\infty$. Then $f':=h'+(t-T)g$ belongs to $\mathcal F$ and solves~\eqref{eq:periodic-bvp-hard}, so $h'+(t-T)g=h+(t-T)g^*+k$ for every $t$, which gives $g=g^*$ and $h'=h+k$.

\emph{(iii)} By part (i), $f^*\in\mathcal F$, and \Cref{lem:weak-duality} on the horizon $[0,T]$ gives $g_T(\tilde\pi)\ge\langle f^*_T-f^*_0,\tilde\pi\rangle=Tg^*$ for every $\tilde\pi$ with $\langle V,\tilde\pi\rangle<\infty$. For $q$ as in (i), the constant flow $\rho_t\equiv\pi_q$ with the control $q$ is admissible for $\tilde\pi=\pi_q$, with finite energy because $\langle V,\pi_q\rangle<\infty$, and its cost is $T\langle c^q,\pi_q\rangle=Tg^*$ by part (i). Hence $g_T(\pi_q)=Tg^*$, and $\pi_q$ minimizes $g_T$.

\emph{Convexity.} Let $q_1,q_2$ be as in (i), let $\lambda\in[0,1]$, and let $\nu_i(dx,du):=\pi_{q_i}(dx)\,q_i(du\mid x)$. The mixture $\nu:=\lambda\nu_1+(1-\lambda)\nu_2$ satisfies the stationary BAR $\int\mathcal G^u\phi(x)\,\nu(dx,du)=0$ for bounded $\phi$, which is linear in $\nu$. Its state marginal is $\bar\pi:=\lambda\pi_{q_1}+(1-\lambda)\pi_{q_2}$, and it disintegrates as $\bar\pi(dx)\,\bar q(du\mid x)$, where
\[
\bar q(\cdot\mid x):=\frac{\lambda\pi_{q_1}(x)\,q_1(\cdot\mid x)+(1-\lambda)\pi_{q_2}(x)\,q_2(\cdot\mid x)}{\bar\pi(x)}\qquad\text{if }\bar\pi(x)>0,
\]
and $\bar q(\cdot\mid x)$ is any kernel supported on $A(x)$ otherwise. Then $\bar q$ is supported on $A$, and the stationary BAR states that $\bar\pi$ is invariant for $\bar q$, so $\bar\pi=\pi_{\bar q}$ by the uniqueness in part (i).

\emph{Optimal stationary kernels.} Every $q$ as in (i) is optimal. Conversely, let $q$ be an optimal stationary control kernel. By \Cref{lem:ec-reg}, as in part (i), its chain has a single recurrent class $R_q\ni o$ and a unique invariant law $\pi_q$, with $\langle1+V,\pi_q\rangle\le b_U$. The truncation argument of part (i), with $c^q$ in place of $1+V$ and the identity $J(q;x)=g^*$ in place of the bound by $b_U$, gives $\langle c^q,\pi_q\rangle\le g^*$. On the other hand, $c^q+\mathcal G^q\hat w\ge g^*$ with $\hat w=-h+N\in B_V$, and $\langle\mathcal G^q\hat w,\pi_q\rangle=0$ by \Cref{lem:ec-pair}, so $\langle c^q,\pi_q\rangle\ge g^*$. Hence $\langle c^q,\pi_q\rangle=g^*$. Now
\[
\mathcal G^qh(x)-c^q(x)+g^*=\int_{\mathcal U(x)}\big[\mathcal G^uh(x)-c(x,u)+g^*\big]\,q(du\mid x)\le0,
\]
because the integrand is at most $\mathcal Hh(x)+g^*=0$, with equality exactly on $A(x)$. The $\pi_q$-integral of the left side is $\langle\mathcal G^qh,\pi_q\rangle-\langle c^q,\pi_q\rangle+g^*=0$, so $q(A(x)\mid x)=1$ for every $x$ with $\pi_q(x)>0$, that is, for every $x\in R_q$. Replacing $q$ off $R_q$ by a kernel supported on $A$ does not change the dynamics on the closed class $R_q$, so $\pi_q$ remains invariant and, by the uniqueness in part (i), is the invariant law of a kernel as in (i). Thus the laws $\pi_q$ of part (i) are exactly the invariant laws of the optimal stationary control kernels. If $A(x)$ is a singleton for every $x$, only one stationary kernel is supported on $A$, and the set is a single law.
\end{proof}

\begin{proof}[Proof of \Cref{prop:queue-periodic}]
Let $\omega(x):=\sum_ix_i/\mu_i$ be the workload and $\bar c:=\max_ic_i\mu_i$, so that $c\le\bar c\,\omega$. Every action at $x\ne0$ is a randomized non-idling schedule, which lowers $\omega$ at total rate one, so for every $u\in\mathcal U(x)$ and $x\ne0$,
\[
\mathcal G^u\omega^2(x)=2\omega(x)(\rho-1)+a_2(x),\qquad \mathcal G^u\omega^3(x)=3\omega(x)^2(\rho-1)+a_3(x),
\]
with $|a_2|$ bounded and $|a_3(x)|\le a(1+\omega(x))$; at $x=0$ only arrivals occur. \Cref{assump:standing}(i) holds because the jumps are nearest-neighbor with rates affine in $u$, and the hypotheses of \Cref{lem:weak-duality} hold for the quadratic $V$, as noted after that lemma. For \Cref{assump:standing}(ii), take $q^\circ$ uniform on the vertices $\{\mu_ie_i:x_i>0\}$; its chain is irreducible because all $\lambda_i>0$ and every nonempty class is served, $c^{q^\circ}=c\le C_c(1+V)$, and $W^\circ:=\frac{\bar c}{1-\rho}\omega^2$ satisfies $\mathcal G^{q^\circ}W^\circ\le-2c+b_2\le-c+b_2\mathbf 1_C$ with $C:=\{c\le b_2\}$ finite. For \Cref{assump:uniform}(i), the second identity gives $\mathcal G^u\omega^3\le-\frac32(1-\rho)\omega^2+b_3$ for every action at every state, and since $1+V\le1+(\max_i\mu_i)^2\omega^2$, a suitable multiple $U$ of $\omega^3$ satisfies \Cref{assump:uniform}(i) with $F$ a finite sublevel set of $V$. For \Cref{assump:uniform}(ii), let $m:=\max_{x\in F}\sum_ix_i$. While the system is nonempty, every policy serves at total rate at least $\mu_{\min}:=\min_i\mu_i$, so with probability at least $\delta:=e^{-t_0\sum_i\lambda_i}\,\mathbb P\big(\mathrm{Poisson}(\mu_{\min}t_0)\ge m\big)>0$ there is no arrival in $[0,t_0]$ and the system empties by time $t_0$, from every $x\in F$ and under every policy.

For the hypotheses of \Cref{thm:ansatz}: \Cref{assump:affine} holds for the convex hulls $\mathcal U(x)$. Every non-idling stationary control is irreducible, since with no arrivals the system empties and from $0$ every state is reachable by arrivals. For the sublevel sets, let $w:=-h$ and $G:=\{c\le g^*+1\}\cup F$, a finite set containing $o$. Under a stationary selector of $A$, $\mathcal Gw=g^*-c\le-1$ off $G$, and optional stopping at the hitting time $\tau_G$, justified by \Cref{lem:ec-reg}, gives $w(x)\ge\min_Gw+\mathbb E_x\tau_G\ge\min_Gw+\mathrm{dist}(x,G)/\bar Q$, because each nearest-neighbor jump changes the distance to $G$ by at most one. Hence $-h$ has finite sublevel sets. Finally $\langle V,\tilde\pi^*\rangle<\infty$ by \Cref{thm:periodic-wellposed}(i).
\end{proof}

\endgroup

\paragraph{The regularized problem.} Throughout this part, $\varepsilon>0$ and the reference kernel $r_x$ satisfy the hypotheses of \Cref{thm:soft-periodic}, $c^q_\varepsilon:=c^q+\varepsilon D_{\mathrm{KL}}(q\,\|\,r)$, and $J_\varepsilon(\Pi;x):=\limsup_{t\to\infty}t^{-1}\mathbb E^\Pi_x\int_0^tc^{q_s}_\varepsilon(X_s)\,ds$.

\begin{lemma}[Soft Hamiltonian]\label{lem:ec-soft}
Let $\phi\in B_V$. Then $0<\mathcal Z^\varepsilon\phi(x)<\infty$, and
\[
\mathcal H^\varepsilon\phi(x)=\sup_q\big\{\mathcal G^q\phi(x)-c^q_\varepsilon(x)\big\},
\]
the supremum over probability measures $q$ on $\mathcal U(x)$, attained uniquely at the Gibbs kernel $q^\varepsilon_\phi(\cdot\mid x)$. The kernels $q^\varepsilon_\phi(\cdot\mid x)$ and $r_x$ are mutually absolutely continuous, and the chain corresponding to the stationary kernel $q^\varepsilon_\phi$ is irreducible. Moreover $\mathcal H^\varepsilon$ maps $B_V$ into $B_V$, and $|\mathcal H^\varepsilon\phi(x)-\mathcal H^\varepsilon\phi'(x)|\le L\|\phi-\phi'\|_V(1+V(x))$.
\end{lemma}

\begin{proof}
{Recall that $B_V$ is the space of functions $\phi$ on $\mathcal S$ with $\|\phi\|_V<\infty$, that $L=\kappa_V+\bar Q$ is the constant in~\eqref{eq:ec-gen-bound}, that $q^\circ$ is the stationary control kernel of \Cref{assump:standing}\,(ii), whose chain is irreducible and whose cost satisfies $c^{q^\circ}\le C_c(1+V)$, and that $K_r=\sup_xD_{\mathrm{KL}}(q^\circ(\cdot\mid x)\,\|\,r_x)<\infty$ by the hypothesis of \Cref{thm:soft-periodic}.} Since {\eqref{eq:ec-gen-bound} gives} $|\mathcal G^u\phi(x)|\le L\|\phi\|_V(1+V(x))$ uniformly in $u$ and $c$ is finite and nonnegative, the integrand of $\mathcal Z^\varepsilon\phi(x)$ takes values in $(0,e^{L\|\phi\|_V(1+V(x))/\varepsilon}]$. The variational identity and the uniqueness of its maximizer are the Gibbs variational principle {\citep[Prop.~1.4.2]{dupuisellis1997}} for the function $u\mapsto(\mathcal G^u\phi(x)-c(x,u))/\varepsilon$.

The density of $q^\varepsilon_\phi(\cdot\mid x)$ with respect to $r_x$ is finite and strictly positive, which gives mutual absolute continuity. Since $K_r<\infty$, $q^\circ(\cdot\mid x)$ is absolutely continuous with respect to $r_x$, so every transition with positive rate under $q^\circ$ also has positive rate under $r$ and hence under $q^\varepsilon_\phi$, and the chain of $q^\varepsilon_\phi$ inherits the irreducibility of the chain of $q^\circ$.

For the bounds, $\mathcal H^\varepsilon\phi(x)\le\sup_u\mathcal G^u\phi(x)\le L\|\phi\|_V(1+V(x))$, and taking $q=q^\circ$ in the variational identity gives $\mathcal H^\varepsilon\phi(x)\ge-(L\|\phi\|_V+C_c)(1+V(x))-\varepsilon K_r$. The Lipschitz bound holds because $z\mapsto\varepsilon\log\int e^{z(u)/\varepsilon}r_x(du)$ is monotone and commutes with adding constants.
\end{proof}

{Next, recall the function $W^\circ:\mathcal S\to[0,\infty)$, the finite set $C$ and the constant $b<\infty$ of \Cref{assump:standing}\,(ii), which satisfy $W^\circ\le C_W(1+V)$ and $\mathcal G^{q^\circ}W^\circ\le-c^{q^\circ}+b\,\mathbf 1_C$.}

\begin{lemma}[Vanishing discount for the regularized problem]\label{lem:ec-vd}
For $\beta\in(0,1]$ let $w_\beta(x):=\inf_\Pi\mathbb E^\Pi_x\int_0^\infty e^{-\beta t}c^{q_t}_\varepsilon(X_t)\,dt$, with infimum over all continuous-time policies.
\begin{enumerate}
\item[(a)] $0\le w_\beta\le W^\circ+(b+\varepsilon K_r)/\beta$, and $\beta w_\beta=\mathcal L^\varepsilon w_\beta$ with $\mathcal L^\varepsilon w(x):=\inf_{q}\{c^q_\varepsilon(x)+\mathcal G^qw(x)\}=-\mathcal H^\varepsilon(-w)(x)$, the infimum over probability measures $q$ on $\mathcal U(x)$.
\item[(b)] Fix $x_0\in\mathcal S$ and set $\eta_\beta:=w_\beta-w_\beta(x_0)$. There exists $N<\infty$ and $M:\mathcal S\to[0,\infty)$ with $M\le C_M(1+V)$, both independent of $\beta$, such that $-N\le\eta_\beta\le M$, and $\beta w_\beta(x_0)\le\beta W^\circ(x_0)+b+\varepsilon K_r$.
\end{enumerate}
\end{lemma}

\begin{proof}
(a) Following $q^\circ$, whose regularized cost satisfies $c^{q^\circ}_\varepsilon\le c^{q^\circ}+\varepsilon K_r$, \Cref{lem:ec-comp} applied to $e^{-\beta t}W^\circ$ and \Cref{assump:standing}(ii) give $\mathbb E^{q^\circ}_x\int_0^\infty e^{-\beta t}c^{q^\circ}_\varepsilon(X_t)\,dt\le W^\circ(x)+(b+\varepsilon K_r)/\beta$. The equation $\beta w_\beta=\mathcal L^\varepsilon w_\beta$ is the dynamic-programming equation of a discounted problem with nonnegative cost, bounded transition rates and compact action sets, the actions being probability measures $q$ on $\mathcal U(x)$ in the weak topology. {After uniformization at the rate $\bar Q=\sup_{x,u}q(x,\mathcal S;u)$, the bound on the total jump intensity, it is the Bellman equation of a discrete-time discounted model with nonnegative one-stage cost and compact action sets}, in which $q\mapsto c^q_\varepsilon(x)+\sum_yw(y)\int q(x,\{y\};u)\,q(du)$ is lower semicontinuous for $w\ge0$, by \Cref{assump:standing}(i), Fatou's lemma and the lower semicontinuity of $D_{\mathrm{KL}}$; for such models the value function is the limit of value iteration from $0$ and solves the Bellman equation \citep[Ch.~9]{bertsekasshreve1978}, see also \citet{guo2009} for the continuous-time model. {For $w$ with $\|w\|_V<\infty$, the identity $\mathcal L^\varepsilon w=-\mathcal H^\varepsilon(-w)$ follows from the variational identity of \Cref{lem:ec-soft} applied to $\phi:=-w\in B_V$. Indeed, since $\mathcal G^q(-w)=-\mathcal G^qw$,
\[
-\mathcal H^\varepsilon(-w)(x)=-\sup_q\big\{-\mathcal G^qw(x)-c^q_\varepsilon(x)\big\}=\inf_q\big\{c^q_\varepsilon(x)+\mathcal G^qw(x)\big\}=\mathcal L^\varepsilon w(x).
\]
In particular the identity applies to $w_\beta$, because $0\le w_\beta\le W^\circ+(b+\varepsilon K_r)/\beta\le C_\beta(1+V)$ for a constant $C_\beta<\infty$, by the bound just proved and $W^\circ\le C_W(1+V)$.}

(b) {This part verifies, for the regularized problem in continuous time, the analogue of the conditions in \citet[Assumption~5.4.1]{hernandez1996discrete} for the vanishing-discount approach, with the upper bound on the relative values strengthened to $M\le C_M(1+V)$ so that their limits lie in $B_V$.} The bound on $\beta w_\beta(x_0)$ is (a). For the upper bound on $\eta_\beta$, let $\tau_0$ be the hitting time of $x_0$, and enlarge $C$ by the finite set $\{x : \min_uc(x,u)<1\}$, which preserves the drift inequality. Following $q^\circ$ until $\tau_0$ and then a $\delta'$-optimal policy from $x_0$ gives $w_\beta(x)\le m(x)+w_\beta(x_0)+\delta'$ with $m(x):=\mathbb E^{q^\circ}_x\int_0^{\tau_0}\big(c^{q^\circ}(X_t)+\varepsilon K_r\big)\,dt$. The chain of $q^\circ$ is positive recurrent by \Cref{assump:standing}(ii) and Foster's criterion~{\citep[Thm.~4.2]{meyntweedie1993}}, since $c^{q^\circ}\ge\min_uc(\cdot,u)$ is coercive, so every state is $c^{q^\circ}$-regular and $\max_{y\in C}m(y)<\infty$. The argument of \Cref{lem:ec-comp} stopped at the hitting time $\tau_C$ of $C$ gives $\mathbb E^{q^\circ}_x\int_0^{\tau_C}c^{q^\circ}\le W^\circ(x)$, and $c^{q^\circ}\ge1$ off $C$ gives $\mathbb E^{q^\circ}_x\tau_C\le W^\circ(x)$; the strong Markov property at $\tau_C$ then gives $m\le(1+\varepsilon K_r)W^\circ+\max_Cm\le C_M(1+V)$. 

For the lower bound, let $L_0:=W^\circ(x_0)+b+\varepsilon K_r+1$ and $G:=\{x:\min_uc(x,u)\le L_0\}$, a finite set by \Cref{assump:coercive}. For $x\notin G$ and any policy, let $\sigma$ be the hitting time of $G$. Since $c^q_\varepsilon\ge c^q\ge\min_uc(\cdot,u)$, the cost accrued before $\sigma$ is at least $L_0\,\mathbb E[1-e^{-\beta\sigma}]/\beta$, and the cost after $\sigma$ is at least $\mathbb E[e^{-\beta\sigma}]\min_Gw_\beta$, so the total is at least $\min(L_0/\beta,\min_Gw_\beta)$. Since $\inf_{\mathcal S}w_\beta\le w_\beta(x_0)\le W^\circ(x_0)+(b+\varepsilon K_r)/\beta<L_0/\beta$, it follows that $\inf_{\mathcal S}w_\beta=\min_Gw_\beta$. Following $q^\circ$ from $x_0$ until the hitting time of $y\in G$ gives $w_\beta(x_0)\le m_y+w_\beta(y)$, with $m_y:=\mathbb E^{q^\circ}_{x_0}\int_0^{\tau_y}\big(c^{q^\circ}+\varepsilon K_r\big)<\infty$, so $\eta_\beta\ge-N$ with $N:=\max_{y\in G}m_y$.
\end{proof}

\begin{proof}[Proof of \Cref{thm:soft-periodic}]
\emph{(i) Existence.} By \Cref{lem:ec-vd}(b) and a diagonal argument over the countable $\mathcal S$ {\citep[proof of Thm.~5.5.4, Remark~5.5.2 and Notes on \S5.5]{hernandez1996discrete}}, there exists $\beta_n\downarrow0$, $g^*_\varepsilon\ge0$ and $\bar\eta$ with $\beta_nw_{\beta_n}(x_0)\to g^*_\varepsilon$, $\eta_n:=\eta_{\beta_n}\to\bar\eta$ pointwise and $-N\le\eta_n\le M$. 

Since $\mathcal L^\varepsilon(w+k)=\mathcal L^\varepsilon w$ for constants $k$, \Cref{lem:ec-vd}(a) gives $\beta_n\eta_n(x)+\beta_nw_{\beta_n}(x_0)=\mathcal L^\varepsilon\eta_n(x)$, whose left side tends to $g^*_\varepsilon$. 

{By \Cref{lem:ec-soft}, for every $w$ with $\|w\|_V<\infty$ and every $x\in\mathcal S$,
\[
\mathcal L^\varepsilon w(x)=-\varepsilon\log\int_{\mathcal U(x)}e^{-(c(x,u)+\mathcal G^uw(x))/\varepsilon}\,r_x(du).
\]
Fix $x\in\mathcal S$ and an action $u\in\mathcal U(x)$. By the definition~\eqref{eq:generator} of the generator on the countable state space,
\[
\mathcal G^u\eta_n(x)=\sum_{y\in\mathcal S}\big[\eta_n(y)-\eta_n(x)\big]\,q(x,\{y\};u),
\]
where $q(x,\{y\};u)$ is the jump rate. Each summand converges as $n\to\infty$, since $\eta_n\to\bar\eta$ pointwise. Moreover, $-N\le\eta_n\le M\le C_M(1+V)$ gives $|\eta_n|\le(N+C_M)(1+V)$, so the summand indexed by $y$ is bounded in absolute value by $(N+C_M)\big(2+V(x)+V(y)\big)\,q(x,\{y\};u)$ for every $n$. The jump kernel satisfies the hypotheses of \Cref{lem:weak-duality}, namely
\[
q(x,\mathcal S;u)\le\bar Q\qquad\text{and}\qquad\sum_{y\in\mathcal S}(1+V(y))\,q(x,\{y\};u)\le\kappa_V\,(1+V(x)),
\]
so the bound is summable in $y$, and dominated convergence gives $\mathcal G^u\eta_n(x)\to\mathcal G^u\bar\eta(x)$. Summing the same bound over $y$ gives
\[
\big|\mathcal G^u\eta_n(x)\big|\le(N+C_M)(\kappa_V+\bar Q)(1+V(x))\qquad\text{for every }n\text{ and every }u\in\mathcal U(x).
\]
Since $c\ge0$, the integrand $e^{-(c(x,u)+\mathcal G^u\eta_n(x))/\varepsilon}$ is therefore at most $e^{(N+C_M)(\kappa_V+\bar Q)(1+V(x))/\varepsilon}$ for every $n$ and $u$, and it converges for each $u$. Because $r_x$ is a probability measure on $\mathcal U(x)$, dominated convergence in $u$ shows that the integral converges, and its limit is strictly positive because the limiting integrand is. Consequently, $\mathcal L^\varepsilon\eta_n(x)\to\mathcal L^\varepsilon\bar\eta(x)$.} 

{Letting $n\to\infty$ in $\beta_n\eta_n(x)+\beta_nw_{\beta_n}(x_0)=\mathcal L^\varepsilon\eta_n(x)$ therefore gives $g^*_\varepsilon=\mathcal L^\varepsilon\bar\eta(x)$ for every $x\in\mathcal S$. Set $h_\varepsilon:=-\bar\eta$. Since $\mathcal L^\varepsilon w=-\mathcal H^\varepsilon(-w)$ by \Cref{lem:ec-vd}(a),
\[
g^*_\varepsilon+\mathcal H^\varepsilon h_\varepsilon(x)=g^*_\varepsilon-\mathcal L^\varepsilon\bar\eta(x)=0\qquad\text{for every }x\in\mathcal S.
\]
The bounds $-N\le\bar\eta\le M\le C_M(1+V)$ give $\|h_\varepsilon\|_V\le N+C_M$ and $-h_\varepsilon=\bar\eta\ge-N$. Finally, let $f^\varepsilon(t,x):=h_\varepsilon(x)+(t-T)g^*_\varepsilon$, as in the statement, and write $f^\varepsilon_t:=f^\varepsilon(t,\cdot)$. For each $t$, the function $x\mapsto(t-T)g^*_\varepsilon$ is constant, and $\mathcal G^u$ annihilates constants, so $\mathcal G^uf^\varepsilon_t=\mathcal G^uh_\varepsilon$ for every $u$, and hence $\mathcal H^\varepsilon f^\varepsilon_t=\mathcal H^\varepsilon h_\varepsilon$. Therefore
\[
\partial_tf^\varepsilon_t(x)+\mathcal H^\varepsilon f^\varepsilon_t(x)=g^*_\varepsilon+\mathcal H^\varepsilon h_\varepsilon(x)=0
\qquad\text{and}\qquad
f^\varepsilon_T(x)-f^\varepsilon_0(x)=Tg^*_\varepsilon\qquad\text{for every }x\in\mathcal S,
\]
so $f^\varepsilon$ solves~\eqref{eq:periodic-bvp}.}

\emph{(i) Lower bound for every policy.} {Recall that the regularized long-run average cost of a policy $\Pi$ from the initial state $x$ is
\[
J_\varepsilon(\Pi;x)=\limsup_{t\to\infty}\frac1t\,\mathbb E^\Pi_x\int_0^tc^{q_s}_\varepsilon(X_s)\,ds,\qquad c^{q}_\varepsilon=c^{q}+\varepsilon D_{\mathrm{KL}}(q\,\|\,r),
\]
where $q_s$ is the control kernel that $\Pi$ selects at time $s$.} %
{Let $\Pi$ be a policy with $J:=J_\varepsilon(\Pi;x)<\infty$, and let
\[
A(t):=\mathbb E^\Pi_x\int_0^tc^{q_s}_\varepsilon(X_s)\,ds,\qquad t\ge0,
\]
be the expected regularized cost accrued by time $t$. Since $c^q_\varepsilon\ge0$, the function $A$ is nondecreasing with $A(0)=0$, and by Tonelli's theorem $A(t)=\int_0^ta(s)\,ds$ with $a(s):=\mathbb E^\Pi_xc^{q_s}_\varepsilon(X_s)$. By the definition of $J$ as a $\limsup$, for every $\delta>0$ there is $t_\delta<\infty$ such that $A(t)\le(J+\delta)t$ for all $t\ge t_\delta$.

Fix $\beta>0$. The $\beta$-discounted regularized cost of $\Pi$ from $x$ is $\int_0^\infty e^{-\beta t}a(t)\,dt$, again by Tonelli's theorem, and integration by parts gives
\[
\beta\int_0^\infty e^{-\beta t}a(t)\,dt=\beta^2\int_0^\infty e^{-\beta t}A(t)\,dt.
\]
Split the right side at $t_\delta$. On $[0,t_\delta]$, monotonicity gives $A(t)\le A(t_\delta)\le(J+\delta)t_\delta$, and on $[t_\delta,\infty)$, $A(t)\le(J+\delta)t$. Since $\beta^2\int_0^\infty te^{-\beta t}\,dt=1$,
\[
\beta^2\int_0^\infty e^{-\beta t}A(t)\,dt\le\beta^2t_\delta\cdot(J+\delta)t_\delta+(J+\delta)=(J+\delta)\big(1+\beta^2t_\delta^2\big).
\]
Letting $\beta\to0$ and then $\delta\to0$ gives
\[
\limsup_{\beta\to0}\,\beta\int_0^\infty e^{-\beta t}a(t)\,dt\le J_\varepsilon(\Pi;x).
\]
Next, $w_\beta(x)$ is the infimum of the $\beta$-discounted regularized cost over all continuous-time policies (\Cref{lem:ec-vd}), so $\beta w_\beta(x)\le\beta\int_0^\infty e^{-\beta t}a(t)\,dt$. Along the sequence $\beta_n$ constructed in the existence step,
\[
\beta_nw_{\beta_n}(x)=\beta_nw_{\beta_n}(x_0)+\beta_n\eta_n(x)\to g^*_\varepsilon,
\]
since $\beta_nw_{\beta_n}(x_0)\to g^*_\varepsilon$ and $-N\le\eta_n(x)\le M(x)$. Therefore $g^*_\varepsilon\le J_\varepsilon(\Pi;x)$.} 
This bound covers time-dependent policies, which (ii) requires.

\emph{(i) The Gibbs kernel.} {By \Cref{lem:ec-soft}, the chain of $q^*_\varepsilon$ is irreducible, and $q^*_\varepsilon$ attains the supremum defining $\mathcal H^\varepsilon h_\varepsilon$, that is, $\mathcal H^\varepsilon h_\varepsilon=\mathcal G^{q^*_\varepsilon}h_\varepsilon-c^{q^*_\varepsilon}_\varepsilon$. Set $w^*:=-h_\varepsilon+N$, which is nonnegative because $-h_\varepsilon\ge-N$. Since $\mathcal G^{q^*_\varepsilon}$ annihilates constants and $g^*_\varepsilon+\mathcal H^\varepsilon h_\varepsilon=0$,
\[
\mathcal G^{q^*_\varepsilon}w^*=-\mathcal G^{q^*_\varepsilon}h_\varepsilon=-\big(\mathcal H^\varepsilon h_\varepsilon+c^{q^*_\varepsilon}_\varepsilon\big)=g^*_\varepsilon-c^{q^*_\varepsilon}_\varepsilon\le g^*_\varepsilon-\min_uc(\cdot,u).
\]
Off the finite set $\{\min_uc(\cdot,u)\le g^*_\varepsilon+1\}$ the drift is at most $-1$, so Foster's criterion~\citep[Thm.~4.2]{meyntweedie1993} gives positive recurrence, with invariant law $\tilde\pi^*_\varepsilon$.

Next, apply \Cref{lem:ec-comp} to the stationary policy $q^*_\varepsilon$ with the time-independent function $W(t,x):=w^*(x)$, with $\Phi:=c^{q^*_\varepsilon}_\varepsilon\ge0$ and with $s:=g^*_\varepsilon$. The display above gives $(\partial_t+\mathcal G^{q^*_\varepsilon})W=-\Phi+s$, and the integrability hypothesis of \Cref{lem:ec-comp} holds because $0\le w^*\le(\|h_\varepsilon\|_V+N)(1+V)$ and the jump kernel satisfies the compatibility bound of \Cref{lem:weak-duality}. Hence
\[
\mathbb E^{q^*_\varepsilon}_xw^*(X_t)+\mathbb E^{q^*_\varepsilon}_x\int_0^tc^{q^*_\varepsilon}_\varepsilon(X_v)\,dv\le w^*(x)+g^*_\varepsilon t,\qquad t\ge0,\ x\in\mathcal S.
\]
Since $w^*\ge0$, the first term on the left may be dropped, and dividing by $t$ and letting $t\to\infty$ gives $J_\varepsilon(q^*_\varepsilon;x)\le g^*_\varepsilon$ for every $x$. In particular $J_\varepsilon(q^*_\varepsilon;x)<\infty$, so the lower bound for every policy applies to $q^*_\varepsilon$ and gives $J_\varepsilon(q^*_\varepsilon;x)\ge g^*_\varepsilon$. Therefore
\[
J_\varepsilon(q^*_\varepsilon;x)=g^*_\varepsilon=\inf_\Pi J_\varepsilon(\Pi;x)\qquad\text{for every }x\in\mathcal S,
\]
so $g^*_\varepsilon$ is the optimal regularized long-run average cost and the Gibbs kernel $q^*_\varepsilon$ attains it from every initial state.

It remains to prove the identity $g^*_\varepsilon=\langle c^{q^*_\varepsilon}_\varepsilon,\tilde\pi^*_\varepsilon\rangle$ of part (i), which equates the optimal rate with the expected regularized cost under the invariant law of the Gibbs chain. For the upper bound, fix $\ell<\infty$, let $c_\ell:=c^{q^*_\varepsilon}_\varepsilon\wedge\ell$, and let $P_v\phi(x):=\mathbb E^{q^*_\varepsilon}_x\phi(X_v)$. Since $\tilde\pi^*_\varepsilon$ is invariant, $\langle P_vc_\ell,\tilde\pi^*_\varepsilon\rangle=\langle c_\ell,\tilde\pi^*_\varepsilon\rangle$ for every $v\ge0$, and Tonelli's theorem gives
\[
\langle c_\ell,\tilde\pi^*_\varepsilon\rangle=\sum_{x\in\mathcal S}\tilde\pi^*_\varepsilon(x)\,\frac1t\int_0^tP_vc_\ell(x)\,dv\qquad\text{for every }t>0.
\]
The sum on the right is the integral of the time average $a_t(x):=\frac1t\int_0^tP_vc_\ell(x)\,dv$ against the probability measure $\tilde\pi^*_\varepsilon$, and $0\le a_t\le\ell$. Applying reverse Fatou's lemma to these integrals along $t=m$, $m=1,2,\dots$, and using $\limsup_{m\to\infty}a_m(x)\le\limsup_{t\to\infty}a_t(x)$, we obtain
\[
\langle c_\ell,\tilde\pi^*_\varepsilon\rangle\le\sum_{x\in\mathcal S}\tilde\pi^*_\varepsilon(x)\,\limsup_{t\to\infty}\frac1t\int_0^tP_vc_\ell(x)\,dv\le\sum_{x\in\mathcal S}\tilde\pi^*_\varepsilon(x)\,J_\varepsilon(q^*_\varepsilon;x)=g^*_\varepsilon,
\]
where the second inequality uses $c_\ell\le c^{q^*_\varepsilon}_\varepsilon$. Letting $\ell\uparrow\infty$, monotone convergence gives $\langle c^{q^*_\varepsilon}_\varepsilon,\tilde\pi^*_\varepsilon\rangle\le g^*_\varepsilon$; this is the argument in the proof of \citet[Thm.~4.3(ii)]{meyntweedie1993}.} 
{For the lower bound,} averaging $\beta_nw_{\beta_n}(x)\le\beta_n\mathbb E^{q^*_\varepsilon}_x\int_0^\infty e^{-\beta_nt}c^{q^*_\varepsilon}_\varepsilon(X_t)\,dt$ against the invariant law $\tilde\pi^*_\varepsilon$ and applying Fatou's lemma with $\beta_nw_{\beta_n}(x)\to g^*_\varepsilon$ gives $g^*_\varepsilon\le\langle c^{q^*_\varepsilon}_\varepsilon,\tilde\pi^*_\varepsilon\rangle$.

\emph{(ii)} {Let $f$ be a solution of~\eqref{eq:periodic-bvp} with the properties in part (ii) of the theorem, and write $f_t:=f(t,\cdot)$. Thus $f(\cdot,x)\in C^1$ for each $x$, $\sup_t(\|f_t\|_V+\|\partial_tf_t\|_V)<\infty$, and $-f\ge-B$ for a constant $B<\infty$. By the boundary condition in~\eqref{eq:periodic-bvp}, the endpoint gap $f_T-f_0$ is constant in $x$, and we write it as $T\gamma'$ with $\gamma'\in\mathbb R$. We show that $\gamma'\ge g^*_\varepsilon$.

Extend $f$ to all $t\ge0$ by $f_{t+T}:=f_t+T\gamma'$. Since $\mathcal H^\varepsilon(\phi+k)=\mathcal H^\varepsilon\phi$ for every constant $k$, the extension solves $\partial_tf_t+\mathcal H^\varepsilon f_t=0$ on each interval $(mT,(m+1)T)$, $m\ge0$, and its time derivative is continuous at the points $mT$, because the left derivative at $T$ is $-\mathcal H^\varepsilon f_T=-\mathcal H^\varepsilon f_0$, which is the right derivative at $0$. 

Let $q_t:=q^\varepsilon_{f_t}$ be the Gibbs kernel of $f_t$, and let $\Pi_f$ be the Markov policy that uses the kernel $q_t$ at time $t$, which is $T$-periodic. By \Cref{lem:ec-soft}, $q_t$ attains the supremum defining $\mathcal H^\varepsilon f_t$, that is, $\mathcal H^\varepsilon f_t=\mathcal G^{q_t}f_t-c^{q_t}_\varepsilon$, so the equation for $f$ reads
\[
\partial_tf_t+\mathcal G^{q_t}f_t=c^{q_t}_\varepsilon,\qquad t\ge0.
\]
Set $\hat W(t,x):=-f_t(x)+\gamma't+K$ with $K:=B+T|\gamma'|$. Then $\hat W(t+T,\cdot)=\hat W(t,\cdot)$, and $\hat W\ge-B-T|\gamma'|+K=0$ on $[0,T]$, hence for all $t\ge0$. Moreover,
\[
(\partial_t+\mathcal G^{q_t})\hat W=-\big(\partial_tf_t+\mathcal G^{q_t}f_t\big)+\gamma'=\gamma'-c^{q_t}_\varepsilon.
\]
The bounds on $\|f_t\|_V$ and $\|\partial_tf_t\|_V$, together with the compatibility bound of \Cref{lem:weak-duality}, give the integrability hypotheses of \Cref{lem:ec-comp}. Applying \Cref{lem:ec-comp} to $\Pi_f$ with $W=\hat W$, $\Phi=c^{q_t}_\varepsilon$ and $s=\gamma'$ gives
\[
\mathbb E^{\Pi_f}_x\hat W(t,X_t)+\mathbb E^{\Pi_f}_x\int_0^tc^{q_v}_\varepsilon(X_v)\,dv\le\hat W(0,x)+\gamma't,\qquad t\ge0,\ x\in\mathcal S.
\]
Dropping the nonnegative first term, dividing by $t$ and letting $t\to\infty$ gives $J_\varepsilon(\Pi_f;x)\le\gamma'<\infty$. The lower bound for every policy in part (i) applies to the time-dependent policy $\Pi_f$ and gives $g^*_\varepsilon\le J_\varepsilon(\Pi_f;x)\le\gamma'$. Hence the endpoint gap $T\gamma'$ of $f$ is at least $Tg^*_\varepsilon$.}

\emph{(iii)} {Assume in addition that $\langle V,\tilde\pi^*_\varepsilon\rangle<\infty$. Let $f$ be a solution of~\eqref{eq:periodic-bvp} with the properties in part (ii), with $f_t:=f(t,\cdot)$ and endpoint gap $f_T-f_0=T\gamma'$ as in the proof of (ii), and let $f^\varepsilon=h_\varepsilon+(t-T)g^*_\varepsilon$ be the solution constructed in part (i). Set $e_t:=f_t-f^\varepsilon_t$. Since $f^\varepsilon_t$ differs from $h_\varepsilon$ by a constant and $\mathcal G^u$ annihilates constants, the Gibbs kernel of $f^\varepsilon_t$ is $q^*_\varepsilon$ for every $t$, so by \Cref{lem:ec-soft}
\[
\mathcal H^\varepsilon f^\varepsilon_t=\mathcal G^{q^*_\varepsilon}f^\varepsilon_t-c^{q^*_\varepsilon}_\varepsilon\qquad\text{and}\qquad\mathcal H^\varepsilon f_t\ge\mathcal G^{q^*_\varepsilon}f_t-c^{q^*_\varepsilon}_\varepsilon,
\]
the second because $q^*_\varepsilon$ is one of the kernels over which the supremum defining $\mathcal H^\varepsilon f_t$ is taken. Both $f$ and $f^\varepsilon$ satisfy $\partial_t\phi_t=-\mathcal H^\varepsilon\phi_t$, so subtracting the two relations gives
\[
\partial_te_t+\mathcal G^{q^*_\varepsilon}e_t=\big(\mathcal G^{q^*_\varepsilon}f_t-c^{q^*_\varepsilon}_\varepsilon-\mathcal H^\varepsilon f_t\big)-\big(\mathcal G^{q^*_\varepsilon}f^\varepsilon_t-c^{q^*_\varepsilon}_\varepsilon-\mathcal H^\varepsilon f^\varepsilon_t\big)\le0\qquad\text{on }[0,T]\times\mathcal S.
\]
Moreover $\sup_t(\|e_t\|_V+\|\partial_te_t\|_V)<\infty$, by the hypotheses on $f$ and because $\|h_\varepsilon\|_V<\infty$ and $\partial_tf^\varepsilon_t=g^*_\varepsilon$.

By part (i), the chain of $q^*_\varepsilon$ is irreducible and positive recurrent on $\mathcal S$ with invariant law $\tilde\pi^*_\varepsilon$, and $\langle V,\tilde\pi^*_\varepsilon\rangle<\infty$ by assumption, so \Cref{lem:ec-pair} applies with $R=\mathcal S$ and gives $\langle\mathcal G^{q^*_\varepsilon}e_t,\tilde\pi^*_\varepsilon\rangle=0$. Differentiating under the sum, which is justified by the bound $|\partial_te_t|\le\sup_s\|\partial_se_s\|_V(1+V)$, gives
\[
\frac{d}{dt}\langle e_t,\tilde\pi^*_\varepsilon\rangle=\langle\partial_te_t+\mathcal G^{q^*_\varepsilon}e_t,\tilde\pi^*_\varepsilon\rangle\le0.
\]
The difference $e_T-e_0=(f_T-f_0)-(f^\varepsilon_T-f^\varepsilon_0)=T(\gamma'-g^*_\varepsilon)$ is constant in $x$, so integrating over $[0,T]$ gives $T(\gamma'-g^*_\varepsilon)=\langle e_T-e_0,\tilde\pi^*_\varepsilon\rangle\le0$. Together with $\gamma'\ge g^*_\varepsilon$ from part (ii), this gives $\gamma'=g^*_\varepsilon$, and hence $e_T=e_0$.

The function $e$ therefore satisfies the hypotheses of \Cref{lem:ec-periodic} with $q=q^*_\varepsilon$, $R=\mathcal S$ and $\pi=\tilde\pi^*_\varepsilon$, which gives a constant $k$ with $e\equiv k$ on $[0,T]\times\mathcal S$. Hence $f(t,x)=h_\varepsilon(x)+(t-T)g^*_\varepsilon+k$, its endpoint gap is $Tg^*_\varepsilon$, and its Gibbs kernel is that of $h_\varepsilon$, namely $q^*_\varepsilon$, whose stationary law is $\tilde\pi^*_\varepsilon$ by part (i).}
\end{proof}


\end{document}